\documentclass[a4paper,11pt]{article}
\usepackage{amsfonts}
\usepackage{latexsym}
\usepackage{amsmath}
\usepackage{amssymb}
\usepackage{amsthm}
\usepackage{amsmath}
\usepackage{mathtools}
\usepackage{verbatim}
\usepackage{hyperref}
\usepackage{array}
\usepackage{breqn}
\usepackage{bibspacing}
\usepackage[shortlabels]{enumitem}
\usepackage{tikz}

\newtheorem{thm}{Theorem}[section]
\newtheorem*{thm*}{Theorem}
\newtheorem{fact}[thm]{Fact}

\newtheorem{corollary}[thm]{Corollary}

\newtheorem{lem}[thm]{Lemma}
\newtheorem{lemma}[thm]{Lemma}
\newtheorem{prop}[thm]{Proposition}
\newtheorem*{prop*}{Proposition}
\newtheorem{proposition}[thm]{Proposition}

\newtheorem*{conj*}{Conjecture}
\newtheorem{definition}[thm]{Definition}
\newtheorem*{dfn*}{Definition}
\theoremstyle{definition}
\newtheorem{remark}[thm]{\textbf{Remark}}
\newtheorem*{remark*}{Remark}
\newtheorem*{fact*}{Fact}

\theoremstyle{proof}

\newcommand{\mrm}{\mathrm}

\newcommand{\vol}{\mrm{vol}}
\newcommand{\Vol}{\mrm{Vol}}
\newcommand{\Area}{\mrm{Per}}
\newcommand{\Ric}{\mrm{Ric}}
\newcommand{\hRic}{\mrm{Ric}} 
\newcommand{\tr}{\mrm{tr}}

\DeclareMathOperator{\argmax}{arg\,max}
\DeclareMathOperator*{\argmin}{arg\,min} 
\newcommand{\n}{\mathbf{n}}
\renewcommand{\c}{\mathbf{c}}

\newcommand{\G}{\mathbb{G}}
\newcommand{\HH}{\mathbb{H}}
\newcommand{\cyclic}{\mathcal{C}}

\newcommand{\Y}{\mathbf{Y}}
\newcommand{\T}{\mathbf{T}}
\renewcommand{\n}{\mathfrak{n}}
\renewcommand{\k}{\mathbf{k}}

\newcommand{\norm}[1]{\left\Vert#1\right\Vert}
\newcommand{\snorm}[1]{\Vert#1\Vert}
\newcommand{\abs}[1]{\left\vert#1\right\vert}
\DeclarePairedDelimiter{\sabs}{\lvert}{\rvert}
\newcommand{\set}[1]{\left\{#1\right\}}
\newcommand{\brac}[1]{\left(#1\right)}
\newcommand{\scalar}[1]{\left \langle #1 \right \rangle}
\newcommand{\sscalar}[1]{\langle #1 \rangle}

\newcommand{\R}{\mathbb{R}}

\renewcommand{\S}{\mathbb{S}}

\newcommand{\II}{\mathrm{II}}

\newcommand{\M}{\mathbb{M}}

\newcommand{\D}{\mathcal{D}}
\newcommand{\eps}{\epsilon}

\renewcommand{\H}{\mathcal{H}}
\newcommand{\cH}{\mathcal{H}}
\renewcommand{\div}{\text{\rm div}}

\newcommand{\tang}{\mathbf{t}}
\DeclareMathOperator{\interior}{int}

\numberwithin{equation}{section}

\newenvironment{displayme}
  {\[ \begin{tabular}{m{0.95\textwidth} } \em
  \leftskip=0.5cm plus 0.5fil \rightskip=0.5cm plus -0.5fil \parfillskip=0cm plus 0.5\textwidth}
{\end{tabular}\]}

\begin{document}

\renewcommand*{\thefootnote}{\fnsymbol{footnote}}

\author{Emanuel Milman\textsuperscript{$*$,$\dagger$} \and Botong Xu\textsuperscript{$*$,$\ddagger$}}

\footnotetext{$^*$ Department of Mathematics, Technion-Israel Institute of Technology, Haifa 32000, Israel.}
\footnotetext{$^\dagger$ Email: emilman@tx.technion.ac.il.}
\footnotetext{$^\ddagger$ Email: botongxu@campus.technion.ac.il.}

\begingroup    \renewcommand{\thefootnote}{}    
    \footnotetext{The research leading to these results is part of a project that has received funding from the European Research Council (ERC) under the European Union's Horizon 2020 research and innovation programme (grant agreement No 101001677).}
\endgroup

\title{Standard bubbles (and other M\"obius-flat partitions)\\ on model spaces are stable}
\date{}

\maketitle

\begin{abstract}
We verify that for all $n \geq 3$ and $2 \leq k \leq n+1$, the standard $k$-bubble clusters, conjectured to be minimizing total perimeter in $\R^n$, $\S^n$ and $\HH^n$, are stable -- an infinitesimal regular perturbation preserving volume to first order yields a non-negative second variation of area modulo the volume constraint. In fact, stability holds for all standard \emph{partitions}, in which several cells are allowed to have infinite volume. In the Gaussian setting, any partition  in $\G^n$ ($n\geq 2$) obeying Plateau's laws and whose interfaces are all \emph{flat}, is stable. 
 Our results apply to non-standard partitions as well -- starting with any (regular) flat Voronoi partition in $\S^n$ and applying M\"obius transformations and stereographic projections, the resulting partitions in $\R^n$, $\S^n$ and $\HH^n$ are stable. Our proof relies on a new conjugated Brascamp-Lieb inequality on partitions with conformally flat umbilical boundary, and the construction of a good conformally flattening boundary potential. 
\end{abstract}

\section{Introduction}

A weighted Riemannian manifold $(M^n,g,\mu)$ consists of a smooth complete $n$-dimensional Riemannian manifold $(M^n,g)$ endowed with a measure $\mu$ with $C^\infty$ smooth positive density $\exp(-W)$ with respect to the Riemannian volume measure $\vol_g$. Denote by $\cH^k$ the corresponding $k$-dimensional Hausdorff measure. Let $\mu^k=e^{-W} \cH^k$ and set the $\mu$-weighted volume to be $V_{\mu}:= \mu$. The $\mu$-weighted perimeter of a Borel subset $U \subset M$ of locally finite perimeter is defined as $A_\mu (U) : = \mu^{n-1} (\partial^* U)$, where $\partial^* U$ is the reduced boundary of $U$. 

The Euclidean, spherical and hyperbolic model spaces $(M^n,g)$ are denoted by $\R^n$, $\S^n$ and $\HH^n$, respectively. They are endowed with their standard Riemannian volume measure $\mu = \vol_g$, and we will simply write $V$ and $A$ for volume and perimeter. Another important model is the Gaussian space $\G^n$, obtained by endowing Euclidean space $\R^n$ with the standard Gaussian density $\mu = \gamma^n := (2 \pi)^{-n/2} \exp(-|x|^2/2) dx$. 
 
A $q$-partition $\Omega = (\Omega_1, \ldots, \Omega_q)$ on $(M,g,\mu)$ is a $q$-tuple of Borel subsets $\Omega_i \subset M$ having locally finite perimeter, such that $\{\Omega_i\}$ are pairwise disjoint and $V_\mu (M \setminus \cup_{i=1}^q \Omega_i) = 0$. Note that the sets $\Omega_i$, called cells, are not required to be connected. In this work, we allow using $q=\infty$, as long as the partition is locally finite, namely that every compact set $K \subset M$ intersects only a finite number of cells. A $k$-tuple ($k < \infty$) of pairwise disjoint cells $(\Omega_1,\ldots,\Omega_{k})$ so that $V_{\mu}(\Omega_i), A_{\mu}(\Omega_i) < \infty$ for all $i=1,\ldots,k$ is called a $k$-cluster, and its cells are called bubbles. Every $k$-cluster induces a partition by simply adding the ``exterior cell" $\Omega_{k+1} := M \setminus \cup_{i=1}^{k} \Omega_i$; by abuse of notation, we will call the resulting $(k+1)$-partition  $\Omega = (\Omega_1,\ldots,\Omega_{k+1})$ a $k$-cluster as well. 
However, when $V_\mu(M) = \infty$, there are partitions with multiple cells having infinite volume, and so these partitions are not induced by clusters. 
Define the $\mu$-weighted volume $V_\mu(\Omega)$ and total perimeter $A_\mu(\Omega)$ of a $q$-partition $\Omega$ to be
\begin{align*}
	V_{\mu} (\Omega) & := ( V_{\mu} (\Omega_1), \ldots, V_{\mu} (\Omega_q) )  \in \Delta_{V_\mu(M)}^{(q-1)} , \\
	A_{\mu} (\Omega) & := \frac{1}{2} \sum_{i=1}^q A_{\mu} (\Omega_i) = \sum_{1\leq i <j \leq  q} \mu^{n-1} (\Sigma_{ij}),
\end{align*}
where $\Sigma_{ij}:= \partial^* \Omega_i \cap \partial^* \Omega_j$ denotes the $(n-1)$-dimensional  interface between cells $\Omega_i$ and $\Omega_j$, and  $\Delta_{T}^{(q-1)}:= \{ v \in [0,\infty]^q  :  \sum_{i=1}^q v_i =T \}$. 
Note that when $\Omega$ is not induced by a cluster, $A_\mu(\Omega)$ will be infinite.  
	
The isoperimetric problem for $k$-clusters consists of identifying those clusters $\Omega$ of prescribed volume $V_\mu (\Omega) = v \in \interior \Delta^{(k)}_{V_\mu(M)}$ that minimize the total perimeter $A_{\mu}(\Omega)$. By modifying an isoperimetric minimizing cluster on a null set, we may and will assume that its cells $\Omega_i$ are open and satisfy $\overline{\partial^* \Omega_i} = \partial \Omega_i$. The classical isoperimetric problem corresponds to the single-bubble case $k=1$, and it is well-known that geodesic balls $\Omega_1$ of prescribed volume uniquely minimize perimeter on all model spaces $\R^n$, $\S^n$ and $\HH^n$ \cite{BuragoZalgallerBook}. It is also classical that halfplanes $\Omega_1$ of prescribed Gaussian volume uniquely minimize Gaussian perimeter on $\G^n$ \cite{SudakovTsirelson,Borell-GaussianIsoperimetry,CarlenKerceEqualityInGaussianIsop}; their flat boundaries can be thought of as degenerate flat spheres. Consequently, we will collectively refer to complete constant curvature hypersurfaces on $M^n \in \{\R^n, \S^n, \HH^n\}$ as ``generalized spheres" -- on $\R^n$ these are spheres and hyperplanes, and on $\HH^n$ these are geodesic spheres, horospheres, and equidistant hypersurfaces.  

The multi-bubble isoperimetric problem for $k$-clusters (when $k \geq 2$) already poses a much greater challenge. It was conjectured by J.~Sullivan (on $\R^n$, but his construction naturally extends to $\S^n$ and $\HH^n$) that when $k \leq n+1$, standard $k$-bubbles uniquely minimize total perimeter among all $k$-clusters of prescribed volume. A precise description of standard bubbles (and more generally, partitions) is deferred to Section \ref{sec:prelim}, but for now let us just remark that a standard partition of $\R^n$ is obtained as a stereographic projection of a Voronoi equipartition $\Omega^\S$ of $\S^n$ into $q = k+1$ cells (this is always possible whenever $k \leq n+1$ by taking the Voronoi cells of $q$-equidistant points in $\S^n$), resulting in (generalized) spherical interfaces between every pair of cells -- see Figures \ref{fig:stereographic}, \ref{fig:standard-bubbles} and \ref{fig:standard-partitions}. When the North pole in $\S^n$ lies in the interior of $\Omega^\S_q$, only the $q$-th cell in the stereographic projection onto $\R^n$ will have infinite volume, and so the resulting standard partition of $\R^n$ is called a standard $k$-cluster of bubbles.
The natural analogue of this in $\G^n$ when $k \leq n$ is to take the Voronoi partition of $q=k+1$ equidistant points in $\R^n$ (appropriately translated), resulting in flat interfaces between every pair of cells. 

\begin{figure}[htbp]
\centering
\includegraphics[scale=0.52]{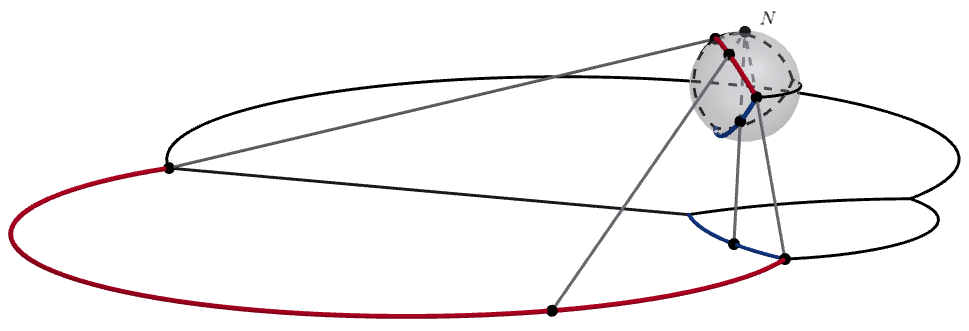}
\caption{ \label{fig:stereographic}
Stereographic projection of an equipartition of $\S^2$ into 4 Voronoi cells, yielding a standard triple-bubble in $\R^2$. 
}
\end{figure}

\begin{figure}
    \begin{center}
    \includegraphics[scale=0.35]{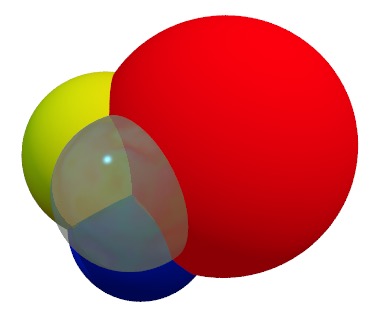}
    \includegraphics[scale=0.32]{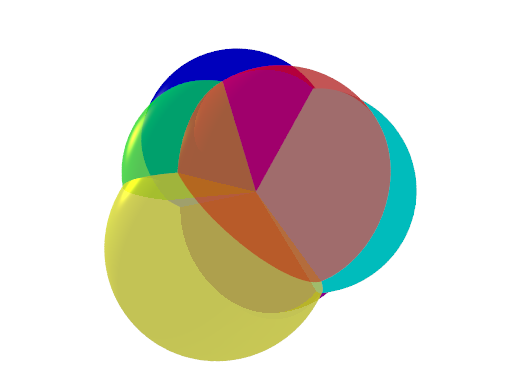}
     \end{center}
         \vspace{-10pt}
     \caption{
         \label{fig:standard-bubbles}
         Left: a standard quadruple-bubble in $\R^3$. Right: a $3$D cross-section of a sextuple-bubble in $\R^5$. Both are conjectured to be minimizing total perimeter under volume constraint.      }
\end{figure}

\begin{figure}
    \begin{center}
  \includegraphics[scale=0.28]{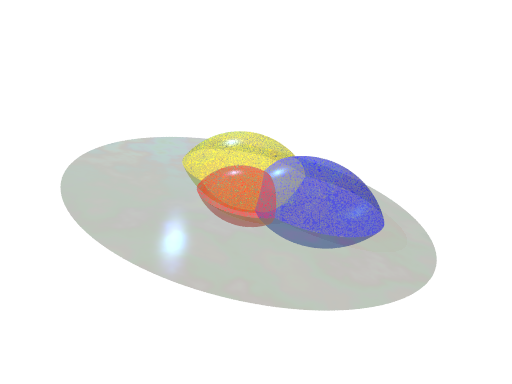}
     \includegraphics[scale=0.28]{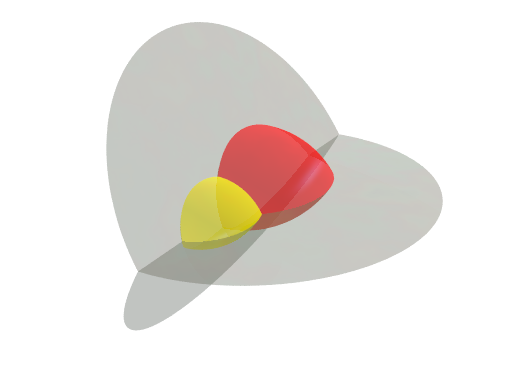}
     \includegraphics[scale=0.28]{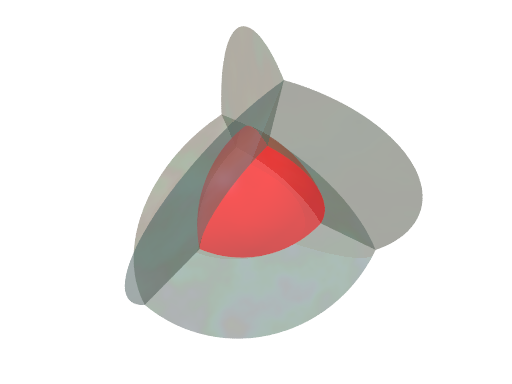}
     \end{center}
     \vspace{-10pt}
     \caption{
         \label{fig:standard-partitions}
         Standard $5$-partitions in $\R^3$ with $2$ (left), $3$ (middle) and $4$ (right) unbounded cells. All three are conjectured to be locally minimizing perimeter under volume constraint. 
     }
\end{figure}

When $n=2$, the double-bubble conjectures (case $k=2$) on these model spaces are well-understood and fully resolved. On $\R^2$, $\S^2$ and $\HH^2$ this was established in \cite{SMALL93}, \cite{Masters-DoubleBubbleInS2} and \cite{CottonFreeman-DoubleBubbleInSandH}, respectively (see below for the case of $\G^2$). The triple-bubble conjectures (case $k=3$) on $\R^2$ and $\S^2$ were established by Wichiramala \cite{Wichiramala-TripleBubbleInR2} and Lawlor \cite{Lawlor-TripleBubbleInR2AndS2}, respectively, but to the best of our knowledge remains open on $\HH^2$. 

In a landmark work, the double-bubble conjecture ($k=2$) on $\R^3$ was confirmed by Hutchings--Morgan--Ritor\'e--Ros \cite{DoubleBubbleInR3-Announcement,DoubleBubbleInR3} following prior contributions in \cite{HHS95,HassSchlafly-EqualDoubleBubbles,Hutchings-StructureOfDoubleBubbles}, and later extended to all $\R^n$ \cite{SMALL03,Reichardt-DoubleBubbleInRn,Lawlor-DoubleBubbleInRn}. Recently, in \cite{EMilmanNeeman-TripleAndQuadruple,EMilmanNeeman-QuintupleBubble}, the double-, triple-, quadruple- and quituple-bubble conjectures (cases $k=2,3,4,5$) were established by Milman and Neeman on $M^n$ for $M^n \in \{\R^n,\S^n\}$ and $n \geq k$ (without uniqueness on $\R^n$ in the quintuple case). Before that, Milman and Neeman fully resolved in \cite{EMilmanNeeman-GaussianMultiBubble} the multi-bubble isoperimetric conjecture on $\G^n$ for all $1 \leq k \leq n$. The remaining cases on $\R^n$, $\S^n$ and $\HH^n$ are still open.

When $V_{\mu}(M^n)=\infty$, such as for $M^n \in \{\R^n,\HH^n \}$, one can also consider the \emph{local} isoperimetric problem for $q$-partitions $\Omega$ with prescribed volume $V_\mu(\Omega) = v$ when at least two of the prescribed volumes are infinite ($\exists i \neq j$ with $v_i=v_j = \infty$). In that case, $A_\mu(\Omega) = \infty$ and so the global minimization problem does not make sense. Instead, one considers \emph{locally} minimizing partitions, which minimize the total relative perimeter in any bounded open $K \subset M^n$ among all competing $q$-partitions $\Omega'$ with $V_\mu(\Omega') = V_\mu(\Omega)$ so that $\Omega'_i \Delta \Omega_i \Subset K$ for all $i$. This line of investigation was pioneered by Alama, Bronsard and Vriend \cite{ABV-LensClusters} on $\R^2$ for the case $v = (1,\infty,\infty)$, and systematically studied by Novaga, Paolini and Tortorelli in \cite{NPT-LocallyIsoperimetricPartitions}. Using a limit argument and a closure theorem, it was shown in \cite[Section 3]{NPT-LocallyIsoperimetricPartitions} that the results 
from the preceding paragraphs
imply that standard $q$-partitions of $\R^n$ for $q=2,3,4$ ($n \geq 2$), $q=5$ ($n \geq 4$) and $q=6$ ($n \geq 5$) are locally minimizing (but does not exclude the existence of other non-standard partitions which are also locally minimizing, except in the planar case of $\R^2$ for which uniqueness was established in \cite[Section 4]{NPT-LocallyIsoperimetricPartitions}).

Some additional isoperimetric results for clusters and partitions are obtained in 
\cite{PaoliniTamagnini-PlanarQuadraupleBubbleEqualAreas,PaoliniTortorelli-PlanarQuadrupleEqualAreas,NPST-ClustersViaConcentrationCompactness, NPST-ClustersWithInfinitelyManyCells,DeRosaTione-ConvexStationaryBubbles,BronsardNovak-DifferentTensions,PeriodicDoubleTilings}.

\subsection{Informal presentation of main results} \label{subsec:intro-informal}

Showing the global (or local) minimality of the conjectured clusters (or partitions) is widely acknowledged to be extremely difficult, and so confirmation of their local minimality, even in an infinitesimal sense, is already highly challenging, and would  provide valuable evidence for the validity of the conjectures. 
A standard way in the calculus of variations to probe the local (infinitesimal) minimality of a given configuration is to test the non-negativity of the second variation (modulo the volume constraint, which should be preserved to first order), a property called ``stability". 

\medskip

Modulo technicalities (see Remark \ref{rem:intro-stability}), our first main result in this work is a confirmation that:
\begin{displayme}
For all $1 \leq k \leq n+1$ and $n \geq 3$, standard $k$-bubbles in $\R^n$, $\S^n$ and $\HH^n$ are stable.
\end{displayme}
Moreover, this actually holds for all standard partitions, such as the ones depicted in Figures \ref{fig:standard-bubbles} and \ref{fig:standard-partitions}:
\begin{displayme}
For all $2 \leq q \leq n+2$ and $n \geq 3$, standard $q$-partitions in $\R^n$, $\S^n$ and $\HH^n$ are stable.
\end{displayme}
A precise formulation is deferred to Subsection \ref{subsec:intro-results}. An interesting phenomenon occurs on $\HH^n$: there exist standard partitions (not induced by a cluster) for which stability holds without requiring that the volume is preserved to first order -- see Theorem \ref{thm:intro-V} and the subsequent comments.

\medskip

\begin{figure}
    \begin{center}
    \includegraphics[scale=0.3]{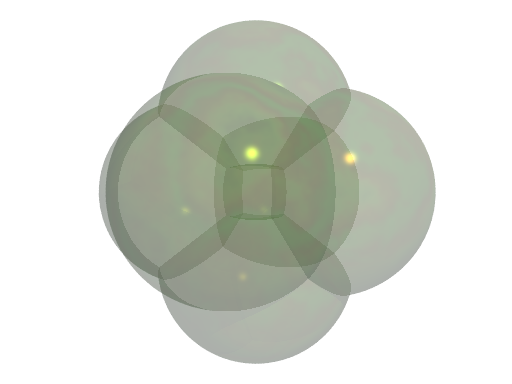}
     \includegraphics[scale=0.3]{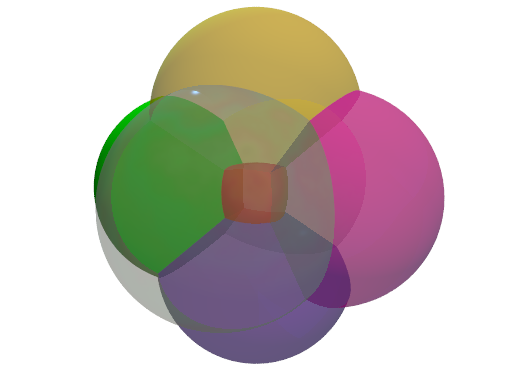}
      \includegraphics[scale=0.3]{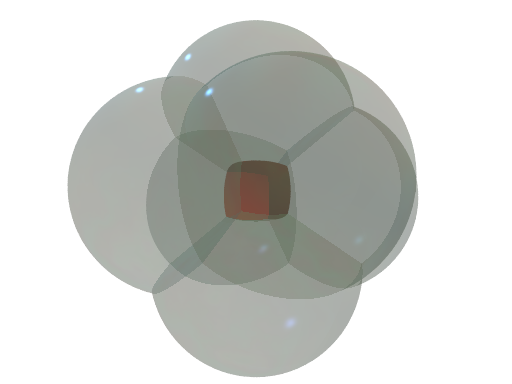}
      \includegraphics[scale=0.3]{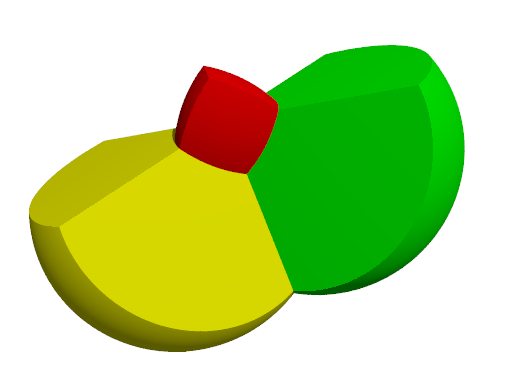}
     \end{center}
     \vspace{-10pt}
     \caption{
         \label{fig:nonstandard-partition1}
         A non-standard 7-bubble in $\R^3$ with a cubical inner cell, often created by soap-bubble magicians, is stable. 
        }
\end{figure}

\begin{figure}
    \begin{center}
     \includegraphics[scale=0.3]{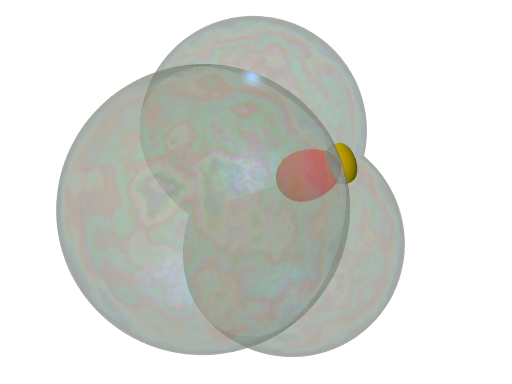}
     \includegraphics[scale=0.3]{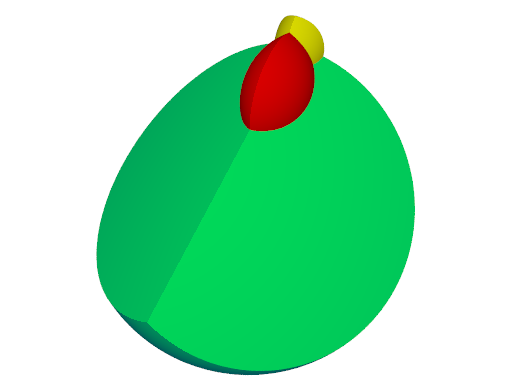}
     \includegraphics[scale=0.3]{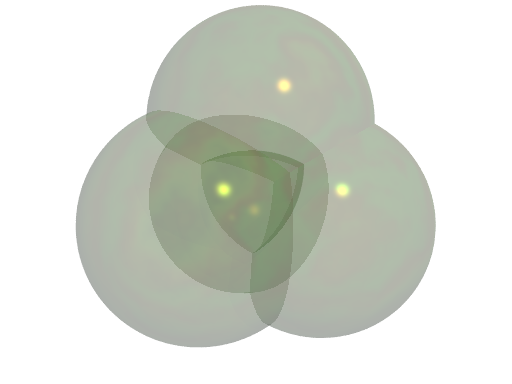}
     \includegraphics[scale=0.3]{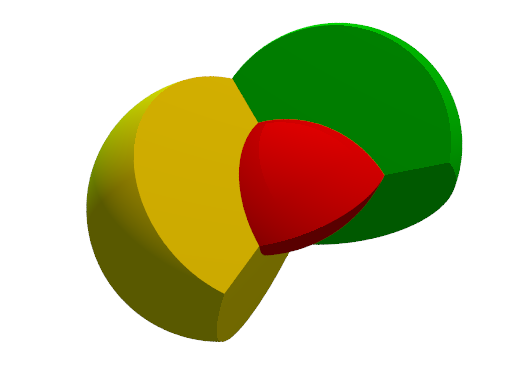}
     \end{center}
     \caption{
         \label{fig:nonstandard-partition2}
         Non-standard 5-bubbles in $\R^3$ with two smaller cells highlighted (top) and with equal-volume cells enclosing a tetrahedral inner cell (bottom), are stable. 
        }
\end{figure}

Our results are not restricted to standard partitions. We actually show that:
\begin{displayme}
For all $2 \leq q \leq \infty$ and $n \geq 3$, regular, M\"obius-flat, spherical Voronoi, $q$-partitions
 in $\R^n$, $\S^n$ and $\HH^n$ are stable.
\end{displayme}
We defer the precise definitions to Section \ref{sec:prelim}, but for now only mention that:
\begin{itemize}
\item A \emph{regular} partition obeys Plateau's laws -- three cells meet like the cone over the vertices of a triangle, and four cells meet like the cone over the edges of a tetrahedron. These are known necessary conditions for area-minimizing soap-bubbles, as confirmed in $\R^3$ by Taylor \cite{Taylor-SoapBubbleRegularityInR3} and in $\R^n$ (and hence in any weighted Riemannian manifold) by Colombo--Edelen--Spolaor \cite{CES-RegularityOfMinimalSurfacesNearCones}. 
\item All (non-empty) interfaces $\Sigma_{ij}$ of a \emph{spherical Voronoi partition} are subsets of (generalized) spheres, whose curvatures $\k_{ij}$ and (quasi-)centers $\c_{ij}$ enjoy a separable form $\k_{ij} = \k_i - \k_j$ and $\c_{ij} = \c_i - \c_j$ for some $\{\k_i\}$ and $\{\c_i\}$. In particular, a regular spherical Voronoi partition is automatically \emph{stationary}, meaning that its interfaces meet in threes at $120^{\circ}$-degree angles, and its curvatures cyclically sum to $0$ at these triple points. 
\item A partition is called flat if all of its interfaces $\Sigma_{ij}$ are flat (i.e.~totally geodesic). A partition $\Omega$ in $\S^n$ is called \emph{M\"obius-flat} if  there exists a M\"obius automorphism $T : \S^n \rightarrow \S^n$ so that $T\Omega$ is flat. 
The same terminology is used on $\R^n$ and $\HH^n$ if this holds after a stereographic projection to $\S^n$ (or a hemisphere). 
\end{itemize}
By definition, standard bubbles and partitions satisfy all the above properties, but there are natural additional examples. 
One just needs to start with a \emph{non-standard} flat regular spherical Voronoi partition in $\S^n$ ($n \geq 3$), and apply a M\"obius transformation to obtain a non-flat partition in $\S^n$, followed by a stereographic projection onto $\R^n$ or $\HH^n$. For example, an initial non-standard (regular) flat spherical Voronoi $(2n+2)$-partition in $\S^n$ is the following one,
generating the cone over the facets of the hypercube in $\R^{n+1}$, namely
\[
 \Omega_i = \{ p \in \S^n : \argmax_{j=1 \ldots, 2n+2} \scalar{p,v_j} = \{i\} \} ,
\]
for $v_k = e_k$ and $v_{k+n+1} = -e_k$, $k=1,\ldots,n+1$. See Figures \ref{fig:nonstandard-partition1}, \ref{fig:nonstandard-partition2}, \ref{fig:nonstandard-partition3} and \ref{fig:nonstandard-partition4} for a depiction of various non-standard bubbles and partitions which satisfy our assumptions and are thus proved to be stable.

\begin{figure}[htbp]
\vspace{-50pt}
    \centering
        \includegraphics[scale=0.35]{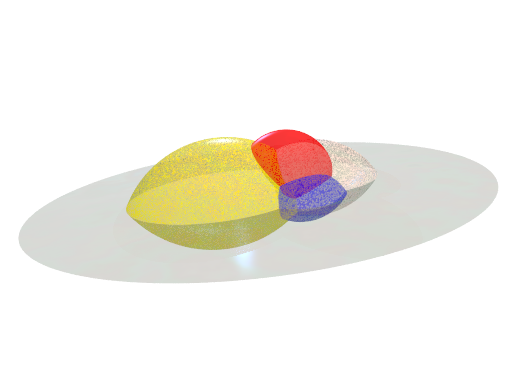}
          \includegraphics[scale=0.35]{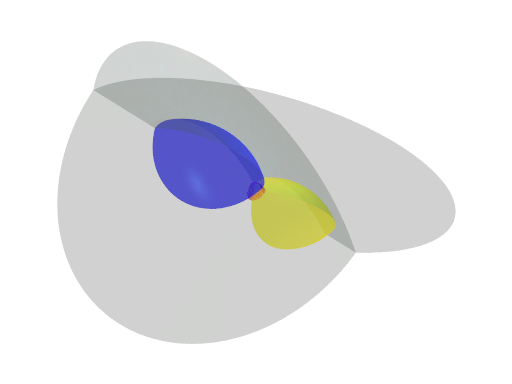}
     \includegraphics[scale=0.4]{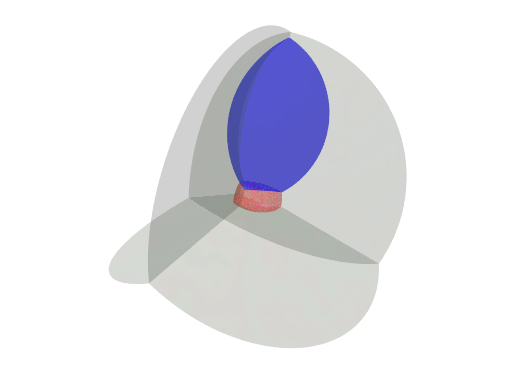}
          \caption{
         \label{fig:nonstandard-partition3} 	Non-standard 6-partitions in $\R^3$ with 2 (left), 3 (right) and 4 (bottom) unbounded cells, are all stable. 
        }
      \includegraphics[scale=0.35]{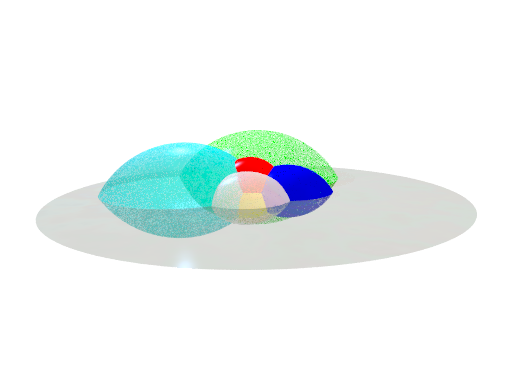}
     \includegraphics[scale=0.35]{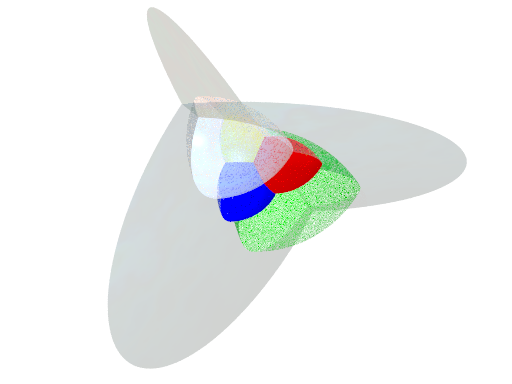}
     \includegraphics[scale=0.4]{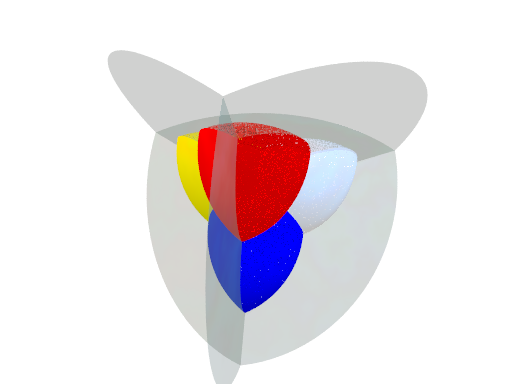}
     \caption{
         \label{fig:nonstandard-partition4} 	Non-standard 8-partitions in $\R^3$ with 2 (left), 3 (right) and 4 (bottom) unbounded cells, are all stable. 
        }
\end{figure}

\medskip
These results provide a partial answer to a question of Kusner \cite[Problem 3]{OpenProblemsInSoapBubbles96}, who asked whether clusters in $\R^3$ with spherical interfaces which meet according to Plateau's laws are necessarily stable. Moreover, our formulation provides a feasible extension of Kusner's question to all model spaces in any dimension $n \geq 3$. A much stronger conjecture of Morgan \cite[Conjecture 2.4]{Morgan-StrictCalibrations} predicts that any stationary regular partition in $\R^n$, all of whose tangent cones are strictly area-minimizing, is locally area-minimizing under a volume constraint (see \cite[Theorem 2.1]{Morgan-StrictCalibrations} for a proof of this for $C^1$ deformations in a small ball where 3 cells meet). 

\medskip

On $\G^n$, we provide a complete answer to the analogous question to Kusner's:
\begin{displayme}
For all $2 \leq q \leq \infty$ and $n \geq 2$, stationary, regular, flat $q$-partitions in $\G^n$ are stable.
\end{displayme}
For example, any collection of line segments (or half-lines) meeting at $120^{\circ}$-degree angles in the plane delineates a stable partition in $\G^2$ -- see Figure \ref{fig:Gaussian}.

\begin{figure}[htbp]
	\centering
	\begin{tikzpicture}[scale=0.53]				\coordinate (UL) at (-2.8,2.8);
		\coordinate (LL) at (-2.8,-3.8);
		\coordinate (UR) at (3.8,2.8);
		\coordinate (LR) at (3.8,-3.8);
				\coordinate (A) at (0,0);
		\coordinate (M) at (-0.573,0.992);
		\coordinate (N) at (-1.101,0.992);
		\coordinate (P) at (-1.620,0.093);
		\coordinate (I) at (-1.266,-0.520);
		\coordinate (D) at (-0.3,-0.520);
		\draw[thick] (A)--(M)--(N)--(P)--(I)--(D)--(A);
				\coordinate (J) at (-1.447,-0.833);
		\coordinate (K) at (-0.437,-2.582);
		\coordinate (L) at (0.505,-2.582);
		\coordinate (E) at (0.698,-2.247);
		\draw[thick] (I)--(J)--(K)--(L)--(E)--(D)--(I);
				\coordinate (F) at (1.859,-2.247);
		\coordinate (G) at (2.730,-0.739);
		\coordinate (H) at (2.304,0);
		\draw[thick] (E)--(F)--(G)--(H)--(A)--(D)--(E);
				\coordinate (R) at (0.471,2.8);
		\coordinate (Q) at (-2.145,2.8);
		\draw[thick] (R)--(M)--(N)--(Q);
		\coordinate (S) at (-2.8,0.093);
		\draw[thick] (Q)--(N)--(P)--(S);
		\coordinate (U) at (-2.8,-0.833);
		\draw[thick] (S)--(P)--(I)--(J)--(U);
		\coordinate (V) at (-1.140,-3.8);
		\draw[thick] (U)--(J)--(K)--(V);
		\coordinate (W) at (1.208,-3.8);
		\draw[thick] (V)--(K)--(L)--(W);
		\coordinate (Z) at (2.755,-3.8);
		\draw[thick] (W)--(L)--(E)--(F)--(Z);
		\coordinate (A1) at (3.8,-0.739);
		\draw[thick] (Z)--(F)--(G)--(A1);
		\coordinate (T) at (3.8,2.592);
		\draw[thick] (A1)--(G)--(H)--(T);
		\draw[thick] (T)--(H)--(A)--(M)--(R);
	\end{tikzpicture}
	\qquad
	\begin{tikzpicture}[scale=0.53]				\coordinate (UL) at (-2.8,2.8);
		\coordinate (LL) at (-2.8,-3.8);
		\coordinate (UR) at (3.8,2.8);
		\coordinate (LR) at (3.8,-3.8);
				\coordinate (A) at (0,0);
		\coordinate (M) at (-0.573,0.992);
		\coordinate (N) at (-1.101,0.992);
		\coordinate (P) at (-1.620,0.093);
		\coordinate (I) at (-1.266,-0.520);
		\coordinate (D) at (-0.3,-0.520);
		\draw[thick] (A)--(M)--(N)--(P)--(I)--(D)--(A);
				\coordinate (J) at (-1.447,-0.833);
		\coordinate (K) at (-0.437,-2.582);
		\coordinate (L) at (0.505,-2.582);
		\coordinate (E) at (0.698,-2.247);
		\draw[thick] (I)--(J)--(K)--(L)--(E)--(D)--(I);
				\coordinate (F) at (1.859,-2.247);
		\coordinate (G) at (2.730,-0.739);
		\coordinate (H) at (2.304,0);
		\draw[thick] (E)--(F)--(G)--(H)--(A)--(D)--(E);
				\coordinate (R) at (0.471,2.8);
		\coordinate (Q) at (-2.145,2.8);
		\draw[thick] (R)--(M)--(N)--(Q);
		\coordinate (S) at (-2.8,0.093);
		\draw[thick] (Q)--(N)--(P)--(S);
		\coordinate (U) at (-2.8,-0.833);
		\draw[thick] (S)--(P)--(I)--(J)--(U);
		\coordinate (V) at (-1.140,-3.8);
		\draw[thick] (U)--(J)--(K)--(V);
		\coordinate (W) at (1.208,-3.8);
		\draw[thick] (V)--(K)--(L)--(W);
		\coordinate (Z) at (2.755,-3.8);
		\draw[thick] (W)--(L)--(E)--(F)--(Z);
		\coordinate (A1) at (3.8,-0.739);
		\draw[thick] (Z)--(F)--(G)--(A1);
		\coordinate (T) at (3.8,2.592);
		\draw[thick] (A1)--(G)--(H)--(T);
		\draw[thick] (T)--(H)--(A)--(M)--(R);
																														\filldraw[red] (A)--(M)--(N)--(P)--(I)--(D)--(A);
		\filldraw[yellow] (I)--(J)--(K)--(L)--(E)--(D)--(I);
		\filldraw[blue] (E)--(F)--(G)--(H)--(A)--(D)--(E);
		\filldraw [green] (R)--(M)--(N)--(Q);
		\filldraw[pink] (Q)--(N)--(P)--(S)--(UL)--(Q);
		\filldraw[gray] (S)--(P)--(I)--(J)--(U)--(S);
		\filldraw[brown] (U)--(J)--(K)--(V)--(LL)--(U);
		\filldraw[magenta] (V)--(K)--(L)--(W)--(V);
		\filldraw[cyan] (W)--(L)--(E)--(F)--(Z);
		\filldraw[orange] (Z)--(F)--(G)--(A1)--(LR)--(Z);
		\filldraw[purple] (A1)--(G)--(H)--(T)--(A1);
		\filldraw[teal] (T)--(H)--(A)--(M)--(R)--(UR);
	\end{tikzpicture}

          \hspace{20pt} \includegraphics[scale=0.4]{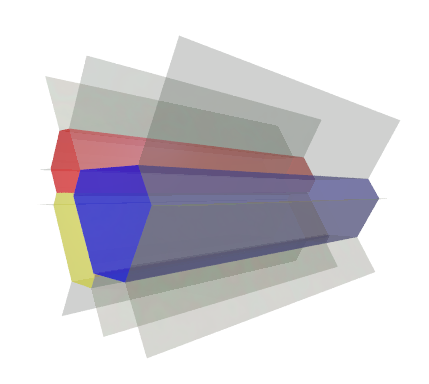}
     \hspace{-30pt} \includegraphics[scale=0.4]{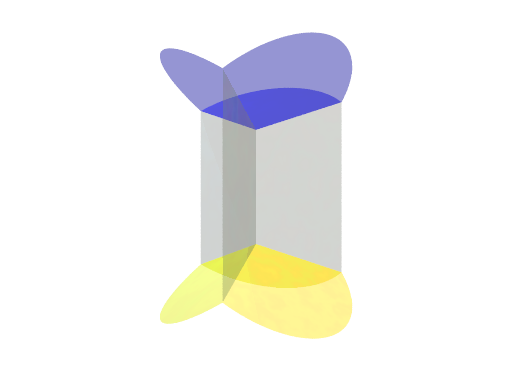}
     \caption{
         \label{fig:Gaussian}
         Stationary, regular, flat $12$-partitions in $\G^2$ (top) and $\G^3$ (bottom left) and $5$-partition in $\G^3$ (bottom right), are all stable. 
        }
\end{figure}
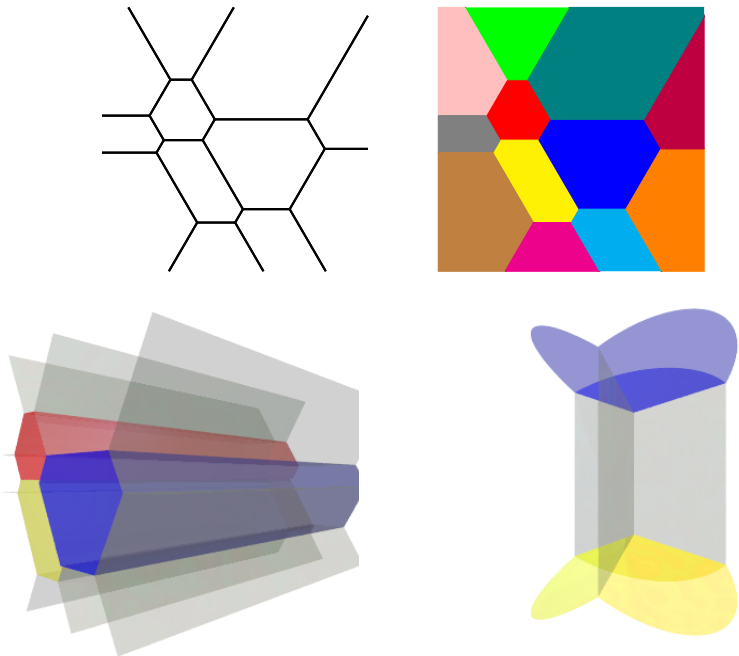

\smallskip

We now turn to a more precise description of our results and method of proof. To the best of our knowledge, there were no prior methods for establishing the stability of a partition on spaces with positive (weighted) Ricci curvature such as $\S^n$ and $\G^n$, or non-flat partitions on spaces with non-positive Ricci curvature such as $\R^n$ and $\HH^n$, besides restricting to a small ball as in \cite{Morgan-StrictCalibrations}, or outright resolving the (local) isoperimetric problem (which of course implies stability). 
We will require  two new ingredients:
\begin{enumerate}
\item A conjugated Brascamp-Lieb inequality on partitions with conformally flat umbilical boundary.
\item Construction of a good conformally flattening boundary potential on M\"obius-flat spherical Voronoi partitions. 
\end{enumerate}
Before describing those, let us recall the analytic description of stability. 

\subsection{Stability and the Jacobi operator}

Let $\Omega$ be a regular partition in $(M,g,\mu = \exp(-W) d\vol_g)$. Given a vector-field $X \in C_c^\infty(M)$ supported in the interior of a compact $K \subset M$, the flow along $X$ for time $t$ is denoted by $F_t$, the perturbed partition is denoted by $F_t(\Omega)$, and the $k$-th variation in volume $V_\mu(F_t(\Omega),K)$ and perimeter $A_\mu(F_t(\Omega),K)$ (relative to $K$) are denoted by $\delta^k_X \Vol$ and $\delta^k_X \Area$, respectively. Clearly, $\delta^1_X \Vol$ only depends on $f = (f_{ij})$, where $f_{ij} = X^{\n_{ij}}$ is the normal component of $X$ on $\Sigma_{ij}$, which is co-oriented by the normal $\n_{ij}$ pointing from $\Omega_i$ to $\Omega_j$, and
\[
\delta^1_X \Vol = \delta^1_f \Vol := \brac{\int_{\partial \Omega_i} f_{ij} d\mu^{n-1}}_i = \brac{\sum_{j \neq i} \int_{\Sigma_{ij}} f_{ij} d\mu^{n-1}}_i .
\]
Note that $f_{ji} = -f_{ij}$ is ``oriented". 

A partition is called stationary if $\delta^1_X \Area = \scalar{\lambda,\delta^1_X \Vol}$ for some vector $\lambda \in \R^q$ of Lagrange multipliers for all $X$ as above.
On a stationary regular partition, whenever three interfaces $\Sigma_{ij}$, $\Sigma_{jk}$ and $\Sigma_{ki}$ meet, they do so at $120^{\circ}$-degree angles. Denoting their common $(n-2)$-dimensional boundary by $\Sigma_{ijk}$, this means that $\n_{ij} + \n_{jk} + \n_{ki} = 0$ and hence $f_{ij} + f_{jk} + f_{ki} = 0$ on $\Sigma_{ijk}$ (referred to as ``Dirichlet-Kirchoff conditions"). 

A stationary regular partition is called stable if
\[
\delta^1_X \Vol = 0 \;\; \Rightarrow \;\; Q(X) := \delta^2_X \Area - \scalar{\lambda, \delta^2_X \Vol} \geq 0 . 
\]
It turns out that like $\delta^1_X \Vol$, $Q(X)$ only depends on $f = (f_{ij}) = (X^{\n_{ij}})$. We write $Q(X) = Q^0(f)$, and refer to $f = (f_{ij})$ as a physical scalar-field. More generally, a collection $f = (f_{ij})$ of (mildly regular) oriented functions $f_{ij}$ on $\Sigma_{ij}$ satisfying the Dirichlet-Kirchoff conditions is simply called a scalar-field. 
 We set: \[
\Sigma^1 := \cup_{i<j} \Sigma_{ij} ~,~ \Sigma^2 := \cup_{i<j<k} \Sigma_{ijk} . 
\]

It is known that a (locally) isoperimetric minimizing partition is necessarily regular, stationary and stable. Let $\Omega$ be a stationary regular partition with locally bounded curvature in $(M,g,\mu)$, meaning that the second fundamental form $\II^{ij}$ of $\Sigma_{ij}$ is bounded for all $i<j$ on every compact subset $K \subset M$. 
In that case, it was shown in \cite{EMilmanNeeman-GaussianMultiBubble,EMilmanNeeman-TripleAndQuadruple} that for any physical scalar-field $f$:
\begin{equation} \label{eq:intro-Q0}
Q^0(f)  = \sum_{i<j} \brac{ -\int_{\Sigma_{ij}} f_{ij} L_{Jac} f_{ij} d\mu^{n-1} + \int_{\partial \Sigma_{ij}} (\nabla_{\n_{\partial ij}} f_{ij} - \bar \II^{\partial ij} f_{ij}) f_{ij} d\mu^{n-2}} ,
\end{equation}
where $\partial \Sigma_{ij}$ denotes the $(n-2)$-dimensional part of the boundary of $\Sigma_{ij}$, $\n_{\partial ij}$ denotes the outer unit-normal to $\partial \Sigma_{ij}$ in $\overline{\Sigma_{ij}}$, and $\bar \II^{\partial ij}$ is defined on $\Sigma_{ijk}$ as
\begin{equation} \label{eq:intro-barII}
 \bar \II^{\partial ij} =  \bar \II^{\partial ij}_{ijk} := \frac{\II^{ik}(\n_{\partial ik},\n_{\partial ik}) + \II^{jk}(\n_{\partial jk},\n_{\partial jk})}{\sqrt{3}}.
 \end{equation}
The Jacobi operator $L_{Jac}$ is defined as
\begin{equation} \label{eq:intro-LJac}
L_{Jac} f_{ij} := \Delta_{\Sigma^1,\mu} f_{ij} + (\Ric_{M,\mu}(\n_{ij},\n_{ij}) + \snorm{\II^{ij}}^2) f_{ij} ,
\end{equation}
where
\[
\Delta_{\Sigma^1,\mu} f_{ij} := \Delta_{\Sigma^1} f_{ij} - \scalar{\nabla_{\Sigma^1} f_{ij} , \nabla_{\Sigma^1} W} 
\]
is the $\mu$-weighted Laplacian on $(\Sigma^1,\mu^{n-1})$, 
\[
\Ric_{M,\mu} := \Ric_M + \nabla^2  W 
\]
is the $\mu$-weighted Ricci curvature on $(M,g,\mu)$, and $\norm{\cdot}$ denotes the Hilbert-Schmidt norm.

\subsection{Conjugated Brascamp-Lieb inequality on partitions with conformally flat umbilical boundary}

The classical Brascamp-Lieb inequality \cite{BrascampLiebPLandLambda1} was originally established on weighted Euclidean space, but has since been extended to various more general settings. One relevant formulation established by Kolesnikov--Milman \cite{KolesnikovEMilmanReillyPart1} states that if $(\Sigma^{n-1},g,\mu = \exp(-W) d\vol_g)$ is a compact weighted Riemannian manifold with boundary $\partial \Sigma$, so that $\II_{\partial \Sigma} \geq 0$ is locally convex, then for any $f \in C^1(\Sigma)$
\begin{equation} \label{eq:intro-simple-BL}
\int_\Sigma f d\mu = 0 \;\; \Rightarrow \;\; \frac{N}{N-1} \int_\Sigma f^2 d\mu \leq \int_\Sigma \Ric_{\Sigma,\mu,N}^{-1}(\nabla f, \nabla f) d\mu ,
\end{equation}
assuming $N \in [n-1,\infty] \cup (-\infty,0)$, $N \neq1$, and $\Ric_{\Sigma,\mu,N} > 0$, where the $N$-dimensional $\mu$-weighted Ricci tensor on $\Sigma$ is defined as:
\[
\Ric_{\Sigma,\mu,N} := \Ric_{\Sigma} + \nabla^2_{\Sigma} W - \frac{1}{N-n+1} \nabla_\Sigma W \otimes \nabla_\Sigma W . 
\]
When $n=2$, we simply set $\Ric_{\Sigma} \equiv 0$. 

An obvious problem arises when attempting to apply this to the interfaces of a stationary regular partition: since $\n_{ij} + \n_{jk} + \n_{ki} = 0$ on $\Sigma_{ijk}$, and hence
\[
\II_{\Sigma_{ijk} \subset \partial \Sigma_{ij}} + \II_{\Sigma_{ijk} \subset \partial \Sigma_{jk}} + \II_{\Sigma_{ijk} \subset \partial \Sigma_{ki}} = 0,
\]
it is impossible for $\Sigma_{ijk}$ to be locally convex in $\overline{\Sigma_{ij}}, \overline{\Sigma_{jk}}, \overline{\Sigma_{ki}}$ simultaneously, unless it is flat in all three interfaces. While this will not be a problem on $\G^n$ when we treat flat stationary partitions, it is definitely a serious obstacle for partitions with spherical (and more general) interfaces. 

A useful idea is that a conformal change of metric $\tilde g = \frac{1}{V^2} g$ can make an umbilical $\Sigma_{ijk}$ flat. However, in our context, the resulting tensor $\Ric_{\Sigma_{ij},\tilde g,\mu,N}$ will not yield a useful expression which we can control. Instead, we turn to another tensor from the General Relativity literature, which appears in the context of \emph{static} space-times \cite{LiXia-GenReillyFormulaSubStatic}. Given a positive smooth function $V>0$ on $(\Sigma,g,\mu)$, define, following \cite{LiXia-GenReillyFormulaSubStatic,HuangZhu-ConjugatedBrascampLieb},
\begin{equation} \label{eq:intro-RicV}
\hRic^V_{\Sigma,\mu,N} := \Ric_{\Sigma,\mu,N} - \frac{\Delta_{\Sigma,\mu} V}{V} g + \frac{\nabla^2_{\Sigma} V}{V} .
\end{equation}
In the unweighted case $\mu = \vol_g$, we simply write $\hRic^V_{\Sigma}$. 
An unweighted space-time $(\Sigma^{n-1},g)$ is called static if $\hRic^V_{\Sigma} \equiv 0$ for some $V>0$, which is called a static potential. The zero set of a static potential is always a totally geodesic hypersurface, explaining the ``static" nomenclature. It is known \cite{Brendle-ScalarCurvatureSurvey} that $\hRic^V_{\Sigma}$ arises as the formal $L^2$-adjoint of the linearization of the scalar curvature. 
We remark that a closely related tensor appears in the ``intertwining" approach of Bonnefont--Joulin \cite{BonnefontJoulin-ConjugatedBLWithBoundary} for Brascamp-Lieb inequalities, but it does not exactly coincide with $\hRic^V_{\Sigma,\mu,\infty}$ when $n \geq 3$. 
Given a smooth positive potential $V > 0$ on a Riemannian manifold $(\Sigma,g)$ with boundary $\partial \Sigma$ co-oriented by the unit-normal $\n_{\partial}$, define:
\[
\II_{\partial \Sigma}^V := \II_{\partial \Sigma} - \nabla_{\n_{\partial}} \log V \; g_{\partial \Sigma} . 
\]
Note that $\frac{1}{V} \II_{\partial \Sigma}^V$ is the second fundamental form of $\partial \Sigma$ relative to $(\Sigma,\frac{1}{V^2} g)$. However, we emphasize again that $\hRic^V_{\Sigma,\mu,N}$ is not directly related to $\Ric_{\Sigma,\frac{1}{V^2} g, \mu , N}$. When $(\Sigma^{n-1},g)$ is a compact Riemannian manifold with boundary satisfying $\II_{\partial \Sigma}^V \geq 0$ and $\hRic^V_{\Sigma,\mu,N} > 0$ for some $N \in [n-1,\infty] \cup (-\infty,0)$, $N \neq1$, Huang--Zhu \cite{HuangZhu-ConjugatedBrascampLieb} (see also \cite{HuangMaZhu-ReillyFormula}) established that, for any $h \in C^1(\Sigma)$,
\[ \int_\Sigma h V d\mu = 0 \;\; \Rightarrow \;\; \frac{N}{N-1} \int_\Sigma h^2 V d\mu \leq \int_\Sigma (\hRic^V_{\Sigma,\mu,N})^{-1}(\nabla h, \nabla h) V d\mu ,
\] generalizing (\ref{eq:intro-simple-BL}) from the case that $V \equiv 1$. Equivalently, setting $f = h V$, we have for any $f \in C^1(\Sigma)$,
\begin{equation} \label{eq:intro-HZ-BL}
\int_\Sigma f d\mu = 0 \;\; \Rightarrow \;\; \frac{N}{N-1} \int_\Sigma \frac{f^2}{V} d\mu \leq \int_\Sigma  (V \hRic^V_{\Sigma,\mu,N})^{-1}\brac{\nabla f - \frac{\nabla V}{V} f}  d\mu ,
\end{equation}
where we use $T(f)$ to denote $T(f,f)$ for a quadratic form $T$.
We will refer to this as a ``conjugated Brascamp-Lieb inequality", as $V$ serves some type of conjugation role. 

\medskip

Our first new ingredient extends this to partitions with umbilical boundary which is conformally flat. \begin{definition}[Conformally flat boundary]
A regular partition $\Omega$ of $(M^n,g,\mu)$ is said to have conformally flat boundary if there exists a smooth positive function $V>0$ on $(M^n,g)$, called the conformally flattening boundary potential, so that for all non-empty $\Sigma_{ijk}$, 
\[
\II_{\Sigma_{ijk} \subset \partial \Sigma_{ij}}^V  = \II_{\Sigma_{ijk} \subset \partial \Sigma_{ij}} - \nabla_{\n_{\partial ij}} \log V \, g_{\Sigma_{ijk}} \equiv 0.
\] 
\end{definition}
We defer the somewhat more technical definition of umbilical boundary to Section \ref{sec:Bochner}. For now, let us just emphasize that a stationary regular partition with interfaces $\Sigma_{ij}$ of constant curvature (namely $\II^{ij} = \k_{ij} g_{\Sigma_{ij}}$) has umbilical boundary (see Lemma \ref{lem:constant-curvature}). 
On the model spaces $\R^n$, $\S^n$, $\HH^n$, this means having (generalized) spherical interfaces. A partition with umbilical boundary satisfies \[
\II_{\Sigma_{ijk} \subset \partial \Sigma_{ij}} = \bar \II^{\partial ij}_{ijk} \, g_{\Sigma_{ijk}} ~,~ \]
and hence a conformally flattening potential $V>0$ is equivalently characterized by:
\begin{equation} \label{eq:intro-non-oriented-BCs}
\nabla_{\n_{\partial ij}} V - \bar \II^{\partial ij}_{ijk} \, V \equiv 0 \;\; \text{on $\Sigma_{ijk} \neq \emptyset$} .
\end{equation}
To abbreviate the ensuing expressions, we use $f_{ij}$ when integrating $f$ over $\Sigma_{ij} \subset \Sigma^1$ (using an arbitrary orientation, as these expressions are quadratic forms).

\begin{thm}[Conjugated Brascamp-Lieb inequality on partitions with conformally flat umbilical boundary] \label{thm:intro-BL}
Let $\Omega$ be a stationary regular partition of $(M^n, g, \mu)$, $n\geq 2$, with locally-bounded curvature and umbilical boundary. 
Assume that $\Omega$ has conformally flat boundary, and let $V>0$ be a conformally flattening boundary potential. 
Let $f = (f_{ij})$ be a scalar-field of the form $f = L_V u$, where $u = (u_{ij})$ is an \emph{admissible} scalar-field and
\[
f_{ij} =  L_V u_{ij} := V \Delta_{\Sigma^1,\mu} u_{ij} - u_{ij} \Delta_{\Sigma^1,\mu} V .
\]
Then:
\begin{enumerate}
\item $\delta^1_f \Vol  =0$. 
\item Recalling the definition (\ref{eq:intro-RicV}), assume that $\hRic^V_{\Sigma^1, \mu, N} > 0$ for some $N \in [n-1,\infty] \cup (-\infty,0)$, $N \neq1$. Then the following conjugated Brascamp-Lieb inequality holds:
 \begin{equation}\label{eq:intro-BL}
        \frac{N}{N-1} \int_{\Sigma^1} \frac{f^2_{ij}}{V} d\mu^{n-1} \leq \int_{\Sigma^1}  (V \hRic^V_{\Sigma^1, \mu, N})^{-1} \brac{\nabla_{\Sigma^1} f_{ij} - \frac{\nabla_{\Sigma^1} V}{V} f_{ij}} d\mu^{n-1}.
\end{equation}
\end{enumerate}
\end{thm}

The precise definition of admissible scalar-field $u = (u_{ij})$ is rather technical, and deferred to Section \ref{sec:scalar-fields}. Besides the scalar-field conditions $u_{ji} = -u_{ij}$ on $\Sigma_{ij}$ and $u_{ij} + u_{jk} + u_{ki} = 0$ on $\Sigma_{ijk}$, and various requirements of sufficient regularity on $\Sigma^1$ and $\Sigma^2$, a crucial requirement is that $u$ satisfies \emph{conformal boundary-conditions} (BCs), meaning that:
\[
\forall p \in \Sigma_{ijk} \;\;\;  \forall (a,b) \in \cyclic(i,j,k) \;\;\; \nabla_{\n_{\partial ab}} u_{ab} - \bar \II^{\partial ab}_{ijk} \, u_{ab} = c_p ,
\]
where $c_p$ depends only on $p$ and not on the cyclically ordered pair $(a,b) \in \cyclic(i,j,k) := \{ (i,j) , (j,k) , (k,i) \}$. 
This nomenclature stems from the fact that when $X$ is a conformal Killing vector-field, meaning that it induces a conformal (angle-preserving) flow $F_t$, then $u = (X^{\n_{ij}})$ satisfies conformal BCs on $\Sigma^2$ \cite{EMilmanNeeman-TripleAndQuadruple}.

\begin{remark}
In the setting of a single compact weighted manifold with boundary $(\Sigma,g,\mu)$, the operator $L_V$ was already considered by Huang--Zhu \cite{HuangZhu-ConjugatedBrascampLieb} (see also \cite{HuangMaZhu-ReillyFormula}). As in \cite{QiuXia-GenReillyFormula,LiXia-GenReillyFormulaSubStatic,HuangZhu-ConjugatedBrascampLieb,HuangMaZhu-ReillyFormula}, the proof of Theorem \ref{thm:intro-BL} relies on integration-by-parts of a generalized Bochner--Reilly identity involving the potential $V$. However, in our partition setting, the boundary $\partial \Sigma_{ij} = \cup_k \Sigma_{ijk}$ is not a closed manifold, and so divergences do not necessarily integrate to zero as in \cite{QiuXia-GenReillyFormula,LiXia-GenReillyFormulaSubStatic,HuangZhu-ConjugatedBrascampLieb,HuangMaZhu-ReillyFormula}. To make things worse, our test functions are not compactly supported in $\Sigma_{ij} \cup \partial \Sigma_{ij}$. Consequently, we employ delicate truncation arguments to verify that all of the integrations-by-parts on $\Sigma_{ij}$ and $\Sigma_{ijk}$ required for the proof are justified, and that the boundary terms ultimately cancel out under our assumptions. To this end, the conformal BCs play a crucial role to get an appropriate cancellation, and moreover, we verify a delicate cancellation of 12 terms on $\Sigma_{ijk\ell}$, the $(n-3)$-dimensional boundary where four cells meet. 
\end{remark}

\begin{remark}\label{rem:intro-stability}
 In complete analogy to the single bubble setting of \cite{HuangZhu-ConjugatedBrascampLieb,HuangMaZhu-ReillyFormula}, we believe that in our partition setting, for an appropriate dense subset of physical scalar-fields $f$ with $\delta^1_f \Vol = 0$,
there exists an admissible scalar-field $u$ so that $L_V u  = f$. Indeed, it was shown in \cite{EMilmanNeeman-QuintupleBubble}, at least for bounded clusters with nice Lipschitz boundaries, that $(L_{Jac},\D_{con})$ is self-adjoint on $L^2(\Sigma^1,\mu^{n-1})$ with compact resolvent, where the domain $\D_{con}$ denotes the class of Sobolev scalar-fields $f \in  H^1(\Sigma^1,\mu^{n-1})$ satisfying conformal BCs and
$\Delta_{\Sigma^1,\mu^{n-1}} f \in L^2(\Sigma^1,\mu^{n-1})$. Equivalently, since $L_V = V L_{Jac} - L_{Jac} V$, this means that $(L_V,\D_{con})$ is self-adjoint on $L^2(\Sigma^1,\frac{1}{V} \mu^{n-1})$, and it follows that $\text{Im}(L_V,\D_{con}) = \ker (L_V , \D_{con})^{\perp}$ in $L^2(\Sigma^1,\frac{1}{V} \mu^{n-1})$ by the Fredholm alternative. 
It is easy to check (see e.g.~Lemma \ref{lem:byparts}) 
that $(-L_V,\D_{con}) \geq 0$ is positive semi-definite in $L^2(\Sigma^1,\frac{1}{V} \mu^{n-1})$ and that $\ker (L_V , \D_{con})$ consists of $f \in \D_{con}$ with $\nabla_{\Sigma^1} (f/V)  \equiv 0$. 
If each $\Sigma_{ij}$ is connected, this means that $\ker (L_V , \D_{con}) = \{ (a_{ij} V) \}$ with $a_{ij} \in \R$ constant on $\Sigma_{ij}$, and using a cohomological argument such as the one in \cite[Section 8]{EMilmanNeeman-TripleAndQuadruple}, one can deduce that $a_{ij} = a_i - a_j$ for some $a \in \R^q$. In that case, $\ker (L_V , \D_{con})^{\perp}$ coincides with $f  \in L^2(\Sigma^1,\frac{1}{V} \mu^{n-1})$ so that $\int_{\Sigma^1} f_{ij} a_{ij} V \frac{d\mu^{n-1}}{V} = 0$ for all $a \in \R^q$, which is equivalent to the property $\delta^1_f \Vol = 0$. It follows that whenever $\delta^1_f \Vol = 0$, there exists $u \in \D_{con}$ so that $L_V u = f$. However, even for very smooth and compactly supported physical scalar-fields $f$, 
establishing the regularity of $u$ \textbf{up to the boundaries $\Sigma_{ijk}$ and $\Sigma_{ijk\ell}$}  (which we require in Theorem \ref{thm:intro-BL}) would constitute an entire long and difficult project in itself. At first glance this does not seem difficult since $L_V$ is elliptic, and so the interior regularity of $u_{ij}$ in $\Sigma_{ij}$ is immediate. It is also possible to obtain the boundary regularity of $u$ around $\Sigma_{ijk}$ where 3 interfaces meet, which creates an interaction of $u_{ij}$, $u_{jk}$ and $u_{ki}$ via their conformal BCs, by invoking the Agmon--Douglis--Nirenberg theory \cite{ADN2} of elliptic \textbf{systems} of PDEs around flat (and hence smooth) boundaries. Unfortunately, we also require the regularity of $u$ up to $\Sigma_{ijk\ell}$, where $6$ interfaces meet as sectors of angle $\cos^{-1}(1/3) \simeq 109^{\circ}$, and so the ADN theory would need to be \textbf{extended to convex sectors} (Lipschitz boundary would not be enough, as complications arise even for the simplest elliptic PDEs in the plane \cite[Chapter 1.4]{Grisvard-Book}, \cite{Savare-RegularityInLipDomains}). 
We have decided not pursue these technicalities here, as these would only obscure the new geometric ideas in this work, and so our stability results are not formulated for \emph{any} physical scalar-field $f$ with $\delta^1_f \Vol = 0$, but just for the ones of the form $f = L_V u$ with $u$ admissible. 
\end{remark}

\subsection{Construction of a good conformally flattening potential} \label{subsec:intro-results}

Let us now compare (\ref{eq:intro-BL}) with the stability property $Q^0(f) \geq 0$ where $f$ is a physical scalar-field and $Q^0$ is given by (\ref{eq:intro-Q0}). 
Write:
\[
- f_{ij} L_{Jac} f_{ij} = -f_{ij} \brac{L_{Jac} f_{ij} - \frac{L_{Jac} V}{V} f_{ij}} - \frac{L_{Jac} V}{V} f_{ij}^2 = - f_{ij} \brac{\Delta_{\Sigma^1,\mu}  f_{ij}- \frac{\Delta_{\Sigma^1,\mu} V}{V} f_{ij}} - \frac{L_{Jac} V}{V} f_{ij}^2 ,
\]
and so integrating (\ref{eq:intro-Q0}) by parts (see Lemma \ref{lem:Q0V} for a justification), we obtain: 
\begin{equation} \label{eq:intro-Q0V}
Q^0(f) = \sum_{i<j}  \int_{\Sigma_{ij}} \brac{\abs{\nabla_{\Sigma^1} f_{ij} - \frac{\nabla_{\Sigma^1} V}{V} f_{ij}}^2 - \frac{L_{Jac} V}{V} f_{ij}^2} d\mu^{n-1} + 
\int_{\partial \Sigma_{ij}} \brac{\frac{\nabla_{\n_{\partial ij}} V}{V} - \bar \II^{\partial ij}} f_{ij}^2 d\mu^{n-2} .
\end{equation}
When $V>0$ satisfies (\ref{eq:intro-non-oriented-BCs}), the boundary term vanishes, 
and we see that stability amounts to showing that for any physical scalar-field $f$, 
\begin{equation} \label{eq:intro-Q0VB}
\delta^1_f \Vol = 0 \;\; \Rightarrow \;\; \int_{\Sigma^1} \abs{\nabla_{\Sigma^1} f_{ij} - \frac{\nabla_{\Sigma^1} V}{V} f_{ij}}^2 d\mu^{n-1} \geq \int_{\Sigma^1} \frac{L_{Jac} V}{V} f_{ij}^2 d\mu^{n-1} . 
\end{equation}

One simple scenario, when stability can be established for all physical scalar-fields $f$, without even requiring that $\delta^1_f \Vol = 0$, is an immediate consequence of (\ref{eq:intro-Q0V}):
\begin{prop} \label{prop:intro-LJacNegative}
Let $\Omega$ be a stationary regular $q$-partition in $(M^n,g,\mu)$, $n \geq 2$ and $2 \leq q \leq \infty$. Assume the existence of a positive smooth $V>0$ satisfying (\ref{eq:intro-non-oriented-BCs}) and $L_{Jac} V \leq 0$. Then for all physical scalar-fields $f$, $Q^0(f) \geq 0$; in particular, $\Omega$ is stable.
\end{prop}
\begin{corollary}
The conclusion holds whenever $\Omega$ is in addition flat and $\Ric_{M,\mu} \leq 0$.
\end{corollary}
\begin{proof}
Since $\II^{ij} = 0$, $V \equiv 1$ satisfies (\ref{eq:intro-non-oriented-BCs}) and $L_{Jac} 1 = \Ric_{M,\mu}(\n_{ij},\n_{ij}) \leq 0$. 
\end{proof}
For example, this applies to any flat stationary regular partition in $\R^n$, $\HH^n$ (or their quotients), or on $\R^n$ equipped with $\mu = \exp(+|x|^2/2) dx$ -- these are all stable even without preserving the volume constraint.  On $\R^n$, it was shown by Choe \cite{Choe-PolyhedralIsMinimizing} that a flat stationary partition (in fact, any $m$-dimensional polyhedral set) is not only stable (without volume constraints), but in fact area minimizing under diffeomorphisms (preserving the polyhedral set's boundary). In \cite{Hales-Honeycomb}, Hales showed that on $\R^2$,  the regular hexagonal honeycomb partition minimizes perimeter-to-area density among all partitions with equal unit area cells, and that on a flat torus $\R^2 / \Gamma$ admitting such a hexagonal tiling, the hexagonal partition is perimeter minimizing among all unit area partitions (see also \cite{CDM-AlmostHoneycomb} for additional refinements).

\medskip

A more interesting scenario is a consequence of  (\ref{eq:intro-Q0VB}) in conjunction with (\ref{eq:intro-BL}):
\begin{corollary} \label{cor:intro-stability}
With the same notation and assumptions as in Theorem \ref{thm:intro-BL}, let the conformally flattening boundary potential $V>0$ also satisfy
\[
L_{Jac} V \leq K \;\; \text{and} \;\; V \hRic^V_{\Sigma^1,\mu,N} \geq \frac{N-1}{N} K g_{\Sigma^1},
\]
for some $K > 0$ and $N \in (-\infty,0) \cup [n-1,\infty]$, $N \neq 1$. Then for all physical scalar-fields $f$ of the form $f = L_V u$ where $u$ is an admissible scalar-field, 
we have $\delta^1_f \Vol = 0$ and  $Q^0(f) \geq 0$.
\end{corollary}

In view of Corollary \ref{cor:intro-stability}, to conclude the announced stability results from Subsection \ref{subsec:intro-informal}, it remains to establish:
\begin{thm} \label{thm:intro-V}
For any regular, M\"obius-flat, spherical Voronoi, $q$-partition $\Omega$ in $\R^n$ or $\S^n$, $n \geq 3$ and $2 \leq q \leq \infty$, 
  there exists a conformally flattening boundary potential $V>0$ for $\Omega$ so that 
 \begin{equation} \label{eq:intro-goodV}
 L_{Jac} V = n-1 \;\; \text{and} \;\; V \hRic^V_{\Sigma^1,n-1} = (n-2) g_{\Sigma^1}.
 \end{equation}
 With the same assumptions in $\HH^n$, there exists a conformally flattening boundary potential $V>0$ so that either (\ref{eq:intro-goodV}) holds, or else $L_{Jac} V \leq 0$. \\
 In particular, this applies to all  standard $q$-partitions of $\R^n$, $\S^n$ and $\HH^n$, $n \geq 3$, $2 \leq q \leq n+2$. 
\end{thm}

The construction of $V$ is actually fairly different in these 3 model spaces -- see Section \ref{sec:V}. As evident from the formulation, an interesting phenomenon occurs in $\HH^n$ -- some spherical Voronoi partitions (which are not induced by a cluster) actually accommodate a flattening potential $V >0$ satisfying (\ref{eq:intro-non-oriented-BCs}) with $L_{Jac} V \leq 0$, and thus fall under the scenario of Proposition \ref{prop:intro-LJacNegative}; for those partitions, we deduce that $Q^0(f) \geq 0$ for all physical scalar-fields $f$, without requiring that $\delta^1_f \Vol = 0$. In $\R^n$, we are only aware that this can happen for \emph{flat} partitions (as explained after Proposition \ref{prop:intro-LJacNegative}). 
However, on $\HH^n$, even if all cells are unbounded, they will have interfaces of constant curvature $\k_{ij} \in [-1,1]$ (when $\abs{\k_{ij}} = 1$ these are horospheres), yielding many more examples. A precise condition governing which of the above two scenarios occur in $\HH^n$ depends on the M\"obius-flattening map, and is described in Section \ref{sec:V}. 

Note that Theorem \ref{thm:intro-V} is formulated only for $n \geq 3$. 
In the unweighted ($\mu = \vol_g$) case in dimension $n=2$, one always has $V \hRic^V_{\Sigma^1} \equiv 0$ regardless of $V$, and so we are unable to use Corollary \ref{cor:intro-stability} to deduce analogous stability results on $\R^2$, $\S^2$ nor $\HH^2$. This is not entirely surprising in view of the fact that, while the Lichnerowicz spectral-gap estimate $\lambda_1 \geq \frac{n-1}{n-2} K$ on any closed $(n-1)$-dimensional manifold $M$ with $\Ric_M \geq K > 0$ yields a sharp inequality whenever $n \geq 3$ (as witnessed by $\S^{n-1}$), it fails miserably to yield the Wirtinger--Poincar\'e inequality on $\S^1$ when $n=2$. It seems that the \emph{a priori} simplest planar case requires some additional non-trivial arguments, and so this will be investigated elsewhere.

\medskip

As for the Gaussian setting, one immediately observes:
\begin{fact}
For all $2 \leq q \leq \infty$, for any stationary regular flat $q$-partition $\Omega$ of $\G^n$, $n \geq 2$, $V \equiv 1$ is a conformally flattening boundary potential with $L_{Jac} V = 1$ and $V \hRic^V_{\Sigma^1,\gamma,\infty} = 1 \cdot g_{\Sigma^1}$. 
\end{fact}
\begin{corollary}
Let $\Omega$ be stationary regular flat $q$-partition in $\G^n$, $n \geq 2$ and $2 \leq q \leq \infty$. 
Then for all physical scalar-fields $f$ of the form $f = \Delta_{\Sigma^1,\gamma} u$ where $u$ is an admissible scalar-field, 
 we have $\delta^1_f \Vol = 0$ and  $Q^0(f) \geq 0$.
\end{corollary}

The rest of this work is organized as follows. In Section \ref{sec:prelim} we go over some preliminaries and provide additional details on the various definitions used above. In Section \ref{sec:divergence} we derive some delicate divergence theorems in the presence of singularities. In Section \ref{sec:scalar-fields} we describe several classes of scalar-fields, and obtain some integration by parts formulas. In Section \ref{sec:Bochner}, we derive an integrated Bochner identity for partitions with conformally flat umbilical boundary. In Section \ref{sec:BL} we conclude the proof of Theorem \ref{thm:intro-BL}. In Section \ref{sec:V}, we establish Theorem \ref{thm:intro-V}. 

\bigskip

\noindent
\textbf{Acknowledgments.} We thank Francesco Maggi and Joe Neeman for their comments and assistance with graphics.

\section{Preliminaries} \label{sec:prelim}

\begin{definition}[Weighted Riemannian Manifold]
A smooth complete $n$-dimensional Riemannian manifold $(M^n,g)$, $n\geq 2$, endowed with a measure $\mu$ with $C^\infty$ smooth positive density $\exp(-W)$ with respect to the Riemannian volume measure $\vol_g$ is called a weighted Riemannian manifold $(M^n,g,\mu)$. 
\end{definition}

The Levi-Civita connection on $(M,g)$ is denoted by $\nabla$. The Riemannian metric $g$ will often be denoted by $\scalar{\cdot,\cdot}$. It induces a geodesic distance on $(M,g)$, and we denote by $B(x,r)$ an open geodesic ball of radius $r >0$ in $(M,g)$ centered at $x \in M$. 
Recall that $\mu^k = e^{-W} \H^k$, where $\H^k$ denotes the $k$-dimensional Hausdorff measure, and that $V_\mu = \mu$ denotes the $\mu$-weighted volume. 

We denote $E^{(q-1)} := \{ x \in \R^q_0 : \sum_{i=1}^q x_i = 0 \}$, where $\R_0^q := \{ x \in \R^q : \# \{ x_i \neq 0 \} < \infty \}$.
 We abbreviate $T(v) = T(v,v)$ for any $2$-tensor $T$.

Throughout this work we will often use the convention that $a_{ij}$ denotes $a_i - a_j$ whenever the individual objects $a_i$ are defined. 
Given distinct $i,j,k \in \{1,\ldots,q\}$, we define the set of cyclically ordered pairs in $\{i,j,k\}$:
\[
\cyclic(i,j,k) := \{ (i,j) , (j,k) , (k,i) \}  .
\]
We refer to \cite{EMilmanNeeman-GaussianMultiBubble,EMilmanNeeman-TripleAndQuadruple} and the references therein for proofs of the subsequent statements in this section.

\subsection{Weighted divergence and mean-curvature}

For a smooth hypersurface $\Sigma \subset M^n$ co-oriented by a unit-normal field $\n$, let $H_\Sigma$ denote its mean-curvature, defined as the trace of its second fundamental form $\II_{\Sigma}$. The weighted mean-curvature $H_{\Sigma,\mu}$ is defined as:
\[
H_{\Sigma,\mu} := H_{\Sigma} - \scalar{\nabla W, \n} .
\]
The induced Levi-Civita connection on $\Sigma$ is denoted by $\nabla_{\Sigma}$. We write $\div_\Sigma X$ for the surface divergence of a vector-field $X$ defined on $\Sigma$, i.e. $\sum_{i=1}^{n-1} \scalar{\tang_i,\nabla_{\tang_i} X}$ where $\{\tang_i\}$ is a local orthonormal frame on $\Sigma$. The weighted surface divergence $\div_{\Sigma,\mu}$ is defined as:
\[
\div_{\Sigma,\mu} X = \div_{\Sigma} X - \scalar{\nabla W, X},
\]
so that $\div_{\Sigma,\mu} X = \div_{\Sigma} (X e^{-W}) e^{+W}$ if $X$ is tangential to $\Sigma$.  
We will also abbreviate $\scalar{X,\n}$ by $X^\n$, and we will write $X^\tang$ for the tangential part of $X$, i.e. $X - X^{\n} \n$.  

Note that the above definitions ensure the following weighted divergence theorem: if $\Sigma$ is a smooth $(n-1)$-dimensional manifold with $C^1$ boundary denoted $\partial \Sigma$ (completeness of $\Sigma \cup \partial \Sigma$ is not required), with exterior normal $\n_{\partial}$,  and $X$ is a smooth
 tangential 
 vector-field on $\Sigma$, continuous up to $\partial \Sigma$ and with compact support in $\Sigma \cup \partial \Sigma$, then:
 \[
 \int_\Sigma \div_{\Sigma,\mu} X d\mu^{n-1} = \int_{\partial \Sigma} X^{\n_{\partial}} d\mu^{n-2} .
 \]

We denote the surface Laplacian of a smooth function $f$ on $\Sigma$ by $\Delta_{\Sigma} f := \div_{\Sigma} \nabla_{\Sigma} f$, which coincides with $\sum_{i=1}^{n-1} \nabla^2_{\tang_i,\tang_i} f - H_{\Sigma} \nabla_{\n} f$ for any smooth extension of $f$ to a neighborhood of $\Sigma$ in $M$. The weighted surface Laplacian is defined as:
\[
 \Delta_{\Sigma,\mu} f := \div_{\Sigma,\mu} \nabla_{\Sigma} f = \Delta_{\Sigma} f - \scalar{\nabla_{\Sigma} f, \nabla_{\Sigma} W} . 
 \]

\subsection{Partitions and interfaces}

Given a Borel set $U \subset \R^n$ with locally-finite perimeter, its reduced boundary $\partial^* U$ is defined as the subset of $\partial U$ for which there is a uniquely defined outer unit normal vector to $U$ in a measure theoretic sense (see \cite[Chapter 15]{MaggiBook} for a precise definition). The definition of reduced boundary canonically extends to the Riemannian setting by using a local chart, as it is known that $T(\partial^* U) = \partial^* T(U)$ for any smooth diffeomorphism $T$ (see \cite[Lemma A.1]{KMS-LimitOfCapillarity}). It is known that $\partial^* U$ is a Borel subset of $\partial U$, and that modifying $U$ on a null-set does not alter $\partial^* U$. If $U$ is an open set with $C^1$ smooth boundary, it holds that $\partial^* U = \partial U$  (e.g. \cite[Remark 15.1]{MaggiBook}). Given a Borel set $K \subset M$, the $\mu$-weighted volume and relative perimeter of $U$ in $K$ are defined as:
\[
V_\mu(U ; K) = \mu^n( U \cap K) ~,~ A_\mu(U; K) = \mu^{n-1}(\partial^* U \cap K) .
\]

Let $\Omega = (\Omega_1, \dots, \Omega_q)$ denote a generalized $q$-partition in $(M^n,g,\mu)$, $2 \leq q \leq \infty$. 
Recall that the cells $\set{\Omega_i}_{i=1,\ldots,q}$ are assumed to be pairwise disjoint Borel subsets with locally finite perimeter such that $V_\mu(M \setminus \cup_{i=1}^q \Omega_i) = 0$. Given an open subset $M' \subset M$, we say that $\Omega$ is locally finite in $M'$ if every compact set $K \subset M'$ intersects only finitely many cells; we simply say that $\Omega$ is locally finite if it is locally finite in $M$. A locally finite generalized partition is called a partition; clearly this is always the case when $q < \infty$. A $q$-partition with $q<\infty$ is called a $(q-1)$-cluster if $V_\mu(\Omega_i), A_\mu(\Omega_i) < \infty$ for $i=1,\ldots,q-1$. The cluster is said to be bounded if all $\{\Omega_i\}_{i=1,\ldots,q-1}$ are bounded. 

The $\mu$-weighted volume and total-perimeter of $\Omega$ relative to $K$ are defined as:
\[
V_{\mu}(\Omega; K) := (V_\mu(\Omega_i; K))_{i=1,\ldots,q} ~,~  A_\mu(\Omega;K) := \frac{1}{2} \sum_{i=1}^q A_\mu(\Omega_i ; K)  . 
\]
Omitting $K$ signifies using $K = M$. When $K$ is precompact, $V_{\mu}(\Omega; K) \in \R^q_0$, and the sum on the right is necessarily over a finite number of indices.

The interface between cells $i$ and $j$ (for $i \ne j$) is defined as:
\[
    \Sigma_{ij}  := \partial^* \Omega_i \cap \partial^* \Omega_j .
\]
It is standard to show (see \cite[Subsection 3.2]{EMilmanNeeman-GaussianMultiBubble}) that 
\[ \forall i \;\;\; \overline{\partial^* \Omega_i} = \overline{\cup_{j \neq i} \Sigma_{ij}} ,
\]and
\[
A_\mu(\Omega_i ; K) = \sum_{j \neq i} \mu^{n-1}(\Sigma_{ij} \cap K) ~,~
A_\mu(\Omega ; K) = \sum_{i < j} \mu^{n-1}(\Sigma_{ij} \cap K) . \]
We set:
\[     \Sigma := \bigcup_{i} \partial \Omega_i ~,~  \Sigma^1 := \bigcup_{i < j} \Sigma_{ij} .
\]

\subsection{Regularity}

Define the following cones:
\begin{align*}
    \Y &:= \{x \in E^{(2)} : \text{ there exist $i \ne j \in \{1,2,3\}$ with $x_i = x_j = \max_{k \in \{1,2,3\}} x_k$}\} , \\
    \T &:= \{x \in E^{(3)} : \text{ there exist $i \ne j \in \{1,2,3,4\}$ with $x_i = x_j = \max_{k \in \{1,2,3,4\}} x_k$}\}.
\end{align*}
Note that $\Y$ consists of $3$ half-lines meeting at the origin in $120^\circ$ angles, and that $\T$ consists of $6$ two-dimensional sectors meeting in threes at $120^{\circ}$ angles along $4$ half-lines, which in turn all meet at the origin in $\cos^{-1}(-1/3) \simeq 109^{\circ}$ angles. 

\begin{definition}[Regularity]
A partition $\Omega$ in $(M^n,g,\mu)$ is called regular if it satisfies the following:
\begin{enumerate}
\item $\Omega$ may and will be modified on a $\mu$-null set so that all of its cells are open, and so that for every $i$, $\overline{\partial^* \Omega_i} = \partial \Omega_i$ and $\mu^{n-1} (\partial \Omega_i \setminus \partial^* \Omega_i) =0$; in particular, $\Sigma = \overline{\Sigma^1}$. \item $\Sigma$ is the disjoint union of $\Sigma^1$ and sets $\Sigma^2,\Sigma^3,\Sigma^4$ satisfying (for some fixed $\alpha > 0$):
\begin{enumerate}
\item $\Sigma^1$ is a locally-finite union of embedded $(n-1)$-dimensional $C^{\infty}$ manifolds, and for every $p \in \Sigma^1$, $\Sigma$ around $p$ is locally $C^\infty$ diffeomorphic to $\{0\} \times \R^{n-1}$.
\item $\Sigma^2$ is a locally-finite union of embedded $(n-2)$-dimensional $C^{\infty}$ manifolds, and for every $p \in \Sigma^2$, $\Sigma$ around $p$ is locally $C^\infty$ diffeomorphic to $\Y \times \R^{n-2}$.
\item $\Sigma^3$ is a locally-finite union of embedded $(n-3)$-dimensional $C^{1,\alpha}$ manifolds, and for every $p \in \Sigma^3$, $\Sigma$ around $p$ is locally $C^{1,\alpha}$ diffeomorphic to $\T \times \R^{n-3}$.
\item $\Sigma^4$ is closed and has locally-finite $\H^{n-4}$ measure. 
\end{enumerate}
\item (Density upper bound) For any compact set $K$ in $M$, there exist constants $\Lambda_K, r_K>0$ so that:
\begin{equation} \label{eq:density-Sigma1}
\mu^{n-1} (\Sigma \cap B(x, r)) \leq \Lambda_K r^{n-1} , \quad \forall  x \in \Sigma \cap K \quad \forall r \in (0, r_K).
\end{equation}
In particular, $\mu^{n-1}(\Sigma^1 \cap K) < \infty$. 
\item (Infiltration property) There exists a constant $\eps > 0$ so that for any $p \in M^n$, $i=1,\ldots,q$ and connected component $\Omega_i^\ell$ of $\Omega_i$: 
\[ \liminf_{r \rightarrow 0+} \frac{\H^n(\Omega^\ell_i \cap B(p,r))}{\H^n(B(p,r))} < \eps \;\; \Rightarrow \;\; p \notin \overline{\Omega^\ell_i} . 
\] \end{enumerate}
We abbreviate $\Sigma^{\leq m} := \cup_{1 \leq i \leq m} \Sigma^i$ and $\Sigma^{\geq m} := \cup_{m \leq i \leq 4} \Sigma^i$. 
\end{definition}

Let $\Omega$ be a regular partition. We denote by $\n_{ij}$ be the (smooth) unit normal field along $\Sigma_{ij}$ that points from $\Omega_i$ to $\Omega_j$. We use $\n_{ij}$ to co-orient $\Sigma_{ij}$, and since $\n_{ij} = -\n_{ji}$, note that $\Sigma_{ij}$ and $\Sigma_{ji}$ have opposite orientations. We denote by $\II^{ij}$ the second fundamental form on $\Sigma_{ij}$. When $i$ and $j$ are clear from the context, we will simply write $\n$ and $\II = \II_{\Sigma^1}$. 
We will typically abbreviate $H_{\Sigma_{ij}}$ and $H_{\Sigma_{ij},\mu}$ by $H_{ij}$ and $H_{ij,\mu}$, respectively. Every point in $\Sigma^2$ (called the \emph{triple-point set}) belongs to the closure of exactly three cells, as well as to the closure of exactly three interfaces. Given distinct $i,j,k$, we will write $\Sigma_{ijk}$ for the subset of $\Sigma^2$ which belongs to the closure of $\Omega_i$, $\Omega_j$ and $\Omega_k$, or equivalently, to the closure of $\Sigma_{ij}$, $\Sigma_{jk}$ and $\Sigma_{ki}$. It follows that $\Sigma^2$ is the disjoint union of $\Sigma_{ijk}$ over all $i < j < k$. 
Similarly, we will call $\Sigma^3$ the \emph{quadruple-point set}, and given distinct $i,j,k,\ell$, denote by $\Sigma_{ijk\ell}$ the subset of $\Sigma^3$ which belongs to the closure of $\Omega_i$, $\Omega_j$, $\Omega_k$ and $\Omega_\ell$, or equivalently, to the closure of all six $\Sigma_{ab}$ for distinct $\{a,b\} \subset \{ i , j , k, \ell \}$. It follows that $\Sigma^3$ is the disjoint union of $\Sigma_{ijk\ell}$ over all $i < j < k<\ell$.

\smallskip

We denote $\partial \Sigma_{ij} := \cup_{k} \Sigma_{ijk}$, i.e.~the intersection of the topological boundary of $\Sigma_{ij}$ with $\Sigma^2$. Similarly, we denote by $\partial \Sigma_{ijk} := \cup_{\ell} \Sigma_{ijk\ell}$, the intersection of the topological boundary of $\Sigma_{ijk}$ with $\Sigma^3$. 
We will extend the normal fields $\n_{ij}$ to $\Sigma_{ijk}$ and $\Sigma_{ijk\ell}$ 
by continuity (thanks to $C^1$ regularity). We also denote by $\II_{\partial \Sigma_{ij}}$ and $H_{\partial \Sigma_{ij}}$ the second fundamental form and mean-curvature (respectively) of $\partial \Sigma_{ij}$ in $\Sigma_{ij}$; at times we will write $\II_{\Sigma_{ijk} \subset \partial \Sigma_{ij}}$.

\smallskip

The outer normal to $\partial \Sigma_{ij}$ in $\Sigma_{ij}$ is denoted by $\n_{\partial ij}$, and the outer normal to $\partial \Sigma_{ijk}$ in $\Sigma_{ijk}$ by $\n_{\partial ijk}$. On $\partial \Sigma_{ij}$, we abbreviate $\II^{ij}(\n_{\partial ij},\n_{\partial ij})$ by $\II^{ij}_{\partial \partial}$. 

\smallskip

We remark that by local finiteness of a partition, expressions of the form $\sum_{i<j} \int_{\Sigma_{ij}} w_{ij} d\mu^{n-1}$ or $\sum_{i<j<k} \int_{\Sigma_{ijk}} w_{ijk} d\mu^{n-2}$ only consist of a finite sum whenever $w_{ij}$ or $w_{ijk}$ are compactly supported in $M^n$ (respectively).

\subsection{Stationarity}

Given a vector-field $X \in C_c^\infty(M)$ supported in the interior of a compact $K \subset M$, the flow along $X$ for time $t$ is denoted by $F_t$, the perturbed partition is denoted by $F_t(\Omega) = (F_t(\Omega_i))_{i=1,\ldots,q}$, and its $k$-th variation of volume and total perimeter are denoted by:
\[
\delta^k_X \Vol := \left . \frac{d^k}{(dt)^k} \right|_{t=0} V_\mu(F_t(\Omega);K) \in E^{(q-1)} ~,~ \delta^k_X \Area := \left .  \frac{d^k}{(dt)^k}  \right |_{t=0} A_\mu(F_t(\Omega);K)  \in \R .
\]
Clearly, these quantities do not depend on the container set $K$, and so $K$ is omitted from the notation. 

A regular partition $\Omega$ is called stationary if for all vector-fields $X \in C_c^\infty(M)$, 
\[
\delta^1_X \Vol = 0 \;\; \Rightarrow \;\; \delta^1_X \Area = 0 ,
\]
or equivalently (see \cite[Appendix C]{EMilmanNeeman-GaussianMultiBubble}), if for all $X \in C_c^\infty(M)$,
\[
\delta^1_X \Area = \scalar{\lambda,\delta^1_X \Vol} ,
\]
for some vector $\lambda \in \R^q$ of Lagrange multipliers. A well-known geometric interpretation is as follows (see \cite[Lemma 4.3 and Corollary 5.5]{EMilmanNeeman-GaussianMultiBubble} and \cite[Lemma 3.3]{EMilmanNeeman-TripleAndQuadruple}). 

\begin{lemma}[Stationarity] \label{lem:stationarity}
Let $\Omega$ be a regular $q$-partition in $(M,g,\mu)$. If $\Omega$ is stationary with Lagrange multipliers $\lambda \in \R^q$, then:
\begin{enumerate}[(1)]
\item \label{it:stationarity-k}
For all $i < j$, $H_{\Sigma_{ij},\mu}$ is constant and equal to $\lambda_i - \lambda_j$.
\item \label{it:stationarity-n} 
$\sum_{(i,j) \in \cyclic(u,v,w)} \n_{ij} = 0$ on $\Sigma_{uvw}$ for all $u < v < w$. 
\end{enumerate}
Conversely, if $\Omega$ is of locally bounded curvature  (to be defined below) and satisfies \ref{it:stationarity-k} and \ref{it:stationarity-n}, then it is stationary with Lagrange multipliers $\lambda \in \R^q$. 
\end{lemma}
Property \ref{it:stationarity-n} precisely means that $\Sigma_{uv}$, $\Sigma_{vw}$ and $\Sigma_{wu}$ meet at $\Sigma_{uvw}$ in $120^{\circ}$ angles. This is of course equivalent to satisfying:
\[ \n_{\partial ij} = \frac{\n_{ik} + \n_{jk}}{\sqrt{3}} \;\; \text{on $\Sigma_{ijk}$ for all $i,j,k$} ,
\] and also to:
\[
 \sum_{(i,j) \in \cyclic(u,v,w)} \n_{\partial ij} = 0  \;\; \text{on $\Sigma_{uvw}$ for all $u<v<w$} .
\]

A useful consequence of stationarity is the following cancellation identity. 
\begin{lemma}\label{lem:three-tensor-vanishes}
Let $\Omega$ be a stationary regular partition in $(M,g,\mu)$. Then at every point $p \in \Sigma_{uvw}$, the following $3$-tensor identically vanishes:
    \begin{equation} \label{eq:three-tensor}
    T^{\alpha \beta \gamma} = \sum_{(i,j) \in \cyclic(u,v,w)} 
    \brac{\n_{i j}^{\alpha} \n_{i j}^\beta \n_{\partial i j}^{\gamma} - \n_{\partial i j}^{\alpha} \n_{i j}^\beta \n_{i j}^{\gamma}} \equiv 0. 
    \end{equation}
 \end{lemma}

\subsection{Locally bounded curvature}

Note that regularity only ensures that $\Sigma$ is $C^{1,\alpha}$ smooth near $\Sigma^3$, and so the second fundamental form $\II_{\Sigma_1}$ could be blowing up near $\Sigma^3$. In general, $\II_{\Sigma^1}$ is only locally square integrable on $\Sigma^1$, and locally integrable on $\Sigma^2$ (see \cite[Proposition 5.7]{EMilmanNeeman-GaussianMultiBubble}, \cite[Proposition 2.23]{EMilmanNeeman-TripleAndQuadruple}). This will unavoidably create integrability considerations for $\nabla_{\Sigma^1} X^{\n_{ij}}$ on $\Sigma_{ij}$. To avoid these complications, we will henceforth assume that $\Omega$ has locally bounded curvature. 

\begin{definition}[Locally bounded curvature]
A regular partition $\Omega$ is said to have locally bounded curvature if $\II_{\Sigma^1}$ is bounded on every compact set. 
\end{definition}

It is known that an isoperimetric minimizing cluster on $(M^n,g,\mu)$ is regular, stationary and stable (see below for the definition of stability). As these are local properties, the proofs extend to locally minimizing partitions as well. However, the property of having locally bounded curvature is not automatic (and indeed there are area-minimizing cones where the curvature blows up near the vertex); consequently, we did not include this assumption in the definition of regularity. 

\medskip

The locally bounded curvature assumption is very useful. We've already seen in Lemma \ref{lem:stationarity} that locally bounded curvature allows for a convenient  geometric characterization of stationarity. In addition, 
it makes it possible to test the stability of $\Omega$ using general compactly supported fields, without assuming they are bounded away from the singular set $\Sigma^4$ (see Theorem \ref{thm:Q0LJac} below). 
Furthermore, we have:

\begin{lemma} \label{lem:Sigma2}
Let $\Omega$ be a stationary regular partition in $(M,g,\mu)$ having locally bounded curvature. Then for any compact subset $K \subset M$, there exist $\Lambda_{K},r_{K} \in (0,\infty)$ so that:
\begin{equation} \label{eq:density-Sigma2}
\mu^{n-2}(\Sigma^2 \cap B(x,r)) \leq \Lambda_{K} r^{n-2} \;\;\; \forall x \in \Sigma \cap K \;\;\; \forall r \in (0,r_{K}) . 
\end{equation}
In particular $\mu^{n-2}(\Sigma^2 \cap K) < \infty$. 
\end{lemma}

For a regular partition with locally bounded curvature, the second fundamental form $\II^{ij}$ in $\Sigma_{ij}$ extends by continuity to $\overline{\Sigma_{ij}} \setminus \Sigma^4$. For any $v \in T (\overline{\Sigma_{ij}} \setminus \Sigma^4)$, we denote $\II^{ij}_v := \II^{ij}(v,v)$.

\subsection{Stability}

\begin{definition}[Index Form and Stability]
A stationary regular partition $\Omega$ is called stable if for every vector-field $X \in C_c^\infty(M)$:
\[
\delta^1_X \Vol = 0 \;\; \Rightarrow \;\; Q(X) := \delta^2_X \Area - \scalar{\lambda, \delta^2_X \Vol} \geq 0 . 
\]
The quadratic form $Q$ is called the partition's index form. 
\end{definition}

Given a vector-field $X \in C_c^\infty(M)$, define the physical scalar-field $f$ to be the tuple $(f_{ij})_{i \neq j}$ where $f_{ij} = X^{\n_{ij}}$ on $\Sigma_{ij} \cup \partial \Sigma_{ij}$ is its normal component; a precise definition of several classes of scalar-fields will be given in Section \ref{sec:scalar-fields}. It is well-known that $\delta^1_X \Vol$ only depends on $f$, namely
\[
\delta^1_X \Vol = \delta^1_f \Vol ,
\]
where $\delta^1_f \Vol$ is defined as follows. 
\begin{definition}[First variation of volume]
Given an oriented collection $f = (f_{ij})$ of integrable functions $f_{ij} = -f_{ji}$ on $\Sigma_{ij}$, which are supported on a common compact set $K \subset M$, the first variation of volume induced by $f$ is defined as
\[
\delta^1_f \Vol := \brac{\int_{\partial \Omega_i} f_{ij} d\mu^{n-1}}_i = \brac{\sum_{j \neq i} \int_{\Sigma_{ij}} f_{ij} d\mu^{n-1}}_i  \in E^{(q-1)} . 
\]
\end{definition}

A much less trivial fact is that under the appropriate assumptions on $\Omega$, $Q(X)$ only depends on $f$ as well: 

\begin{thm} \label{thm:Q0LJac}
Let $\Omega$ be a stationary regular partition with locally bounded curvature. Then for any vector-field $X \in C_c^\infty(M)$, $Q(X) = Q^0(f)$, where $f = (X^{\n_{ij}})$ and $Q^0(f)$ is given by any of the following two equivalent expressions:
\begin{align} 
\label{eq:Q0LJac} Q^0(f)  := & \sum_{i<j} \brac{ -\int_{\Sigma_{ij}} f_{ij} L_{Jac} f_{ij} d\mu^{n-1} + \int_{\partial \Sigma_{ij}} (\nabla_{\n_{\partial ij}} f_{ij} - \bar \II^{\partial ij} f_{ij}) f_{ij} d\mu^{n-2}} \\
\label{eq:Q0nabla}  = & \sum_{i<j}  \brac{\int_{\Sigma_{ij}} \brac{|\nabla_{\Sigma^1} f_{ij}|^2 - (L_{Jac} 1) f_{ij}^2 } d\mu^{n-1} - \int_{\partial \Sigma_{ij}} \bar \II^{\partial ij} f_{ij}^2 d\mu^{n-2}}  . 
\end{align}
\end{thm}
\noindent
The Jacobi operator $L_{Jac}$ and curvature $\bar \II^{\partial ij}$ were defined in (\ref{eq:intro-LJac}) and (\ref{eq:intro-barII}), respectively. Consequently, stationarity amounts to the property that for any physical scalar-field $f$ as above,
\[
\delta^1_f \Vol = 0 \;\; \Rightarrow \;\; Q^0(f) \geq 0 . 
\]

\subsection{Generalized spheres and quasi-centers} \label{subsec:prelim-gen-spheres}

Let $\R^{n,1}$ denote the $(n+1)$-dimensional Minkowski space-time $\R^n \times \R$ equipped with the metric  
\begin{equation*}
\scalar{x, y}_1 =\sum_{i=1}^n x_i y_i - x_{0} y_{0},  
\end{equation*}
where we write $y = (y_1,\ldots,y_n,y_0) = (\underline{y},y_0) \in \R^n \times \R$ and similarly for $x$. 
The hyperboloid model of the $n$-dimensional hyperbolic space in $\R^{n,1}$ is given by
\begin{equation*}
\HH^n= \set{ y = (\underline{y},y_0) \in \R^{n,1} :  \scalar{y, y}_1 = -1, ~ y_0 >0 }.
\end{equation*}
We use this model throughout all definitions and computations in this work. 

\medskip

Recall that a generalized sphere $S$ in $M^n \in \{\R^n,\S^n,\HH^n\}$ is a complete hypersurface of constant curvature $\k$. On $\S^n$, this is just a usual geodesic sphere, on $\R^n$ this also includes hyperplanes, and on $\HH^n$, $S$ is either 
 a geodesic sphere ($\abs{\k} > 1$), a horosphere ($\abs{\k}=1$), or an equidistant hypersurface ($0 \leq \abs{\k} < 1$).

\begin{definition}[Quasi-center of generalized sphere]
The quasi-center $\c$ of a generalized sphere $S$ co-oriented by the unit-normal $\n$ on a model space $M^n \in \{ \R^n, \S^n, \HH^n\}$ is the vector $\c:= \n-\k p$ at any point $p \in S$, where $\k$ is the curvature of $S$ with respect to $\n$. It is easy to check that $\c$ does not depend on the choice of point $p \in S$. 
\end{definition}

\subsection{Spherical Voronoi and standard partitions} \label{subsec:prelim-Voronoi}

In the definition of spherical Voronoi (generalized) partition below, we employ a slightly more restrictive definition than the (equivalent) ones in \cite{EMilmanNeeman-TripleAndQuadruple,EMilmanNeeman-QuintupleBubble} -- namely, we assume that all cells are non-empty, and in addition that every non-empty interface is a relatively open subset of a geodesic sphere (see \cite[Section 3]{EMilmanNeeman-QuintupleBubble} for a discussion on possible degeneracies that may occur otherwise). The latter assumption is in any case implied by regularity, and the former one is to avoid pathological cases. 

\begin{definition}[Spherical Voronoi (generalized) partition in $\S^n$] \label{def:prelim-Voronoi}
A generalized partition $\Omega = (\Omega_1,\ldots,\Omega_q)$ in $\S^n$, all of whose cells are non-empty, 
is called a spherical Voronoi generalized partition if there exist $\{\c_i^\S \}_{i=1, \ldots, q} \subset \R^{n+1}$ and $\{ \k_i^\S\}_{i=1, \ldots, q} \subset \R$ so that the following holds: 
\begin{enumerate} 
\item For every non-empty interface $\Sigma_{ij} \neq \emptyset$, $\Sigma_{ij}$ is a relatively open subset of a geodesic sphere $S^\S_{ij}$ with quasi-center $\c_{ij}^\S = \c_i^\S -\c_j^\S$ and curvature $\k_{ij}^\S = \k_i^\S - \k_j^\S$. 
\item The following Voronoi representation holds:
\[
\Omega_i = \set{ p \in \S^n :~ \argmin_{j=1, \ldots, q} (\sscalar{\c_j^\S ,p} + \k_j^\S)  = \{i\} } = \bigcap_{j \neq i} \set{ p \in \S^n:~ \sscalar{\c_{ij}^\S ,p} + \k_{ij}^\S <0 }.
\]
\end{enumerate}
When $q < \infty$, $\Omega$ is called a spherical Voronoi partition. 
\end{definition}
Here $\mathit{arg}\,\mathit{min}$ denotes the subset of indices on which the corresponding minimum is attained; if the minimum is not attained (when $q=\infty$), we set the result of $\argmin$ to be the empty set.

\medskip

Let $\bar \R^n = \R^n \cup \{\infty\}$ denote the one-point-at-infinity compactification of $\R^n$, which is diffeomorphic to $\S^n$.
\begin{definition}[M\"obius automorphism of $\S^n$]
A M\"obius automorphism of $\S^n$ is a composition of stereographic projections to and from $\bar \R^n$. 
\end{definition}
\noindent
A more explicit description of the M\"obius group via the Lorentz group of isometries of Minkowski space-time may be found in \cite[Section 10]{EMilmanNeeman-TripleAndQuadruple}, but we will not require this here. 
\begin{lem} \label{lem:MobiusPreserves}
Spherical Voronoi generalized partitions in $\S^n$ are preserved under M\"obius automorphisms. 
\end{lem}
\begin{proof}
For clusters in $\S^n$, this was proved in \cite[Lemma 10.2]{EMilmanNeeman-TripleAndQuadruple}, but the proof remains valid for arbitrary generalized partitions.
\end{proof}

\begin{definition}[M\"obius-flat spherical Voronoi parition in $\S^n$] \label{def:prelim-Mobius-flat}
A spherical Voronoi generalized partition in $\S^n$ is called M\"obius-flat if it has a M\"obius image which is flat, namely all of its non-empty interfaces $\Sigma_{ij}$ are relatively open subsets of spheres $S^\S_{ij}$ having zero curvature. 
\end{definition}

To extend these definitions from $\S^n$ to $\R^n$ and $\HH^n$, we use the stereographic projections $\pi_\R : \R^n \rightarrow \S^n \setminus \{N\}$ and $\pi_\HH : \HH^n \rightarrow \S^n_{+}$, respectively (see Figure \ref{fig:stereographic}). Here $N := e_{n+1}$ is the North pole in $\S^n$ and $\S^n_{\pm} := \{ p \in \S^n : \pm \scalar{p,N} > 0\}$ denotes the Northern and Southern hemispheres, respectively. Observe that the stereographic projection of a countable partition ($q=\infty$) in $\R^n$ or $\HH^n$ will no longer be locally finite in $\S^n$ or $\S^n_+$, and so will only be a generalized partition. An explicit description of $\pi_\R$ and $\pi_\HH$, as well as of the notion of spherical Voronoi partition on $\R^n$ and $\HH^n$, 
is deferred to Section \ref{sec:V}. For now, let us only mention that the non-empty interfaces $\Sigma^\M_{ij}$ of a spherical Voronoi partition in $\M^n$, $\M \in \{\R,\S,\HH\}$, are guaranteed to lie on generalized spheres $S^\M_{ij}$, having ``separable" curvatures $\k^\M_{ij} = \k^\M_i - \k^\M_j$ and quasi-centers $\c^\M_{ij} = \c^\M_i - \c^\M_j$.

\begin{lem}  \label{lem:RegularVoronoiIsStationary}
A regular spherical Voronoi partition in $M^n \in \{ \R^n, \S^n, \HH^n\}$ is stationary. 
\end{lem}
\begin{proof}
This was proved for $M^n \in \{ \R^n, \S^n\}$ in \cite[Subsection 8.5]{EMilmanNeeman-TripleAndQuadruple}. The same idea applies to $\HH^n$, and will be verified in Lemma \ref{lem:HnStationary}. 
\end{proof}

\begin{definition}[Standard flat partition in $\S^n$]
A standard flat $q$-partition in $\S^n$, $2 \leq q \leq n+2$, is the spherical Voronoi partition $\Omega^\S$ given by:
\[
\Omega^\S_i = \set{ p \in \S^n : \argmin_{j=1,\ldots,q} \scalar{p,\c_j} = \{i\} } ,
\]
where $\c_i$ are $q$-equidistant points in $R_{q,n} \S^n \subset \R^{n+1}$ (the value of $R_{q,n} > 0$ is immaterial, but can be chosen so that $|\c_{ij}| = 1$ for all $i \neq j$). 
\end{definition}

\begin{definition}[Standard partition in $\S^n$]
A standard $q$-partition $\Omega^\S$ in $\S^n$, $2 \leq q \leq n+2$, is a M\"obius image of a standard flat $q$-partition in $\S^n$. 
\end{definition}

\begin{definition}[Standard partition in $\R^n$ and $\HH^n$]
A standard $q$-partition $\Omega^\M$ in $\M^n$, $\M \in \{\R, \HH\}$, $2 \leq q \leq n+2$, is defined as $\Omega^\M = \pi_\M^{-1}(\Omega^\S)$, where $\Omega^\S$ is a standard $q$-partition in $\S^n$. 
\end{definition}
By Lemma \ref{lem:pull-back SphVor-hyper}, if $\Omega^\S$ is a spherical Voronoi $q_\S$-partition in $\S^n$, then $\pi_\HH^{-1}(\Omega^\S)$ is a spherical Voronoi $q_\HH$-partition in $\HH^n$ for some $1 \leq q_\HH \leq q_\S$. Consequently, part of the requirement in the definition of a standard $q$-partition $\Omega^\HH$ in $\HH^n$ is that $q_\S = q_\HH = q$. The resulting standard $q$-partition in $\HH^n$ is called a standard $(q-1)$-bubble if it is induced by a $(q-1)$-cluster, or equivalently (see Lemma \ref{lem:bounded}), if $\Omega^\S_q \supset \overline{\S^n_-}$.  Similarly, a standard $q$-partition in $\R^n$ is called a standard $(q-1)$-bubble if it is induced by a $(q-1)$-cluster, or equivalently,  if  $\Omega^\S_q \ni N$. 
  
\smallskip
By Lemmas \ref{lem:MobiusPreserves} and \ref{lem:pull-back SphVor-hyper} and Definition \ref{def:VoronoiR}, standard partitions in $M^n \in \{\R^n, \S^n, \HH^n\}$ are M\"obius-flat spherical Voronoi partitions. By construction, standard partitions are regular in $\S^n$, hence in $\S^n \setminus \{N\}$ and $\S^n_+$, and since regularity is preserved under (conformal) $C^\infty$-diffeomorphisms, we confirm that standard partitions in $\R^n$ and $\HH^n$ are also regular.

\section{Divergence Theorems} \label{sec:divergence}

A function on $M$ which is $C^\infty$ smooth and takes values in $[0, 1]$ is called a cutoff function. 

\begin{lemma}\label{lem:cutoff} 
Let $\Omega$ be a stationary regular partition in $(M^n,g,\mu)$ with locally bounded curvature.  
 Fix $m\in \{1,2\}$ and a compact set $K \subset M$.  Then for every $\epsilon > 0$, there is a cutoff function $\eta_{\eps}$  such that:
 \begin{enumerate}
 \item $\eta_\eps \equiv 0$ on an open neighborhood of $\Sigma^{\geq 2+m} \cap K$.
 \item $\eta_{\eps} \rightarrow 1$ $\H^{n-m}$-almost-everywhere as $\eps \rightarrow 0$.
 \item $\int_{\Sigma^m} |\nabla \eta_\eps|^{2} \, d\mu^{n-m} \le C_K$, where $C_K$ depends only on $K$ (and not on $\eps$).
\end{enumerate}
\end{lemma}
\begin{proof}
Fix $\eps \in (0,1)$. By the upper density estimates (\ref{eq:density-Sigma1}) or (\ref{eq:density-Sigma2}), there exist $\Lambda_K > 0$ and $r_0 \in (0,\eps)$ so that:
\begin{equation} \label{eq:cutoff-density}
    \mu^{n-m}(\Sigma^m \cap B(x,4 r)) \leq \Lambda_K r^{n-m} \;\;\; \forall x \in \Sigma \cap K \;\; \forall r \in (0,r_0) .
\end{equation}
        
By regularity, $\H^{n-m-2}(\Sigma^{\geq 2+m}\cap K)$ is finite. For any $\delta_0 \in (0,r_0)$, the definition of Hausdorff measure and compactness imply the existence of finite sequences $x_i \in \Sigma^{\geq 2+m} \cap K$ and $\delta_i \in (0, \delta_0)$, $i=1,\ldots,N$, such that the sets $\{B(x_i, \delta_i)\}$ cover $\Sigma^{\geq 2+m} \cap K$ and $\sum_{i=1}^N \delta_i^{n-m-2} \le C_n \H^{n-m-2} (\Sigma^{\geq 2+m} \cap K) + 1$, where $C_n >0$ is a numeric constant depending only on $n$. We may of course assume that $B(x_i, \delta_i) \cap (\Sigma^{\geq 2+m} \cap K) \neq \emptyset$ for all $i$. 
         
 For each $i$, define the $\frac{1}{\delta_i}$-Lipschitz function $\eta_i(x) = \frac{1}{\delta_i} d(x, B(x_i,2 \delta_i)) \wedge 1$, so that $\eta_i \equiv 1$ outside of $B(x_i,3 \delta_i)$ and $\eta_i \equiv 0$ inside $B(x_i,2 \delta_i)$. Set $\tilde \eta(x) := \min_{i=1,\ldots,N} \eta_i(x)$, so that $\tilde \eta$ vanishes on $\cup_i  B(x_i,2 \delta_i)$ and is identically $1$ on $M^n \setminus \cup_i B(x_i,3 \delta_i)$. We now define the cutoff function $\eta$ to be a $C^\infty$ mollification of $\tilde \eta$ using a partition of unity as in \cite[Section 2]{GreeneWu-SmoothApproximations}. Since $\delta := \min_{i=1,\ldots,N} \delta_i > 0$, we may ensure that $\eta$ vanishes on $\cup_i  B(x_i, \delta_i)$, an open neighborhood of $\Sigma^{\geq 2+m}\cap K$,  is identically $1$ on $M^n \setminus \cup_i  B(x_i, 4 \delta_i)$, and satisfies for all $x \in M^n$ that
\begin{equation} \label{eq:mollification}
    |\nabla \eta(x)| \leq \max \set{ \frac{2}{\delta_i} :  x \in B(x_i , 4 \delta_i) \setminus B(x_i , \delta_i) } . 
\end{equation}
Since $\delta_i \leq \eps$ and $B(x_i, \delta_i) \cap (\Sigma^{\geq 2+m} \cap K) \neq \emptyset$, it follows that $\eta \equiv 1$ on $M^n \setminus (\Sigma^{\geq 2+m} \cap K)_{8 \eps}$. As $\Sigma^{\geq 2+m} \cap K$ is closed, we confirm that $\eta = \eta_\eps$ tends to $1$ as $\eps \rightarrow 0$ except on $\Sigma^{\geq 2+m} \cap K$, and hence $\H^{n-m}$-almost-everywhere.

From (\ref{eq:mollification}), we deduce for all $x \in M^n$ that
    \[
    | \nabla \eta(x) |^{2} \leq \sum_{i} \brac{\frac{2}{\delta_i}}^2 1_{B(x_i,4 \delta_i) \setminus B(x_i,\delta_i)}(x) .
    \]
    Consequently, by (\ref{eq:cutoff-density}):
    \begin{align*}
    \int_{\Sigma^m} |\nabla \eta|^{2}\, d\mu^{n-m} & \leq  \sum_i \brac{\frac{2}{\delta_i}}^2  \mu^{n-m}(\Sigma^m \cap B(x_i,4 \delta_i) ) \\
    & \leq \sum_{i} 4 \Lambda_K \delta_i^{n-m-2} \leq 4 \Lambda_K( C_n \H^{n-m-2} (\Sigma^{\geq 2+m} \cap K) +1) < \infty . 
    \end{align*}
    This concludes the proof. 
\end{proof}

\begin{lemma}[Divergence theorem on $\Sigma^2$] \label{lem:StokesSigma2}
    Let $\Omega$ be a stationary regular partition in $(M^n,g,\mu)$ with locally bounded curvature.  Given $i<j<k$, let $Z_{ijk}$ denote a $C^1$-smooth tangential vector-field on $\Sigma_{ijk}$ which is continuous up to $\partial \Sigma_{ijk}$. Assume that there exists a compact set $K \subset M^n$ so that all $Z_{ijk}$ are supported in $K$ and satisfy:
    \begin{equation} \label{eq:StokesSigma2-dominant}
        \sum_{i<j<k}  \int_{\Sigma_{ijk}} |\div_{\Sigma^2,\mu} Z_{ijk}| d\mu^{n-2} < \infty ~,~
        \sum_{a<b<c<d} \int_{\Sigma_{abcd}} \sabs[\Big]{\sum_{\substack{\{i,j,k\} \subset \\ \{a,b,c,d\}}}  Z_{ij k}^{\n_{\partial ijk}}} d\mu^{n-3} < \infty . 
      \end{equation}
      In addition, assume that:
      \begin{equation} \label{eq:StokesSigma2-dominant2}
         \sum_{i<j<k} \int_{\Sigma_{ijk}} |Z_{ijk}|^{2} d\mu^{n-2} < \infty . 
    \end{equation}
        Then:
    \begin{equation} \label{eq:StokesSigma2-conclusion}
        \sum_{i<j<k} \int_{\Sigma_{ijk}} \div_{\Sigma^2,\mu} Z_{ij k} \, d\mu^{n-2} = \sum_{a<b<c<d} \int_{\Sigma_{abcd}} \sum_{\substack{\{i,j,k\} \subset \\ \{a,b,c,d\}}}  Z_{ij k}^{\n_{\partial ijk}}  d\mu^{n-3} . 
    \end{equation}
\end{lemma}
\begin{proof}
Given $\eps > 0$, let $\eta_\eps : M^n \rightarrow [0,1]$ be a cutoff function as in Lemma \ref{lem:cutoff} corresponding to $K$. Applying the divergence theorem to $\eta_\eps Z_{ijk}$ (supported away from $\Sigma^4$ and infinity) and summing over $i<j<k$, we obtain:
\begin{align*}
& \sum_{i<j<k} \int_{\Sigma_{ijk}} \eta_\eps \div_{\Sigma^2,\mu} (Z_{ij k}) d\mu^{n-2} + 
\sum_{i<j<k} \int_{\Sigma_{ijk}} \scalar{\nabla_{\Sigma^2} \eta_\eps , Z_{ij k}} d\mu^{n-2} \\
& =  \sum_{i<j<k} \int_{\Sigma_{ijk}} \div_{\Sigma^2,\mu} (\eta_\eps Z_{ij k}) d\mu^{n-2} =  \sum_{a<b<c<d} \int_{\Sigma_{abcd}} \eta_\eps  \sum_{\substack{\{i,j,k\} \subset \\ \{a,b,c,d\}}}  Z_{ij k}^{\n_{\partial ijk}} d\mu^{n-3} .
\end{align*}
By (\ref{eq:StokesSigma2-dominant}) and as $\eta_\eps \rightarrow 1$ $\H^{n-2}$-almost-everywhere as $\eps \rightarrow 0$, the Dominant Convergence Theorem implies that the first and last terms above converge to the left and right terms in (\ref{eq:StokesSigma2-conclusion}), respectively. It remains to show that the second term above tends to $0$. 

Indeed, using that $|\nabla \eta_\eps| = 0$ on the set $\{ \eta_\eps = 1\}$, the Cauchy-Schwarz inequality, and the estimate $\int_{\Sigma^2} |\nabla \eta_\eps|^2 d\mu^{n-2} \leq C_K$ from Lemma \ref{lem:cutoff}, 
we bound:
\[
 \sum_{i<j<k} \int_{\Sigma_{ijk}} \abs{\scalar{\nabla_{\Sigma^2} \eta_\eps , Z_{ijk}}} \, d\mu^{n-2} \leq 
 \brac{C_K  \sum_{i<j<k} \int 1_{\{\eta_\eps < 1\}} |Z_{ijk}|^2 d\mu^{n-2} }^{1/2} . 
 \]
 Using again that $\eta_\eps \rightarrow 1$ $\H^{n-2}$-almost-everywhere and (\ref{eq:StokesSigma2-dominant2}), we confirm that the right-hand-side tends to $0$ as $\eps \rightarrow 0$, concluding the proof. 
\end{proof}

\begin{lemma}[Divergence theorem on $\Sigma^1$] \label{lem:StokesSigma1}
    Let $\Omega$ be a stationary regular partition in $(M^n,g,\mu)$ with locally bounded curvature. Given $i<j$, let $Z_{ij} = Z_{ji}$ denote a (non-oriented) $C^1$-smooth tangential vector-field on $\Sigma_{ij}$ which is continuous up to $\partial \Sigma_{ij}$. Assume that there exists a compact set $K \subset M^n$ so that all $Z_{ij}$ are supported in $K$ and satisfy:
    \begin{equation} \label{eq:StokesSigma1-dominant1}
    \sum_{i<j} \int_{\Sigma_{ij}} |\div_{\Sigma,\mu} Z_{ij}| d\mu^{n-1} < \infty ~,~
                \sum_{u<v<w} \int_{\Sigma_{uvw}} \sabs[\Big]{\sum_{\{i,j\} \subset \{u,v,w\}}  Z^{\n_{\partial ij}}_{ij}} d\mu^{n-2} < \infty .
     \end{equation}
     In addition, assume that:
     \begin{equation} \label{eq:StokesSigma1-dominant2}
     \sum_{i<j} \int_{\Sigma_{ij}} |Z_{ij}|^2 d\mu^{n-1} < \infty .
     \end{equation}
        Then:
    \[
        \sum_{i<j} \int_{\Sigma_{ij}} \div_{\Sigma,\mu} Z_{ij} \, d\mu^{n-1} = \sum_{u<v<w} \int_{\Sigma_{uvw}} \sum_{\{i,j\} \subset \{u,v,w\}}  Z^{\n_{\partial ij}}_{ij} d\mu^{n-2} . 
    \]
\end{lemma}
\begin{proof}
Identical to the one on $\Sigma^2$. 
\end{proof}

\section{Scalar-fields and integration by parts} \label{sec:scalar-fields}

Let $\Omega$ be a stationary regular partition in $(M^n,g,\mu)$ with locally bounded curvature. 

\begin{definition}[scalar-field]
A tuple $u = (u_{ij})$ of continuous functions defined on $\overline{\Sigma_{ij}} \setminus \Sigma^4$ 
and locally Lipschitz in $\Sigma_{ij}$ is called a scalar-field if the following properties hold:
\begin{itemize}
\item $u$ is oriented, i.e. $u_{ji} = -u_{ij}$ for all $i < j$.
\item $u$ satisfies Dirichlet-Kirchoff boundary conditions: at every triple-point in $\Sigma_{ijk}$, we have $u_{ij} + u_{jk} + u_{ki} = 0$.
\item $u \in H^1_{loc}(\Sigma^1,\mu^{n-1})$, meaning that for any compact set $K \subset M^n$, $\int_{\Sigma^1 \cap K} u_{ij}^2 d\mu^{n-1} < \infty$ and $\int_{\Sigma^1 \cap K} |\nabla_{\Sigma^1} u_{ij}|^2 d\mu^{n-1} < \infty$. 
\end{itemize} 
\end{definition}

\begin{definition}[Smooth scalar-field]
A scalar-field $u = (u_{ij})$ is said to be in $C^k_{loc}(\Sigma^{\leq m})$, $k \in \{0,1,\ldots,\infty\}$ and $m\in \{1,2,3\}$, if for all $i < j$, $u_{ij} \in C^{k}_{loc}(\overline{\Sigma_{ij}} \cap \Sigma^{\leq m})$. 
\end{definition}

Note that by regularity, we must restrict $k \in \{0,1\}$ when $m=3$ since locally around $p \in \Sigma^3$, $\Sigma$ is only $C^{1,\alpha}$ diffeomorphic to $\mathbf{T} \times \R^{n-3}$. 

\begin{remark} \label{rem:physical}
Note that for any smooth and compactly-supported vector-field $X \in C_c^\infty(M)$, $u  = (u_{ij} = X^{\n_{ij}})$ is a scalar-field. Indeed, $u_{ij}$ is continuous on $\overline{\Sigma_{ij}} \setminus \Sigma^4$ since $\n_{ij}$ is continuous there by regularity, $u$ is trivially oriented, and $u$ satisfies Dirichlet-Kirchoff BCs since $\n_{ij} + \n_{jk} + \n_{ki} = 0$ on $\Sigma_{ijk}$ by stationarity. In addition, we trivially have $u_{ij} \in L^\infty(\Sigma_{ij})$, and in addition $u \in C^1_{loc}(\Sigma^1)$ with $\nabla_{\alpha} u_{ij} = g_{\gamma \delta} \nabla_{\alpha} X^\gamma \n^\delta_{ij} + \II^{ij}_{\alpha \beta} X^\beta$. Since curvature is locally bounded and $X$ is compactly supported (say in $K\subset M$), if follows that $|\nabla_{\Sigma^1} u_{ij}|\in L^\infty(\Sigma_{ij})$, and in particular, $u_{ij}$ is Lipschitz in $\Sigma_{ij}$. Since $\mu^{n-1}(\Sigma^1 \cap K) < \infty$, it follows that $u \in H^1_{loc}(\Sigma^1,\mu^{n-1})$. 
\end{remark}

\begin{definition}[Physical scalar-field]
A scalar-field $u = (u_{ij})$ is called physical if there exists a vector-field  $X \in C_c^\infty(M)$ so that $u_{ij} = X^{\n_{ij}}$ for all $i<j$. 
\end{definition}

\begin{definition}[Conformal boundary conditions]
A scalar-field $u \in C^1_{loc}(\Sigma^{\leq 2})$ is said to satisfy conformal boundary conditions (BCs), if for all $u<v<w$ and $p \in \Sigma_{uvw}$,
\[
\nabla_{\n_{\partial ij}} u_{ij} - \bar \II^{\partial ij} u_{ij} =  c_p \;\;\; \forall (i,j) \in \cyclic(u,v,w) ,
\]
where $c_p$ only depends on $p$ and not on $(i,j) \in \cyclic(u,v,w)$.
\end{definition}

\begin{definition}[Admissible scalar-field]
A scalar-field $u = (u_{ij})$ on $\Sigma$ is said to be admissible if the following properties hold:
\begin{itemize}
\item $u \in  C^{\infty}_{loc}(\Sigma^{\leq 2}) \cap C^{1}_{loc}(\Sigma^{\leq 3})$. 
\item $u$ satisfies conformal BCs.
\item There exists a compact set $K \subset M^n$ so that for all $i < j$, $u_{ij}$ is supported in $K$. 
\item For all $i < j$, $u_{ij} \in L^\infty(\Sigma_{ij})$ and $|\nabla_{\Sigma^1} u_{ij}| \in L^\infty(\Sigma_{ij})$. 
\item For all $i < j$, $u_{ij} \in H^2(\Sigma_{ij},\mu^{n-1})$, namely $\int_{\Sigma_{ij}} \snorm{\nabla^2_{\Sigma^1} u_{ij}}^2 d\mu^{n-1}< \infty$. 
\item For all $i < j$, $\Delta_{\Sigma^2,\mu} u_{ij} \in L^1(\partial \Sigma_{ij}, \mu^{n-2})$. 
\end{itemize}
\end{definition}

Note that the last $4$ conditions trivially hold whenever $u_{ij}$ is compactly supported in $\Sigma_{ij} \cup \partial \Sigma_{ij}$, i.e.~away from $\Sigma^{\geq 3}$. While the latter assumption would considerably simplify the various appeals to the divergence theorem on $\Sigma^1$ and $\Sigma^2$ which we employ in this work, it may be too restrictive, and so we make the extra effort of carefully justifying all integration-by-parts for the more general admissible fields.

\subsection{Integration by parts}

\begin{definition}[Non-oriented conformal BCs]
Given a partition in $(M,g,\mu)$, a smooth function $V$ on $M$ is said to satisfy the non-oriented conformal BCs if for all $i,j,k$,
\[
\nabla_{\n_{\partial ij}} V - \bar \II^{\partial ij}_{ijk} \, V =  0 \;\; \text{ on $\Sigma_{ijk}$}.
\]
\end{definition}
The nomenclature stems from the fact that for all $a \in \R^q$, the scalar field $u = (a_{ij} V)$, where $a_{ij} = a_i - a_j$, satisfies the conformal BCs. 

\begin{lemma} \label{lem:byparts}
Let $f \in C^1_{loc}(\Sigma^1)$ be a scalar-field, let $u$ be an admissible scalar-field, and let $V > 0$ be a positive smooth function on $M$ satisfying the non-oriented conformal BCs. Then: 
\begin{equation}
\label{eq:byparts}
\begin{split}
& \int_{\Sigma^1} f_{ij} \brac{\Delta_{\Sigma^1,\mu} u_{ij} - \frac{\Delta_{\Sigma^1,\mu} V}{V} u_{ij}} d\mu^{n-1} \\
& = - \int_{\Sigma^1} \scalar{\nabla_{\Sigma^1} f_{ij} - \frac{\nabla_{\Sigma^1} V}{V} f_{ij} , \nabla_{\Sigma^1} u_{ij} - \frac{\nabla_{\Sigma^1} V}{V} u_{ij}} d\mu^{n-1} . 
\end{split}
\end{equation}
\end{lemma}
\begin{proof}
Define the following (non-oriented) tangential vector-field on $\Sigma_{ij}$, $i \neq j$:
\[
Z_{ij} := f_{ij} \brac{\nabla_{\Sigma^1} u_{ij} -  u_{ij} \frac{\nabla_{\Sigma^1} V}{V}} .
\]
The assertion is equivalent to showing that $\int_{\Sigma^1}\div_{\Sigma^1,\mu} Z_{ij} d\mu^{n-1} = 0$. 
 The conformal BCs of $u$, the non-oriented conformal BCs of $V$, and the Dirichlet-Kirchoff conditions of $f$ ensure that for all $u < v < w$, we have at any $p \in \Sigma_{uvw}$:
\[
\sum_{\{i,j\} \subset \{u,v,w\}} Z_{ij}^{\n_{\partial ij}} = \sum_{(i,j) \in \cyclic (u,v,w)} Z_{ij}^{\n_{\partial ij}} = \sum_{(i,j) \in \cyclic (u,v,w)} f_{ij} (\bar \II^{\partial ij} u_{ij} + c_p - \bar \II^{\partial ij} u_{ij}) = 0 .
\]
To justify the application of the divergence theorem in the form of Lemma \ref{lem:StokesSigma1} and conclude the assertion, it remains to show that 
$\int_{\Sigma^1} |\div_{\Sigma^1,\mu} Z_{ij}| d\mu^{n-1} < \infty$ and $\int_{\Sigma^1} |Z_{ij}|^2 d\mu^{n-1} < \infty$. 
Since $u_{ij}$ are all supported in some compact set $K$, it is enough to show that $\int_{\Sigma^1 \cap K} |\div_{\Sigma^1,\mu} Z_{ij}| d\mu^{n-1} < \infty$, and $\int_{\Sigma^1 \cap K} |Z_{ij}|^2 d\mu^{n-1} < \infty$.

 Note that $V$, $1/V$, $\nabla_{\Sigma^1} V$ and $\nabla_{\Sigma^1} W$ are bounded on $K$. In addition, the assumption that $\Sigma^1$ has locally bounded curvature and the compactness of $K$ imply that $\Delta_{\Sigma^1,\mu} V$ is bounded. Now, using that $u$ is admissible, we conclude that $\nabla_{\Sigma^1} u_{ij} -  u_{ij} \frac{\nabla_{\Sigma^1} V}{V}$ is bounded on $\Sigma^1 \cap K$, and since $f \in L^2(\Sigma^1 \cap K , \mu^{n-1})$, we verify that $\int_{\Sigma^1 \cap K} |Z_{ij}|^2 d\mu^{n-1} < \infty$. 
 
 In addition, since $u$ is admissible, we also know that $\Delta_{\Sigma^1} u_{ij} \in L^2(\Sigma^1 \cap K)$. Since $\scalar{\nabla_{\Sigma^1} u_{ij}, \nabla_{\Sigma^1} W}$ and $\frac{\Delta_{\Sigma^1,\mu} V}{V} u_{ij}$ are bounded on $\Sigma^1 \cap K$ and $\mu^{n-1}(\Sigma^1 \cap K) < \infty$, it follows that these terms are also in $L^2(\Sigma^1 \cap K)$. Since $f \in L^2(\Sigma^1 \cap K , \mu^{n-1})$, the Cauchy-Schwarz inequality implies that the left hand side of (\ref{eq:byparts}) is in $L^1(\Sigma^1 \cap K)$. Using that $\nabla f \in L^2(\Sigma^1 \cap K , \mu^{n-1})$, a similar argument verifies that the right hand side of (\ref{eq:byparts}) is in $L^1(\Sigma^1 \cap K)$ as well. Therefore $\int_{\Sigma^1 \cap K} |\div_{\Sigma^1,\mu} Z_{ij}| d\mu^{n-1} < \infty$, concluding the proof. 
\end{proof}

\begin{corollary} \label{cor:delta1V-zero}
Let $u$ be an admissible scalar-field, and let $V > 0$ be a positive smooth function on $M$ satisfying the non-oriented conformal BCs.
Then $h = (h_{ij})$ with $h_{ij} = V \Delta_{\Sigma^1,\mu} u_{ij} -  u_{ij} \Delta_{\Sigma^1,\mu} V$ satisfies $\delta^1_h \Vol = 0$. 
\end{corollary}
\noindent
Note that $h$ is not necessarily a scalar-field, because there is no guarantee that it satisfies the Dirichlet-Kirchoff BCs. 
\begin{proof}Applying the previous lemma to the scalar-field $f = (a_{ij} V)$, $a_{ij} = a_i - a_j$, the right-hand-side of (\ref{eq:byparts}) 
vanishes, and we deduce that for all $a \in \R^q$:
\[
 \sum_{i=1}^q a_i \sum_{j \neq i} \int_{\Sigma_{ij}} h_{ij} d\mu^{n-1} = \int_{\Sigma^1} a_{ij} h_{ij} d\mu^{n-1} = 0 .
\]
Note that even if $q=\infty$, since $u$ and hence $h$ are compactly supported and the partition is locally finite, the sum over $i$ on the left is always finite. It follows that $\delta^1_h \Vol = 0$. 
\end{proof}

Before concluding this section, let us also verify the following formula for $Q^0(f)$ for physical scalar-fields, which we've already stated and employed in the Introduction. 
\begin{lemma} \label{lem:Q0V}
Let $f = (f_{ij})$ be a physical scalar-field, and let $V>0$ be a positive smooth function on $M$. Then:
\begin{equation} \label{eq:Q0V}
Q^0(f) = \sum_{i<j}  \int_{\Sigma_{ij}} \brac{\abs{\nabla_{\Sigma^1} f_{ij} - \frac{\nabla_{\Sigma^1} V}{V} f_{ij}}^2 - \frac{L_{Jac} V}{V} f_{ij}^2} d\mu^{n-1} + 
\int_{\partial \Sigma_{ij}} \brac{\frac{\nabla_{\n_{\partial ij}} V}{V} - \bar \II^{\partial ij}} f_{ij}^2 d\mu^{n-2} .
\end{equation}
\end{lemma}
\begin{proof}
Contrary to the conceptual presentation in the Introduction, it will actually be much simpler to start from the formula (\ref{eq:Q0nabla}) for $Q^0(f)$ rather than from (\ref{eq:Q0LJac}), since the justification of the integration by parts in the latter case would involve repeating some of the tedious computations from \cite[Section 5]{EMilmanNeeman-TripleAndQuadruple} used to justify the equivalence of (\ref{eq:Q0LJac}) and (\ref{eq:Q0nabla}). 

Define the following (non-oriented) tangential vector-field on $\Sigma_{ij}$, $i \neq j$:
\[
Z_{ij} :=  \frac{\nabla_{\Sigma^1} V}{V} f_{ij}^2 .
\]
Note that:
\[
\div_{\Sigma,\mu} Z_{ij} = 2\scalar{\nabla_{\Sigma^1} f, \frac{\nabla_{\Sigma^1} V}{V}} f_{ij} - \frac{|\nabla_{\Sigma^1} V|^2}{V^2} f_{ij}^2 + \frac{\Delta_{\Sigma^1,\mu} V}{V} f_{ij}^2 ,
\]
so if we could justify the application of the divergence theorem in the form of Lemma \ref{lem:StokesSigma1}, it would follow that:
\[
- \sum_{i<j} \int_{\Sigma_{ij}} \div_{\Sigma,\mu} Z_{ij} d\mu^{n-1} = \sum_{i<j} \int_{\partial \Sigma_{ij}} \frac{\nabla_{\n_{\partial ij}} V}{V} f_{ij}^2 d\mu^{n-2} . 
\]
Summing this with (\ref{eq:Q0nabla}), and using that $\Delta_{\Sigma^1,\mu} V + V L_{Jac} 1 = L_{Jac} V$, the asserted formula (\ref{eq:Q0V}) would immediately follow. 

It remains to check the validity of the assumptions of Lemma \ref{lem:StokesSigma1}. Since $f$ is a physical scalar-field, all $f_{ij}$ and hence $Z_{ij}$ are supported in some compact set $K \subset M$, and so it is enough to show that $\int_{\Sigma^1 \cap K} |\div_{\Sigma^1,\mu} Z_{ij}| d\mu^{n-1} < \infty$, $\int_{\Sigma^1 \cap K} |Z_{ij}|^2 d\mu^{n-1} < \infty$ and $\int_{\Sigma^2 \cap K} |Z_{ij}^{\n_{\partial ij}}| d\mu^{n-2} < \infty$. 

As usual, $1/V$, $\nabla_{\Sigma^1} V$, $\nabla_{\Sigma^1} W$ and $\Delta_{\Sigma^1,\mu} V$ are bounded on $K$ (the latter thanks to $\Sigma^1$ having locally bounded curvature). In view of Remark \ref{rem:physical}, $f \in C^1_{loc}(\Sigma^1) \cap C^0_{loc}(\Sigma^{\leq 2})$, and $f_{ij} \in L^\infty(\Sigma_{ij} \cup \partial \Sigma_{ij})$ and $|\nabla f_{ij}| \in L^\infty(\Sigma_{ij})$. Consequently, $Z_{ij}$ is bounded on $\Sigma_{ij} \cup \partial \Sigma_{ij}$ and $\div_{\Sigma,\mu} Z_{ij}$ is bounded on $\Sigma_{ij}$. Since $\mu^{n-1}(\Sigma^1 \cap K), \mu^{n-2}(\Sigma^2 \cap K) < \infty$, the assumptions of Lemma \ref{lem:StokesSigma1} are verified, concluding the proof. 
\end{proof}

\section{Bochner identity on partitions with conformally flat umbilical boundary} \label{sec:Bochner}

Throughout this section, let $\Omega$ be a stationary regular partition in $(M,g,\mu)$ with locally-bounded curvature, and recall its boundary decomposition $\Sigma = \Sigma^1 \cup \Sigma^2 \cup \Sigma^3 \cup \Sigma^4$. On $\Sigma_{ijk}$ we abbreviate $\n_{\partial ij} = \n^{ijk}_{\partial ij}$ and recall that:
\[\bar \II^{\partial ij}_{ijk} = \frac{\II^{ik}_{\n_{\partial ik}} + \II^{jk}_{\n_{\partial jk}}}{\sqrt{3}}.
\]By regularity, the co-normals $\n^{ijk}_{\partial ij}$ are continuous up to $\Sigma_{ijk\ell}$, and we continue to denote them on $\Sigma_{ijk\ell}$ by $\n^{ijk}_{\partial ij}$. Note that on $\Sigma_{ijk\ell}$, $\n^{ijk}_{\partial ij} \neq \n^{ij\ell}_{\partial ij}$, and hence it is imperative to specify the superscript of $\n_{\partial ij}$, as this determines which of the boundaries $\Sigma_{ijk}$ or $\Sigma_{\ij\ell}$ of $\Sigma_{ij}$, $\n_{\partial ij}$ refers to. 

\bigskip

When specified, we will also add the following additional assumptions:

\begin{itemize}
\item For all distinct $i,j,k$, $\Sigma_{ijk}$ is umbilical in $\overline{\Sigma_{ij}}$ with principle curvatures coinciding with $\bar \II^{\partial ij}_{ijk}$, namely:
\begin{equation} \label{eq:A1}
\II_{\Sigma_{ijk} \subset \partial \Sigma_{ij}} = \bar \II^{\partial ij}_{ijk} \, g_{\Sigma_{ijk}} . 
\end{equation}
\item
For all compact $K \subset M^n$, 
\begin{equation} \label{eq:A1B}
\sum_{u<v<w}  \sum_{\{i,j\} \subset \{u,v,w\}} \int_{\Sigma_{uvw}\cap K} \abs{\nabla_{\Sigma^2} \bar \II^{\partial ij}_{uvw}} d\mu^{n-2} < \infty . 
\end{equation}
\item For all distinct $a,b,c,d$, we have on $\Sigma_{abcd}$ that 
\begin{equation} \label{eq:A2}
\II^{ij}_{\n^{ijk}_{\partial ij}} = \II^{ij}_{\n^{ij\ell}_{\partial ij}} \;\;\; \forall \{i,j,k,\ell\} = \{a,b,c,d\}.
\end{equation} 
\end{itemize}

\begin{definition}[Umbilical boundary]
A stationary regular partition with locally-bounded curvature is said to have umbilical boundary if (\ref{eq:A1}), (\ref{eq:A1B}) and (\ref{eq:A2}) hold. 
\end{definition}

For convenience, we also recall the following definition from the Introduction:
\begin{definition}[Conformally flat boundary]
A regular partition is said to have conformally flat boundary if there exists a smooth positive function $V>0$ on $M$, called the conformally flattening boundary potential, so that for all non-empty $\Sigma_{ijk}$, $\II_{\Sigma_{ijk} \subset \partial \Sigma_{ij}} - \nabla_{\n_{\partial ij}} \log V \, g_{\Sigma_{ijk}} \equiv 0$. 
\end{definition}
Note that under assumption (\ref{eq:A1}), a positive smooth function $V>0$ is a conformally flattening boundary potential if and only if $V$ satisfies non-oriented conformal BCs, namely $\nabla_{\n_{\partial ij}} V - \bar \II^{\partial ij}_{ijk} \, V \equiv 0$ on $\Sigma_{ijk}$. 
 Consequently, we will not make a distinction between these two properties. 
 
 \begin{lemma} \label{lem:constant-curvature}
A stationary regular partition in $(M,g,\mu)$ with constant-curvature interfaces $\Sigma_{ij}$ (namely $\II^{ij} = \k_{ij} g_{\Sigma_{ij}}$) has umbilical boundary. In particular, a positive smooth function $V>0$ is a conformally flattening boundary potential if and only if $V$ satisfies non-oriented conformal BCs.
\end{lemma}
\begin{proof}
In that case, $\bar \II^{\partial ij}_{ijk}$ is constant, so  (\ref{eq:A1B}) holds trivially, and so does (\ref{eq:A2}). To see (\ref{eq:A1}), recall that for a stationary partition $\n_{\partial ij} = \frac{\n_{ik} + \n_{jk}}{\sqrt{3}}$, and hence for any unit-vector $v \in T \Sigma_{ijk}$, 
\[
\II_{\Sigma_{ijk} \subset \partial \Sigma_{ij}}(v) = \scalar{\nabla_v \n_{\partial ij},v} = \frac{\II^{ik}(v) + \II^{jk}(v)}{\sqrt{3}} = 
\frac{\II^{ik}(\n_{\partial ik}) + \II^{jk}(\n_{\partial jk})}{\sqrt{3}} = \bar \II^{\partial ij} .
\]
\end{proof}

\medskip

Our goal in this section is to prove the following integrated Bochner identity on clusters with umbilical boundary which is conformally flat. On a compact manifold $(M,g)$ without boundary, when $\mu = vol_g$ and $V=1$, this formula is classical (see e.g. \cite[Chapter 9]{PetersenBook3rdEd}). When $M$ has a boundary, an extension wherein a boundary integral appears was derived by Reilly \cite{ReillyOriginalFormula}. Extensions to the weighted setting for general $\mu$ were obtained in \cite{MaDuGeneralizedReilly,KolesnikovEMilmanReillyPart1}. An extension for a general potential $V$ in the unweighted setting was initiated by Qiu--Xia \cite{QiuXia-GenReillyFormula} and Li--Xia \cite{LiXia-GenReillyFormulaSubStatic}, and extended to the weighted setting by Huang--Ma--Zhu \cite{HuangZhu-ConjugatedBrascampLieb,HuangMaZhu-ReillyFormula}. In out setting, the boundary $\partial \Sigma_{ij} $ is not a closed manifold, and so divergences do not integrate to zero in general. Moreover, our test functions are not compactly supported in $\Sigma_{ij} \cup \partial \Sigma_{ij}$. Consequently, our contribution is in verifying that all of the integrations-by-parts required for the proof are justified, and that the boundary terms ultimately cancel under our assumptions. 

\begin{proposition}[Bochner identity on partitions with conformally flat umbilical boundary] \label{prop:main-computation}
Let $\Omega$ be a stationary regular partition with locally-bounded curvature and conformally flat umbilical boundary. Let $V>0$ be a conformally flattening boundary potential, and let $u = (u_{ij})$ be an admissible scalar field. Then:
\begin{equation} \label{eq:Bochner}
\begin{split}
&  \sum_{i<j} \int_{\Sigma_{ij}}   V \brac{ \Delta_{\Sigma^1, \mu} u_{ij}   - \frac{\Delta_{\Sigma^1, \mu} V}{V} u_{ij} }^2  d\mu^{n-1} \\
 & = \sum_{i<j} \int_{\Sigma_{ij}}  V \left ( \norm{\nabla^2_{\Sigma^1} u_{ij} - \frac{\nabla^2_{\Sigma^1} V}{V} u_{ij} }^2 + \hRic^V_{\Sigma, \mu} \brac{\nabla_{\Sigma^1} u_{ij}- \frac{\nabla_{\Sigma^1} V}{V} u_{ij}  } \right ) d\mu^{n-1} .
\end{split}
\end{equation}
\end{proposition}

To this end, we will introduce a couple of useful vector-fields on $\Sigma^1$ and $\Sigma^2$; we commence with $\Sigma^1$. 
Given $i \neq j$, let $Z_{ij}$ denote the following (non-oriented) tangential vector-field on $\Sigma_{ij}$ (extending continuously to $\partial \Sigma_{ij}$ by partition regularity and adminissibility of $u$):
\begin{align}
\label{eq:Zij}    Z_{ij} =& \frac{V}{2} \nabla_{\Sigma^1} |\nabla_{\Sigma_1} u_{ij}|^2 -V \Delta_{\Sigma, \mu} u_{ij} \nabla_{\Sigma^1} u_{ij} - 2 u \nabla^t_{\nabla_{\Sigma^1} V} \nabla_{\Sigma^1} u_{ij} + |\nabla_{\Sigma^1} u_{ij}|^2 \nabla_{\Sigma^1} V \\
\nonumber		 &+2u_{ij} \Delta_{\Sigma^1, \mu} u_{ij} \nabla_{\Sigma^1} V + \frac{u^2_{ij}}{2V} \nabla_{\Sigma^1} |\nabla_{\Sigma^1} V|^2 - \frac{u^2_{ij}}{V} \Delta_{\Sigma^1, \mu} V \nabla_{\Sigma^1} V - \scalar{\nabla_{\Sigma^1} u_{ij},\nabla_{\Sigma^1} V} \nabla_{\Sigma^1} u_{ij} .
\end{align}

\begin{lemma} \label{lem:divZij}
\begin{align*}
\div_{\Sigma^1,\mu} Z_{ij} & = V \norm{\nabla^2_{\Sigma^1} u_{ij} - \frac{\nabla^2_{\Sigma^1} V}{V} u_{ij} }^2 \\
& + V \hRic^V_{\Sigma, \mu} \brac{\nabla_{\Sigma^1} u_{ij}- \frac{\nabla_{\Sigma^1} V}{V} u_{ij},  \nabla_{\Sigma^1} u_{ij} - \frac{\nabla_{\Sigma^1} V}{V} u_{ij}} \\
 & - V \brac{ \Delta_{\Sigma^1, \mu} u_{ij}  - \frac{\Delta_{\Sigma^1, \mu} V}{V} u_{ij} }^2 . 
\end{align*}
\end{lemma}
This identity was implicitly established in \cite[Formula (2.6)]{HuangMaZhu-ReillyFormula}; for completeness, we verify the computation.
\begin{proof}
The weighted Bochner--Lichnerowicz--Weitzenb\"ock formula \cite{MaDuGeneralizedReilly,KolesnikovEMilmanReillyPart1} states that for any $h \in C_{loc}^3 (\Sigma^1)$, we have
\begin{equation}\label{fml-Bochner-classical}
	\frac{1}{2} \Delta_{\Sigma^1, \mu} |\nabla_{\Sigma^1} h|^2
	= \norm{\nabla_{\Sigma^1}^2 h}^2 + \scalar{\nabla_{\Sigma^1} \Delta_{\Sigma^1, \mu} h, \nabla_{\Sigma^1}h} + \scalar{\Ric_{\Sigma^1, \mu} \nabla_{\Sigma^1} h, \nabla_{\Sigma^1} h}.
\end{equation}
Using this, the divergence of the first two terms of (\ref{eq:Zij}) minus half the 4th term is given by
\begin{align*}
	 \div_{\Sigma^1, \mu} & \brac{ \frac{V}{2} \nabla_{\Sigma^1}|\nabla_{\Sigma^1} u_{ij}|^2 -V \Delta_{\Sigma^1, \mu} u_{ij} \nabla_{\Sigma^1} u_{ij} -\frac{1}{2} |\nabla_{\Sigma^1} u_{ij}|^2 \nabla_{\Sigma^1}V } \\
			= \; & V \brac{ \frac{1}{2}\Delta_{\Sigma^1, \mu} |\nabla_{\Sigma^1} u_{ij}|^2 -\scalar{\nabla_{\Sigma^1} \Delta_{\Sigma^1, \mu} u_{ij}, \nabla_{\Sigma^1} u_{ij}} } -\frac{1}{2} |\nabla_{\Sigma^1} u_{ij}|^2 \Delta_{\Sigma^1, \mu} V \\
	&-\scalar{\nabla_{\Sigma^1} V, \nabla_{\Sigma^1} u_{ij}} \Delta_{\Sigma^1,\mu} u_{ij} -V \brac{\Delta_{\Sigma^1, \mu} u_{ij}}^2 \\
	= \; &V \norm{\nabla_{\Sigma^1}^2 u_{ij}}^2 + \scalar{V\Ric_{\Sigma^1, \mu}\nabla_{\Sigma^1} u_{ij}, \nabla_{\Sigma^1} u_{ij}} -\frac{1}{2} |\nabla_{\Sigma^1} u_{ij}|^2 \Delta_{\Sigma^1, \mu} V \\
	&-\scalar{\nabla_{\Sigma^1} V, \nabla_{\Sigma^1} u_{ij}} \Delta_{\Sigma^1,\mu} u_{ij} -V \brac{\Delta_{\Sigma^1, \mu} u_{ij}}^2.
\end{align*}
On the other hand, it follows from the Ricci identity that
\[
\Delta_{\Sigma^1} \nabla_{\Sigma^1} u_{ij} -\nabla_{\Sigma^1}^2 u_{ij} \nabla_{\Sigma^1} W
= \nabla_{\Sigma^1} \Delta_{\Sigma^1, \mu} u_{ij} + \Ric_{\Sigma^1, \mu} \nabla_{\Sigma^1} u_{ij}.
\]
Therefore, the divergence of the 3rd, 4th and 5th terms in (\ref{eq:Zij}) is
\begin{align*}
	&\div_{\Sigma^1, \mu} \brac{-2u_{ij} \scalar{\nabla_{\Sigma^1}^2 u_{ij}, \nabla_{\Sigma^1} V}+ |\nabla_{\Sigma^1} u_{ij}|^2 \nabla_{\Sigma^1} V + 2u_{ij} \Delta_{\Sigma^1, \mu} u_{ij} \nabla_{\Sigma^1} V}\\
				=& -2 u_{ij} \scalar{\Delta_{\Sigma^1} \nabla_{\Sigma^1} u_{ij}- \nabla_{\Sigma^1}^2 u_{ij} \nabla_{\Sigma^1} W, \nabla_{\Sigma^1} V} - 2 u_{ij} \scalar{\nabla_{\Sigma^1}^2 u_{ij}, \nabla_{\Sigma^1}^2 V}\\
	&+\abs{ \nabla_{\Sigma^1} u_{ij}}^2 \Delta_{\Sigma^1, \mu} V +2 \scalar{\nabla_{\Sigma^1} u_{ij}, \nabla_{\Sigma^1} V}\Delta_{\Sigma^1, \mu} u_{ij} + 2u_{ij} \scalar{\nabla_{\Sigma^1} \Delta_{\Sigma^1, \mu} u_{ij}, \nabla_{\Sigma^1} V} \\
	&+ 2u_{ij} \Delta_{\Sigma^1, \mu} u_{ij} \Delta_{\Sigma^1, \mu} V \\
	=& -2 u_{ij} \scalar{\Ric_{\Sigma^1, \mu} \nabla_{\Sigma^1} u_{ij}, \nabla_{\Sigma^1} V} - 2 u_{ij} \scalar{\nabla_{\Sigma^1}^2 u_{ij}, \nabla_{\Sigma^1}^2 V}+\abs{ \nabla_{\Sigma^1} u_{ij}}^2 \Delta_{\Sigma^1, \mu} V \\
	&+2 \scalar{\nabla_{\Sigma^1} u_{ij}, \nabla_{\Sigma^1} V}\Delta_{\Sigma^1, \mu} u_{ij}  
	+ 2u_{ij} \Delta_{\Sigma^1, \mu} u_{ij} \Delta_{\Sigma^1, \mu} V. 
\end{align*}
Finally, using formula (\ref{fml-Bochner-classical}) again, the divergence of the last 3 terms of (\ref{eq:Zij}) and half the 4th term is
\begin{align*}
	&\div_{\Sigma^1, \mu} \brac{\frac{u_{ij}^2}{2V} \nabla_{\Sigma^1} \abs{\nabla_{\Sigma^1} V}^2 -\frac{u_{ij}^2}{V} \Delta_{\Sigma^1, \mu} V \nabla_{\Sigma^1} V -\scalar{\nabla_{\Sigma^1} u_{ij}, \nabla_{\Sigma^1} V} \nabla_{\Sigma^1} u_{ij} +\frac{1}{2} \abs{\nabla_{\Sigma^1} u_{ij}}^2 \nabla_{\Sigma^1} V }\\
	=& \frac{2u_{ij}}{V} \scalar{\nabla_{\Sigma^1}^2 V \nabla_{\Sigma^1} V, \nabla_{\Sigma^1} u_{ij}} -\frac{u_{ij}^2}{V^2} \scalar{\nabla_{\Sigma^1}^2 V \nabla_{\Sigma^1} V, \nabla_{\Sigma^1} V} \\
	&+ \frac{u_{ij}^2}{V} \brac{\frac{1}{2} \Delta_{\Sigma^1, \mu} \abs{\nabla_{\Sigma^1} V}^2 -\scalar{\nabla_{\Sigma^1} \Delta_{\Sigma^1, \mu} V, \nabla_{\Sigma^1} V} }\\
	&-\frac{2 u_{ij}}{V} \scalar{\nabla_{\Sigma^1} u_{ij}, \nabla_{\Sigma^1} V} \Delta_{\Sigma^1, \mu} V +\frac{u_{ij}^2}{V^2} \abs{\nabla_{\Sigma^1} V}^2 \Delta_{\Sigma^1, \mu} V -\frac{u_{ij}^2}{V} \brac{\Delta_{\Sigma^1, \mu} V}^2\\
	&- \scalar{\nabla_{\Sigma^1}^2 u_{ij} \nabla_{\Sigma^1} u_{ij}, \nabla_{\Sigma^1} V}- \scalar{\nabla_{\Sigma^1}^2 V \nabla_{\Sigma^1} u_{ij}, \nabla_{\Sigma^1} u_{ij}} -\scalar{\nabla_{\Sigma^1} u_{ij}, \nabla_{\Sigma^1} V} \Delta_{\Sigma^1, \mu} u_{ij}\\
	&+ \scalar{\nabla_{\Sigma}^2 u_{ij} \nabla_{\Sigma^1} u_{ij}, \nabla_{\Sigma^1} V} + \abs{\nabla_{\Sigma^1} u_{ij}}^2 \Delta_{\Sigma^1, \mu} V\\
	=&\brac{\Delta_{\Sigma^1, \mu} V g_{\Sigma^1}- \nabla_{\Sigma^1}^2 V} \brac{\nabla_{\Sigma^1} u_{ij} -\frac{\nabla_{\Sigma^1} V}{V} u_{ij},  \nabla_{\Sigma^1} u_{ij} -\frac{\nabla_{\Sigma^1} V}{V} u_{ij}} 
	-\frac{1}{2} \abs{ \nabla_{\Sigma^1} u_{ij}}^2 \Delta_{\Sigma^1, \mu} V\\
	&+\frac{u_{ij}^2}{V} \norm{\nabla_{\Sigma^1}^2 V}^2+V \Ric_{\Sigma^1,\mu} \brac{\frac{\nabla_{\Sigma^1} V}{V} u_{ij}, \frac{\nabla_{\Sigma^1} V}{V} u_{ij}} -V \brac{\frac{\Delta_{\Sigma^1, \mu} V}{V} u_{ij} }^2\\
	&- \scalar{\nabla_{\Sigma^1} u_{ij}, \nabla_{\Sigma^1} V} \Delta_{\Sigma^1, \mu} u_{ij}.
\end{align*}
Combining the above identities, we obtain
\begin{align*}
	\div_{\Sigma^1, \mu} Z_{ij}
	=&V \brac{\norm{\nabla_{\Sigma^1}^2 u_{ij}}^2 - 2 \scalar{\nabla_{\Sigma^1}^2 u_{ij}, \frac{\nabla_{\Sigma^1}^2 V }{V} u_{ij}} + \frac{\norm{\nabla_{\Sigma^1}^2  V}^2 }{V^2} u_{ij}^2 }\\
	&+ V \Ric_{\Sigma^1, \mu} \brac{\nabla_{\Sigma^1} u_{ij}, \nabla_{\Sigma^1} u_{ij}} - 2 V \Ric_{\Sigma^1, \mu} \brac{\nabla_{\Sigma^1} u_{ij}, \frac{\nabla_{\Sigma^1} V}{V} u_{ij} } \\
	&+ V \Ric_{\Sigma^1, \mu} \brac{\frac{\nabla_{\Sigma^1} V}{V} u_{ij}, \frac{\nabla_{\Sigma^1} V}{V} u_{ij}}\\
	&- V \brac{ \brac{\Delta_{\Sigma^1, \mu} u_{ij}}^2 -2 \Delta_{\Sigma^1, \mu} u_{ij} \frac{\Delta_{\Sigma^1, \mu} V}{V} u_{ij} +\brac{\frac{\Delta_{\Sigma^1, \mu} V}{V} u_{ij}}^2 }\\
	&+\brac{\Delta_{\Sigma^1, \mu} V g_{\Sigma^1}- \nabla_{\Sigma^1}^2 V} \brac{\nabla_{\Sigma^1} u_{ij} -\frac{\nabla_{\Sigma^1} V}{V} u_{ij},  \nabla_{\Sigma^1} u_{ij} -\frac{\nabla_{\Sigma^1} V}{V} u_{ij}} \\
	=& V \norm{\nabla_{\Sigma^1}^2 u_{ij} -\frac{\nabla_{\Sigma^1}^2 V}{V} u_{ij}}^2 + V \hRic^V_{\Sigma^1, \mu} \brac{\nabla_{\Sigma^1} u_{ij} -\frac{\nabla_{\Sigma^1} V}{V} u_{ij},  \nabla_{\Sigma^1} u_{ij} -\frac{\nabla_{\Sigma^1} V}{V} u_{ij}}\\
	&- V \brac{\Delta_{\Sigma^1, \mu} u_{ij} -\frac{\Delta_{\Sigma^1, \mu} V}{V} u_{ij}}^2.
\end{align*}
This completes the proof.
\end{proof}

Consequently, Proposition \ref{prop:main-computation} reduces to showing that:
\[
\int_{\Sigma^1} \div_{\Sigma^1,\mu} Z_{ij} d\mu^{n-1} = 0 . 
\]
To this end, given $u < v < w$, let $Z_{uvw}$ denote the following tangential vector-field on $\Sigma_{uvw}$ (extending continuously to $\partial \Sigma_{uvw}$ by partition regularity and admissibility of $u$):
\begin{equation} \label{eq:Zuvw}
Z_{uvw} = \sum_{\{i,j\} \subset \{u,v,w\}} \bar \II^{\partial ij}_{uvw} u_{ij} (V \nabla_{\Sigma^2} u_{ij} - u_{ij} \nabla_{\Sigma^2} V).
\end{equation}
Our main task will be to establish:
\begin{lemma} \label{lem:crazyZ}
Assume (\ref{eq:A1}). Then:
\begin{enumerate}
\item For all $u < v < w$, we have on $\Sigma_{uvw}$:
\[
\sum_{\{i,j\} \subset \{u,v,w\}}  Z^{\n_{\partial ij}}_{ij} = \div_{\Sigma^2,\mu} Z_{uvw} .
\]
\end{enumerate}
In addition, assume (\ref{eq:A1B}). Then:
\begin{enumerate}
\setcounter{enumi}{1}
\item  $\{ Z_{ij} \}$ satisfy the assumptions of Lemma \ref{lem:StokesSigma1}. 
\end{enumerate}
In addition, assume (\ref{eq:A2}). Then:
\begin{enumerate}
\setcounter{enumi}{2}
\item For all $a < b < c < d$, we have on $\Sigma_{abcd}$:
\begin{equation} \label{eq:Zuvw-cancelation-again}
\sum_{\{u,v,w\} \subset \{a,b,c,d\}}  Z_{uvw}^{\n_{\partial uvw}} = 0 .
\end{equation}
\item $\{ Z_{uvw} \}$ satisfy the assumptions of Lemma \ref{lem:StokesSigma2}.
\end{enumerate}
\end{lemma}
\begin{corollary}
\[
\int_{\Sigma^1} \div_{\Sigma^1,\mu} Z_{ij} d\mu^{n-1} = 0 .
\]
\end{corollary}
\begin{proof}
By Lemma \ref{lem:crazyZ}, we may apply Lemma \ref{lem:StokesSigma1} to $\{Z_{ij}\}$ followed by  Lemma \ref{lem:StokesSigma2} to $\{Z_{uvw} \}$, yielding:
\begin{align*}
& \int_{\Sigma^1} \div_{\Sigma^1,\mu} Z_{ij} d\mu^{n-1} = 
\int_{\Sigma^2}  \sum_{\{i,j\} \subset \{u,v,w\}}  Z^{\n_{\partial ij}}_{ij} \, d\mu^{n-2} =
 \int_{\Sigma^2}  \div_{\Sigma^2,\mu} Z_{uvw} \, d\mu^{n-2} \\
 & = \int_{\Sigma^3} \sum_{\{u,v,w\} \subset \{a,b,c,d\}}  Z_{uvw}^{\n_{\partial uvw}} \, d\mu^{n-3} = 0 ,
\end{align*}
where the last term vanishes thanks to (\ref{eq:Zuvw-cancelation-again}). 
\end{proof}

The proof of Proposition \ref{prop:main-computation} thus reduces to establishing Lemma \ref{lem:crazyZ}.

\subsection{Boundary computation on $\Sigma^2$}

\begin{lemma}
On $\partial \Sigma_{ij}$ we have: 
\begin{align}
Z^{\n_{\partial ij}}_{ij} = 
		& -V \brac{H_{\partial \Sigma_{ij}} - \nabla_{\n_{\partial ij}}W} \brac{ \nabla_{\n_{\partial ij}} u_{ij}  - \frac{\nabla_{\n_{\partial ij}} V}{V} u_{ij} }^2  \label{eq:Zij-n-partial} \\
		& - V \brac{\II_{\partial \Sigma_{ij}} - \frac{\nabla_{\n_{\partial ij}}V }{V}g_{\partial \Sigma_{ij}}} \brac{\nabla_{\Sigma^2} u_{ij}- \frac{\nabla_{\Sigma^2} V}{V}u_{ij}, \nabla_{\Sigma^2} u_{ij}- \frac{\nabla_{\Sigma^2} V}{V} u_{ij} }  \nonumber\\
		&- 2 V \brac{\nabla_{\n_{\partial ij}} u_{ij} -\frac{\nabla_{\n_{\partial ij}} V}{V} u_{ij}} \brac{\Delta_{\Sigma^2, \mu} u_{ij} -\frac{\Delta_{\Sigma^2, \mu} V}{V} u_{ij} } \nonumber \\
		& + \div_{\Sigma^2, \mu} \brac{ V  \brac{\nabla_{\n_{\partial ij}} u_{ij} -\frac{\nabla_{\n_{\partial ij}} V}{V} u_{ij}}\brac{\nabla_{\Sigma^2} u_{ij}- \frac{\nabla_{\Sigma^2} V}{V} u_{ij}} } \nonumber \\
		&+\div_{\Sigma^2, \mu} \brac{ u_{ij} \nabla_{\n_{\partial ij}}V  \nabla_{\Sigma^2} u_{ij} - u_{ij} \nabla_{\n_{\partial ij}}u_{ij} \nabla_{\Sigma^2} V }.  \nonumber
\end{align}
\end{lemma}
\begin{proof}
Note that
\[
\left. \nabla_{\Sigma^1}^2 u_{ij} \brac{\n_{\partial ij}, \cdot} \right|_{T \partial \Sigma_{ij}}
= \nabla_{\Sigma^2} \brac{\nabla_{\n_{\partial ij}} u_{ij}}
- \II_{\partial \Sigma_{ij}} \brac{\nabla_{\Sigma^2} u_{ij}, \cdot}.
\]
Then
\begin{align*}
	\frac{V}{2} \nabla_{\n_{\partial ij}} \abs{\nabla_{\Sigma^1} u_{ij}}^2
	=& V \nabla_{\n_{\partial ij}} u_{ij}\nabla_{\Sigma^1}^2 u_{ij} \brac{\n_{\partial ij}, \n_{\partial ij}}\\
	&+ V \scalar{\nabla_{\Sigma^2} \brac{\nabla_{\n_{\partial ij}} u_{ij}} - \II_{\partial \Sigma_{ij}} \nabla_{\Sigma^2} u_{ij}, \nabla_{\Sigma^2} u_{ij}}.
\end{align*}
Similarly, it holds that
\begin{align*}
	\frac{u_{ij}^2}{2V} \nabla_{\n_{\partial ij}} \abs{\nabla_{\Sigma^1} V}^2 
	=& \frac{u_{ij}^2}{V} \nabla_{\n_{\partial ij}} V \nabla_{\Sigma^1}^2 V \brac{\n_{\partial ij}, \n_{\partial ij}} \\
	&+ \frac{u_{ij}^2}{V} \scalar{\nabla_{\Sigma^2} \brac{\nabla_{\n_{\partial ij}} V}  - \II_{\partial \Sigma_{ij}} \nabla_{\Sigma^2} V, \nabla_{\Sigma^2} V} .
\end{align*}
Denote $H_{\partial \Sigma_{ij}, \mu} = \tr \brac{\II_{\partial \Sigma_{ij}}} - \nabla_{\n_{\partial ij}} W$. Then on $\partial \Sigma_{ij}$, we have
\begin{align*}
	\Delta_{\Sigma^1, \mu} u_{ij} =& \Delta_{\Sigma^2, \mu} u_{ij} + H_{\partial \Sigma_{ij}, \mu} \nabla_{\n_{\partial ij}} u_{ij} + \nabla_{\Sigma^1}^2 u_{ij} \brac{\n_{\partial ij}, \n_{\partial ij} } , \\
	\Delta_{\Sigma^1, \mu} V =& \Delta_{\Sigma^2, \mu} V + H_{\partial \Sigma_{ij}, \mu} \nabla_{\n_{\partial ij}} V + \nabla_{\Sigma^1}^2 V \brac{\n_{\partial ij}, \n_{\partial ij} } .
\end{align*}
Plugging everything into the left-hand side of (\ref{eq:Zij-n-partial}), we obtain
\begin{align}
	Z_{ij}^{\n_{\partial ij}} =&
	V \nabla_{\n_{\partial ij}} u_{ij}
	+ V \scalar{ \nabla_{\Sigma^2} \brac{\nabla_{\n_{\partial ij}} u_{ij}}, \nabla_{\Sigma^2} u_{ij}}
	-V \II_{ \partial \Sigma_{ij} } \brac{\nabla_{\Sigma^2} u_{ij}, \nabla_{\Sigma^2} u_{ij} } \nonumber\\
	&- V \Delta_{\Sigma^2, \mu} u_{ij} \nabla_{\n_{\partial ij}} u_{ij} - H_{\partial \Sigma_{ij}, \mu} V \brac{\nabla_{\n_{\partial ij}} u_{ij}}^2
	- V \nabla_{\n_{\partial ij}} u_{ij} \nabla_{\Sigma^1}^2 u_{ij}\brac{\n_{\partial ij}, \n_{\partial ij}} \nonumber\\
	&- 2u_{ij} \nabla_{\Sigma^1}^2 u_{ij} \brac{ \n_{\partial ij}, \n_{\partial ij} } \nabla_{\n_{\partial ij}} V
	-2u \scalar{\nabla_{\Sigma^2} \brac{\nabla_{\n_{\partial ij}} u_{ij}} - \II_{\partial \Sigma_{ij}} \nabla_{\Sigma^2} u_{ij}, \nabla_{\Sigma^2} V} \nonumber\\
	&+ \brac{\nabla_{\n_{\partial ij}} u_{ij}}^2 \nabla_{\n_{\partial ij}} V + \abs{\nabla_{\Sigma^2} u_{ij}}^2 \nabla_{\n_{\partial ij}} V \nonumber\\
	&+ 2u_{ij} \nabla_{\n_{\partial ij}} V\Delta_{\Sigma^2, \mu} u_{ij} + 2u_{ij} \nabla_{\n_{\partial ij}} V H_{\partial \Sigma_{ij}, \mu} \nabla_{\n_{\partial ij}} u_{ij} \nonumber\\
	&+ 2u_{ij}\nabla_{\n_{\partial ij}} V\nabla_{\Sigma^1}^2 u_{ij} \brac{\n_{\partial ij}, \n_{\partial ij}} + \frac{u_{ij}^2}{V} \nabla_{\n_{\partial ij}} V \nabla_{\Sigma^1}^2 V \brac{\n_{\partial ij}, \n_{\partial ij}} \nonumber\\
	&+ \frac{u_{ij}^2}{V} \scalar{\nabla_{\Sigma^2} \brac{\nabla_{\n_{\partial ij}} V} -\II_{\partial \Sigma_{ij}} \nabla_{\Sigma^2} V, \nabla_{\Sigma^2} V} -\frac{u_{ij}^2}{V} \nabla_{\n_{\partial ij}} V \Delta_{\Sigma^2, \mu} V \nonumber\\
	&-\frac{u_{ij}^2}{V} \brac{\nabla_{\n_{\partial ij}} V}^2 H_{\partial \Sigma_{ij}, \mu} -\frac{u_{ij}^2}{V} \nabla_{\n_{\partial ij}} V \nabla_{\Sigma^1}^2 V \brac{\n_{\partial ij}, \n_{\partial ij}} \nonumber\\
	&- \brac{\nabla_{\n_{\partial ij}} u_{ij}}^2 \nabla_{\n_{\partial ij}} V - \scalar{\nabla_{\Sigma^2} u_{ij}, \nabla_{\Sigma^2} V} \nabla_{\n_{\partial ij}} u_{ij}. \label{eq:LHS-bdry terms}
\end{align}
We now calculate the right-hand side of (\ref{eq:Zij-n-partial}). For the first three terms, we have
\begin{align}
	&-V H_{\partial \Sigma_{ij}, \mu} 
    \brac{\nabla_{\n_{\partial ij}}u_{ij} - \frac{\nabla_{\n_{\partial ij}} V}{V} u_{ij} }^2 \nonumber\\ 
	&- V \brac{\II_{\partial\Sigma_{ij}} -\frac{\nabla_{\n_{\partial ij}} V}{V} g_{\Sigma^2} } \brac{\nabla_{\Sigma^2} u_{ij} -\frac{\nabla_{\Sigma^2} V}{V} u_{ij},  \nabla_{\Sigma^2} u_{ij} -\frac{\nabla_{\Sigma^2} V}{V} u_{ij}} \nonumber\\
	&- 2 V \brac{\nabla_{\n_{\partial ij}}u_{ij} - \frac{\nabla_{\n_{\partial ij}} V}{V} u_{ij}}
    \brac{\Delta_{\Sigma^2, \mu} u_{ij} -\frac{\Delta_{\Sigma^2, \mu} V}{V} u_{ij}} \nonumber\\
	=& - H_{\partial \Sigma_{ij}, \mu} V \brac{\nabla_{\n_{\partial ij}} u_{ij}}^2 + 2u_{ij} \nabla_{\n_{\partial ij}} V H_{\partial \Sigma_{ij}, \mu} \nabla_{\n_{\partial ij}} u_{ij}\nonumber\\
	&- \frac{u_{ij}^2}{V} \brac{\nabla_{\n_{\partial ij}} V}^2 H_{\partial \Sigma_{ij}, \mu} - V \II_{\partial \Sigma_{ij}} \brac{\nabla_{\Sigma^2} u_{ij}, \nabla_{\Sigma^2} u_{ij}}+ 2 \II_{\partial \Sigma_{ij}} \brac{\nabla_{\Sigma^2} u_{ij}, \nabla_{\Sigma^2} V} \nonumber\\
	&- \frac{u_{ij}^2}{V} \II_{\partial \Sigma_{ij}} \brac{\nabla_{\Sigma^2} V, \nabla_{\Sigma^2} V} + \abs{\nabla_{\Sigma^2} u_{ij}}^2 \nabla_{\n_{\partial ij}} V -\frac{2 u_{ij}}{V} \nabla_{\n_{\partial ij}} V \scalar{\nabla_{\Sigma^2} u_{ij}, \nabla_{\Sigma^2} V} \nonumber\\
	&+ \frac{u_{ij}^2}{V^2} \nabla_{\n_{\partial ij}} V \abs{\nabla_{\Sigma^2} V}^2 - 2V \Delta_{\Sigma^2, \mu} u_{ij} \nabla_{\n_{\partial ij}} u_{ij} + 2u_{ij} \nabla_{\n_{\partial ij}} u_{ij} \Delta_{\Sigma^2, \mu} V \nonumber\\
	&+ 2u_{ij} \nabla_{\n_{\partial ij}} V \Delta_{\Sigma^2, \mu} u_{ij} - 2 \frac{u_{ij}}{V} \nabla_{\n_{\partial ij}} V \Delta_{\Sigma^2, \mu} V. \label{eq:RHS-bdry terms-1}
\end{align}
For the other terms, we have
\begin{align}
 &  &
  \div_{\Sigma^2, \mu} & \bigg[  V 
    \brac{\nabla_{\n_{\partial ij}} u_{ij} -\frac{\nabla_{\n_{\partial ij}} V}{V} u_{ij} } \brac{\nabla_{\Sigma^2} u_{ij}-\frac{\nabla_{\Sigma^2} V}{V} u_{ij} } \nonumber  \\
&  &     & + u_{ij} \nabla_{\n_{\partial ij}} V \nabla_{\Sigma^2} u_{ij} -u_{ij} \nabla_{\n_{\partial ij}} u_{ij} \nabla_{\Sigma^2} V \bigg] 
  \nonumber\\
	& = &  \div_{\Sigma^2, \mu} & \brac{\frac{u_{ij}^2}{V} \nabla_{\n_{\partial ij}} V \nabla_{\Sigma^2} V + V \nabla_{\n_{\partial ij}} u_{ij}\nabla_{\Sigma^2} u_{ij} - 2 u_{ij} \nabla_{\n_{\partial ij}} u_{ij} \nabla_{\Sigma^2} V} \nonumber\\
	&  =  &  \frac{2u_{ij}}{V} \nabla_{\n_{\partial ij}} &   V \scalar{\nabla_{\Sigma^2} u_{ij}, \nabla_{\Sigma^2} V} -\frac{u_{ij}^2}{V^2} \nabla_{\n_{\partial ij}} V \abs{\nabla_{\Sigma^2} V}^2 + \frac{u_{ij}^2}{V} \scalar{\nabla_{\Sigma^2} \brac{\nabla_{\n_{\partial ij}} V} , \nabla_{\Sigma^2} V} \nonumber\\
      & &  &   +\frac{u_{ij}^2}{V}   \nabla_{\n_{\partial ij}}  V \Delta_{\Sigma^2, \mu} V + V \scalar{\nabla_{\Sigma^2} \brac{\nabla_{\n_{\partial ij}} u_{ij}}, \nabla_{\Sigma^2} u_{ij}} + V \Delta_{\Sigma^2, \mu} u_{ij} \nabla_{\n_{\partial ij}} u_{ij} \nonumber\\
     &   &  & -\scalar{\nabla_{\Sigma^2} u_{ij}, \nabla_{\Sigma^2} V}  \nabla_{\n_{\partial ij}} u_{ij} - 2u_{ij} \scalar{\nabla_{\Sigma^2} \brac{\nabla_{\n_{\partial ij}} u_{ij}} , \nabla_{\Sigma^2} V} \nonumber\\
  &  &   & -2u_{ij} \nabla_{\n_{\partial ij}}   u_{ij} \Delta_{\Sigma^2, \mu} V. \label{eq:RHS-bdry terms-2}
\end{align}
Now (\ref{eq:Zij-n-partial}) follows by summing (\ref{eq:RHS-bdry terms-1}) with (\ref{eq:RHS-bdry terms-2}) and comparing with (\ref{eq:LHS-bdry terms}).
\end{proof}

\begin{lemma} \label{lem:Sigma1Boundary}
 Assume (\ref{eq:A1}). Then on $\Sigma_{uvw}$:
\[
 \sum_{\{i,j\} \subset \{u,v,w\}} Z^{\n_{\partial ij}}_{ij} = \div_{\Sigma^2,\mu} Z_{uvw} . 
\]
\end{lemma}
\begin{proof}
 By assumption (\ref{eq:A1}), a conformally flattening potential $V>0$ satisfies non-oriented conformal BCs, and we have 
\begin{equation} \label{eq:Sigma1Boundary-0}
 \II_{\partial \Sigma_{ij}} = \bar \II^{\partial ij} g_{\partial \Sigma_{ij}} ~,~ \nabla_{\n_{\partial ij}} V = \bar \II^{\partial ij} V .
\end{equation}
Since $u$ satisfies conformal BCs, then for all $(i,j) \in \cyclic(u,v,w)$, we have at $p \in \Sigma_{uvw}$:
\begin{equation} \label{eq:Sigma1Boundary-1}
\nabla_{\n_{\partial ij}} u_{ij} - \bar \II^{\partial ij}  u_{ij} = c_p ~,~ 
\end{equation}
and consequently $\nabla_{\n_{\partial ij}} u_{ij} -\frac{\nabla_{\n_{\partial ij} } V}{V} u_{ij} = c_p$ is independent of $(i,j) \in \cyclic(u, v, w)$. In addition, by the Dirichlet-Kirchoff boundary condition
\begin{equation} \label{eq:Sigma1Boundary-2}
\sum_{(i,j) \in \cyclic(u,v,w)} u_{ij} = 0,
\end{equation}
we have
\[   \sum_{(i, j) \in \cyclic(u,v,w)} \brac{\Delta_{\Sigma^2, \mu} u_{ij} - \frac{\Delta_{\Sigma^2, \mu} V}{V} u_{ij} } =0  ~,~ \sum_{(i,j) \in \cyclic (u,v,w)} \brac{\nabla_{\Sigma^2} u_{ij} - \frac{\nabla_{\Sigma^2} V}{V} u_{ij}}=0.
\] Also note that $\sum_{(i,j) \in \cyclic(u,v,w)} n_{\partial ij} =0$ implies 
\[
\sum_{(i,j) \in \cyclic(u,v,w)} (H_{\partial \Sigma_{ij}} -\nabla_{\n_{\partial ij}}W) =0.
\]
Using all of this in (\ref{eq:Zij-n-partial}), the first 4 terms vanish and we obtain
\[
\sum_{\{i,j\} \subset \{u,v,w\}} Z^{\n_{\partial ij}}_{ij} = \div_{\Sigma^2,\mu} Y_{uvw} ,
\]
where
\[
Y_{uvw} = \sum_{\{ i,j\} \subset \{u,v,w\}} u_{ij} (\nabla_{\n_{\partial ij}} V \nabla_{\Sigma^2} u_{ij} -  \nabla_{\n_{\partial ij}} u_{ij} \nabla_{\Sigma^2} V) . 
\]
Using  (\ref{eq:Sigma1Boundary-0}), (\ref{eq:Sigma1Boundary-1}) and (\ref{eq:Sigma1Boundary-2}) again, we obtain:
\[
Y_{uvw} = \sum_{\{ i,j\} \subset \{u,v,w\}} \bar \II^{\partial ij} u_{ij} ( V \nabla_{\Sigma^2} u_{ij}  - u_{ij} \nabla_{\Sigma^2} V) = Z_{uvw} ,
\]
thereby concluding the proof. 
\end{proof}

\subsection{Boundary computation on $\Sigma^3$}

On $\Sigma_{ijk}$, recall that  $\n_{\partial ij} = \n^{ijk}_{\partial ij}$, and abbreviate:
\[
\bar \II^{ijk} := \bar \II^{\partial ij}_{ijk} = \frac{\II^{ik}_{\n_{\partial ik}} + \II^{jk}_{\n_{\partial jk}}}{\sqrt{3}} .
\]

\begin{lemma} \label{lem:Zuvw-cancelation}
Assume (\ref{eq:A2}) and that $V$ satisfies non-oriented conformal BCs. 
Then the following holds at $p \in \Sigma_{abcd}$:
\[
\sum_{\{u,v,w\} \subset \{a,b,c,d\}}  Z_{uvw}^{\n_{\partial uvw}} = 0 . 
\]
\end{lemma}
\begin{proof}
Let $\{u,v,w\} \subset \{a,b,c,d\}$. 
Observe that by the three-tensor identity (\ref{eq:three-tensor}), applied twice, we have at $p$:
\begin{align} 
\nonumber Z_{uvw}^{\n_{\partial uvw}} & = \sum_{\{i,j\} \subset \{u,v,w\}} \bar \II^{\partial ij}_{uvw} u_{ij} (V \nabla_{\n_{\partial uvw}} u_{ij} - u_{ij} \nabla_{\n_{\partial uvw}} V) \\
\nonumber & = \sum_{(i,j) \in \cyclic(u,v,w)} \bar \II^{\partial ij}_{uvw} u_{ij} (V \nabla_{\n_{\partial uvw}} u_{ij} - u_{ij} \nabla_{\n_{\partial uvw}} V) \\
\label{eq:Z-formula} & = \sum_{(i,j) \in \cyclic(u,v,w)} \II^{ij}_{\n^{ijk}_{\partial ij}}  u_{ij} \brac{V \nabla_{\n_{\partial ijk}} \frac{u_{ik}+u_{jk}}{\sqrt{3}} - \frac{u_{ik}+u_{jk}}{\sqrt{3}} \nabla_{\n_{\partial ijk}} V } ,
\end{align}
where $k$ is the remaining index in $\{u,v,w\}$ distinct from $\{i,j\}$, so that $\{i,j,k\} = \{u,v,w\}$, and similarly we set $\ell$ to be the remaining index in $\{a,b,c,d\}$ so that $\{i,j,k,\ell\} = \{a,b,c,d\}$. 
By regularity and stationarity, each $\Sigma_{ij}$ around $p$ is locally diffeomorphic to a product of $\Sigma_{ijk\ell}$ with a sector of angle $\alpha = \cos^{-1}(-1/3) \simeq 109^{\circ}$. It is easy to check (see Figure \ref{fig:normals}) that
\[
\n_{\partial ijk} = \frac{\n^{ij\ell}_{\partial ij} + \cos(\alpha) \n^{ijk}_{\partial ij}}{\sin(\alpha)} ,
\]
where, recall, $\n_{\partial ijk}$ denotes the outer normal to $\partial \Sigma_{ijk}$ in $\Sigma_{ijk}$.

\begin{figure}[htbp]
\centering
\begin{tikzpicture}[scale=0.8]
\draw[thick] (0,0) -- (0,3);
\draw[thick] (0,0) -- ({-2*sqrt(2)}, -1);
\draw[thick] (0, 0.3) arc[start angle=90, end angle=199, radius=0.3];
\node at (-0.5, 0.3) {$\alpha$};
\node at (-2, 1.2) {$\Sigma_{ij}$};
\node at (0.5,2) {$\Sigma_{ijk}$};
\node at (-1.7,-1){$\Sigma_{ij\ell}$};
\draw[->, thick] (0,0) -- (1,0) node[right]{$\n_{\partial ij}^{ijk}$} ;
\draw[->, thick] (0,0) -- (0,-1) node[below]{$\n_{\partial ijk}$} ;
\draw[->, thick] (0,0) -- ({1/3}, {-2*sqrt(2)/3}) node[right]{$\n_{\partial ij}^{ij\ell}$};
\draw[thick] (0.2,0) -- (0.2,0.2) -- (0,0.2);
\draw[thick]  ({-0.2*(2*sqrt(2)/3)}, {-0.2*(1/3)}) --  ({-0.2*(2*sqrt(2)/3)+ 0.2*(1/3)}, {-0.2*(1/3)+0.2*(-2*sqrt(2)/3)}) -- ({0.2*(1/3)}, {0.2*(-2*sqrt(2)/3)});   
\end{tikzpicture}
\caption{ \label{fig:normals}
Depiction of various normals at $\Sigma_{ijk\ell}$. 
}
\end{figure}
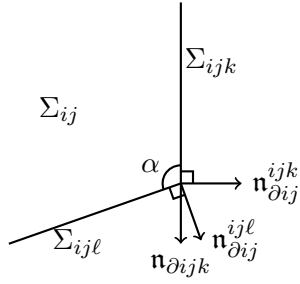

Since $V$ satisfies non-oriented conformal BCs (which continue to hold on $\Sigma_{ijk\ell}$ by regularity), we have at $p \in \Sigma_{ijk\ell}$:
\[
\nabla_{\n_{\partial ij}^{ijm}} V = \bar \II^{ijm} V ~,~ m \in \{k,\ell\} ,
\]
and since $\n_{\partial ijk} = \n_{\partial jki} = \n_{\partial kij}$, we see that:
\begin{align}
\nonumber \nabla_{\n_{\partial ijk}} V & = \frac{1}{\sin(\alpha)} \brac{ \bar \II^{ij\ell} + \cos(\alpha) \bar \II^{ijk} } V \\
\label{eq:V-formula1} & = \frac{1}{\sin(\alpha)} \brac{ \bar \II^{jk\ell} + \cos(\alpha) \bar \II^{jki} } V \\
\label{eq:V-formula2} & = \frac{1}{\sin(\alpha)} \brac{ \bar \II^{k i\ell} + \cos(\alpha) \bar \II^{kij} } V .
\end{align}
Similarly, since $u_{ij}$ satisfies conformal BCs on $\Sigma_{ijk}$
(which continues to hold on $\Sigma_{ijk\ell}$ by regularity and since $u$ is admissible), 
we have at $p \in \Sigma_{ijk\ell}$:
\[
\nabla_{\n_{\partial ij}^{ijm}} u_{ij} = \bar \II^{ijm} u_{ij} + c_{ijm} ~,~ m \in \{k,\ell\}
\]
for some constant $c_{ijm} = - c_{jim}$ which is invariant under cyclic permutation. 
Consequently:
\[
\nabla_{\n_{\partial ijk}} u_{ij} = \frac{1}{\sin(\alpha)} \brac{ \bar \II^{ij\ell} u_{ij} + c_{ij\ell} + \cos(\alpha) (\bar \II^{ijk} u_{ij} + c_{ijk}) } ,
\]
and:
\begin{equation} \label{eq:u-formula} \nabla_{\n_{\partial ijk}} \frac{u_{ik} + u_{jk}}{\sqrt{3}} = \nabla_{\n_{\partial ijk}} \frac{u_{jk} - u_{ki}}{\sqrt{3}} = \frac{1}{\sqrt{3} \sin(\alpha)}  \brac{ \substack{\bar \II^{jk\ell} u_{jk} + c_{jk\ell} + \cos(\alpha) ( \bar  \II^{jki} u_{jk} + c_{jki})  \\
 -\bar \II^{ki\ell} u_{ki} - c_{k i\ell} - \cos(\alpha) ( \bar \II^{kij} u_{ki} + c_{kij})} } . 
\end{equation} Recalling that $c_{jki} = c_{kij}$, and plugging (\ref{eq:V-formula1}), (\ref{eq:V-formula2}) and (\ref{eq:u-formula}) into (\ref{eq:Z-formula}), we obtain:
\[
Z_{uvw}^{\n_{\partial uvw}} = \frac{V}{\sqrt{3} \sin(\alpha)}  \sum_{(i,j) \in \cyclic(u,v,w)} \II^{ij}_{\n^{ijk}_{\partial ij}}  u_{ij} (c_{jk\ell} - c_{ki\ell}) . 
\]
Note that the $(i,j)$ summand above remains the same if we switch $i$ and $j$, since both $\II^{ij}$ and $u_{ij}$ both flip signs, and $c_{jk\ell} - c_{ki\ell} = c_{ik\ell} - c_{kj\ell}$. Denoting $\kappa_{ij} = \II^{ij}_{\n^{ijk}_{\partial ij}} = \II^{ij}_{\n^{ij\ell}_{\partial ij}}$, we obtain:
\begin{align*}
\sum_{\{u,v,w\} \subset \{a,b,c,d\}}  Z_{uvw}^{\n_{\partial uvw}} & = \frac{V}{\sqrt{3} \sin(\alpha)} \sum_{\{i,j\} \subset \{a,b,c,d\}} \kappa_{ij} u_{ij}  \sum_{\{k,\ell\} = \{a,b,c,d\} \setminus \{i,j\}} (c_{jk\ell} - c_{ki \ell}) \\
& = \frac{V}{\sqrt{3} \sin(\alpha)} \sum_{\{i,j\} \subset \{a,b,c,d\}} \kappa_{ij} u_{ij}   (c_{j p q} - c_{p i q} + c_{jq p} - c_{q i p}) , 
\end{align*}
where we define $\{p , q\} = \{a,b,c,d\} \setminus \{i,j\}$ so that $p < q$. But since $c_{j p q} = - c_{j q p}$ and $c_{p i q} = -c_{q i p}$, the expression on the right vanishes, concluding the proof. 
\end{proof}

\subsection{Concluding the proof of Lemma \ref{lem:crazyZ}}

\begin{proof}[Proof of Lemma \ref{lem:crazyZ}] \hfill
\begin{enumerate}
\item
The first assertion was established in Lemma \ref{lem:Sigma1Boundary} using assumption (\ref{eq:A1}). Note that the latter assumption together with the assumption that $\Sigma^1$ is of locally bounded curvature implies that $\Sigma^2$ is itself of locally bounded curvature in $\partial \Sigma^1$. 
\item
For the second assertion, since $u_{ij}$ are all supported in some compact set $K$ and $\Sigma$ is locally finite, it is enough to show that $\int_{\Sigma_{ij} \cap K} |\div_{\Sigma^1,\mu} Z_{ij}| d\mu^{n-1} < \infty$, $\int_{\Sigma_{ij} \cap K} |Z_{ij}|^2 d\mu^{n-1} < \infty$ and $\int_{\Sigma_{uvw} \cap K} |\div_{\Sigma^2,\mu} Z_{uvw}| d\mu^{n-2} < \infty$.
 Note that $V$, $1/V$, $\nabla_{\Sigma^1} V$ and $\nabla_{\Sigma^1} W$ are bounded on $K$, and so do not influence the integrability. In addition, the assumption that $\Sigma^1$ has bounded curvature and the compactness of $K$ imply that $\nabla^2_{\Sigma^1} V$, $\Delta_{\Sigma^1,\mu} V$, $\nabla_{\Sigma^1} |\nabla_{\Sigma^1} V|^2$ and $\nabla^2_{\Sigma^1} W$ are bounded; we will freely use this below.

 By Lemma \ref{lem:divZij}, we have a succinct description of $\div_{\Sigma^1,\mu} Z_{ij}$, and we've already commented on the boundedness of most of the terms. As for $\Ric_{\Sigma,\mu} = \Ric_{\Sigma^1} + \nabla^2_{\Sigma^1} W$, the assumption that  $\Sigma^1$ has bounded curvature, the compactness of $K$ and the Gauss formula, together imply that $\Ric_{\Sigma,\mu}$ and hence $V \hRic^V_{\Sigma,\mu}$ are bounded. The $L^1(\Sigma_{ij} \cap K , \mu^{n-1})$ integrability of $\div_{\Sigma^1,\mu} Z_{ij}$ then follows from the $L^2(\Sigma_{ij} \cap K, \mu^{n-1})$ integrability of $u_{ij}$, $\nabla_{\Sigma^1} u_{ij}$ and $\nabla^2_{\Sigma^1} u_{ij}$. 
 
 As for the $L^2(\Sigma_{ij} \cap K,\mu^{n-1})$ integrability of $Z_{ij}$, recall that $u_{ij}, \nabla u_{ij} \in L^\infty(\Sigma_{ij} \cap K)$, and so 
 most of the terms in (\ref{eq:Zij}) are bounded and thus $L^2$-integrable since $\mu^{n-1}( \Sigma_{ij} \cap K) < \infty$. The remaining terms are:
 \[
 \frac{V}{2} \nabla_{\Sigma^1} |\nabla_{\Sigma_1} u_{ij}|^2 ~,~ V \Delta_{\Sigma, \mu} u_{ij} \nabla_{\Sigma^1} u_{ij} ~,~ 2 u \nabla^t_{\nabla_{\Sigma^1} V} \nabla_{\Sigma^1} u_{ij} ~,~ 2u_{ij} \Delta_{\Sigma^1, \mu} u_{ij} \nabla_{\Sigma^1} V .
 \]
The $L^2(\Sigma_{ij} \cap K,\mu^{n-1})$ integrability of $Z_{ij}$ follows since $\nabla_{\Sigma^1} u, \nabla_{\Sigma^1} V \in L^\infty(\Sigma_{ij} \cap K)$ and $\nabla^2_{\Sigma^1} u_{ij} \in L^2(\Sigma_{ij} \cap K , \mu^{n-1})$. 

As for the $L^1(\Sigma_{uvw} \cap K,\mu^{n-2})$ integrability of $\div_{\Sigma^2,\mu} Z_{uvw}$, recall (\ref{eq:Zuvw}) and compute:
\begin{align*}
\div_{\Sigma^2,\mu} Z_{uvw} =
\sum_{\{ i,j\} \subset \{u,v,w\}} & u_{ij} \scalar{\nabla_{\Sigma^2} \bar \II^{\partial ij}_{\Sigma_{uvw}} ,  V \nabla_{\Sigma^2} u_{ij}  - u_{ij} \nabla_{\Sigma^2} V} \\
& + \bar \II^{\partial ij}_{\Sigma_{uvw}} \scalar{\nabla_{\Sigma^2} u_{ij} , V \nabla_{\Sigma^2} u_{ij}  - u_{ij} \nabla_{\Sigma^2} V} \\
& + \bar \II^{\partial ij}_{\Sigma_{uvw}}  u_{ij} ( V \Delta_{\Sigma^2,\mu} u_{ij} - u_{ij} \Delta_{\Sigma^2,\mu} V) . 
\end{align*}
Recall that $u_{ij}$ and $\nabla_{\Sigma^1} u_{ij}$ are bounded on $\Sigma_{ij} \cap K$, and by continuity this extends to $\Sigma_{uvw} \cap K$. The assumption (\ref{eq:A1B}) ensures the $L^1(\Sigma_{uvw} \cap K,\mu^{n-2})$ integrability of the first term; the locally bounded curvature of $\Sigma$ and $\mu^{n-2}(\Sigma_{uvw} \cap K) < \infty$ ensure the integrability of the second term; the assumption that $\Delta_{\Sigma^2,\mu} u_{ij} \in L^1(\Sigma_{uvw} , \mu^{n-1})$ ensures the integrability of the first half of the third term; and finally the locally bounded curvature of $\Sigma_{uvw}$ in $\partial \Sigma_{ij}$ ensures that $\Delta_{\Sigma^2,\mu} V$ is bounded on $K$, which together with the previous properties establishes the second assertion. 

\item
The third assertion was established in Lemma \ref{lem:Zuvw-cancelation} using assumption (\ref{eq:A2}).

\item
For the fourth assertion, there is no need to check the integrability of the boundary term (as it vanishes by the previous assertion), and $\int_{\Sigma^2} |\div_{\Sigma^2,\mu} Z_{uvw}| d\mu^{n-2} < \infty$ was already established in the second assertion. Consequently, it is enough to show that 
 $\int_{\Sigma_{uvw} \cap K} |Z_{uvw}|^2 d\mu^{n-2} < \infty$. The $L^\infty(\Sigma_{ij})$ bound on $u_{ij}$ and $\nabla u_{ij}$ extends by continuity to $\partial \Sigma_{ij}$, and so recalling that $\Sigma_{uvw}$ has locally bounded curvature and the definition (\ref{eq:Zuvw}), we see that $Z_{uvw}$ is bounded on $\Sigma_{uvw} \cap K$. Since $\mu^{n-2}(\Sigma^2 \cap K) < \infty$, it follows that  $\int_{\Sigma_{uvw} \cap K} |Z_{uvw}|^2 d\mu^{n-2} < \infty$, concluding the proof. 
 \end{enumerate}
 \end{proof}

\section{Brascamp-Lieb inequality on partitions with conformally flat umbilical boundary} \label{sec:BL}

We are now ready to prove  the conjugated Brascamp-Lieb Theorem \ref{thm:intro-BL} from the Introduction. It follows from the integrated Bochner identity (\ref{eq:Bochner}) by a standard duality argument. 

\medskip

\begin{lem}
    For any $u_{ij} \in C^2_{loc} (\Sigma_{ij})$ and $N \in (-\infty,0) \cup [n-1,\infty]$, 
    \begin{align*}
    &V \hRic_{\Sigma^1, \mu}^V\brac{\nabla_{\Sigma^1} u_{ij}-\frac{\nabla_{\Sigma^1} V}{V} u_{ij} }+ V \norm{\nabla_{\Sigma^1}^2 u_{ij} - \frac{\nabla_{\Sigma^1}^2 V}{V} u_{ij}}^2 \\
 & \geq \; V \hRic_{\Sigma^1, \mu, N}^V \brac{\nabla_{\Sigma^1} u_{ij}-\frac{\nabla_{\Sigma^1} V}{V} u_{ij}} + \frac{V}{N}  \brac{ \Delta_{\Sigma^1, \mu} u_{ij}- \frac{\Delta_{\Sigma^1, \mu} V}{V} u_{ij}}^2 .
    \end{align*}
\end{lem}
This is well-known in the case that $V \equiv 1$ (see e.g.~\cite[Lemma 2.3]{KolesnikovEMilmanReillyPart1}), and the extension to general $V$ is straightforward (see e.g.~\cite[Equation (1.5)]{HuangMaZhu-ReillyFormula}), so we omit it. 

\begin{proof}[Proof of Theorem  \ref{thm:intro-BL}]
Applying the above inequality to (\ref{eq:Bochner}) and rearranging terms, we obtain for any admissible scalar-field $u = (u_{ij})$,
\begin{equation}
\label{eq:BochnerInq}
\begin{split} 
\frac{N-1}{N} & \int_{\Sigma^1} V  \brac{ \Delta_{\Sigma^1, \mu} u_{ij}- \frac{\Delta_{\Sigma^1, \mu} V}{V} u_{ij} }^2 d\mu^{n-1} \\
&\geq \int_{\Sigma^1} V \hRic_{\Sigma^1, \mu, N}^V \brac{\nabla_{\Sigma^1} u_{ij} - \frac{\nabla_{\Sigma^1} V}{V} u_{ij}}d\mu^{n-1} . 
\end{split}
\end{equation}
Recall that the scalar-field $f = (f_{ij})$ satisfies $f_{ij}= L_V u_{ij} = V \Delta_{\Sigma^1, \mu} u_{ij} - u_{ij} \Delta_{\Sigma^1, \mu} V$. Corollary \ref{cor:delta1V-zero} verifies that $\delta^1_f \Vol = 0$, establishing the first assertion. Now, using integration-by-parts (Lemma \ref{lem:byparts}) and the Cauchy-Schwarz inequality (with respect to the positive-definite quadratic form $V \hRic^V_{\Sigma,\mu,N} > 0$), we have
\begin{align*}
\int_{\Sigma^1} \frac{f_{ij}^2}{V} d\mu^{n-1} = &\int_{\Sigma^1} V \brac{ \Delta_{\Sigma^1, \mu} u_{ij}- \frac{\Delta_{\Sigma^1, \mu} V}{V} u_{ij} }^2 d\mu^{n-1}\\
=&\int_{\Sigma^1} f_{ij} \brac{ \Delta_{\Sigma^1, \mu} u_{ij}- \frac{\Delta_{\Sigma^1, \mu} V}{V} u_{ij} } d\mu^{n-1}\\
=& - \int_{\Sigma^1} \scalar{\nabla_{\Sigma^1} f_{ij} -\frac{\nabla_{\Sigma^1} V}{V} f_{ij}, \nabla_{\Sigma^1} u_{ij}- \frac{\nabla_{\Sigma^1} V}{V} u_{ij}} d\mu^{n-1}\\
\leq& \brac{\int_{\Sigma^1} \brac{V \hRic_{\Sigma, \mu, N}^V }^{-1} \brac{\nabla_{\Sigma^1} f_{ij} - \frac{\nabla_{\Sigma^1} V}{V} f_{ij}} d\mu^{n-1}}^{\frac{1}{2}} \\
 & \cdot \brac{\int_{\Sigma^1} V \hRic_{\Sigma^1, \mu, N}^V \brac{\nabla_{\Sigma^1} u_{ij} - \frac{\nabla_{\Sigma^1} V}{V} u_{ij}} d\mu^{n-1} }^{\frac{1}{2}}.
\end{align*}
Combining the above inequality with (\ref{eq:BochnerInq}), we obtain the asserted Brascamp-Lieb inequality:
\begin{equation*}
\frac{N}{N-1} \int_{\Sigma^1} \frac{f_{ij}^2}{V} d\mu^{n-1} 
\leq
\int_{\Sigma^1} \brac{V \hRic_{\Sigma^1, \mu, N}^V}^{-1} \brac{\nabla_{\Sigma^1} f_{ij} - \frac{\nabla_{\Sigma^1} V}{V} f_{ij}} d\mu^{n-1}.
\end{equation*}

\end{proof}

Corollary \ref{cor:intro-stability} from the Introduction now immediately follows from Theorem \ref{thm:intro-BL} and Lemma \ref{lem:Q0V}.

\section{Conformally flattening boundary potentials on M\"obius-flat spherical Voronoi partitions} \label{sec:V}

The proof of Theorem \ref{thm:intro-V} from the Introduction is comprised of Propositions \ref{prop:VSn}, \ref{prop:VRn} and \ref{prop:VHn} below. We begin with the following definition and lemma, which we will use several times. 

\begin{definition}
A collection of geodesic spheres $\{S_m\}_{m \in I} \subset \S^n$ is called M\"obius-flat if there exists a M\"obius map $T: \S^n \to \S^n$ such that all $\{T S_m\}_{m \in I}$ are flat (namely, have zero curvature). The same definition applies to a collection $\{\Sigma_m\}_{m \in I}$ of relatively open non-empty subsets of geodesic spheres in $\S^n$. \end{definition}

\begin{lem}\label{lem:Mobius-flat-xi-sphere}
A collection of geodesic spheres $\{S_m\}_{m \in I}\subset \S^n$ with quasi-centers $(\c_m)_{m \in I}$ and curvatures $(\k_m)_{m \in I}$ is M\"obius-flat if and only if:
\[\exists \xi \in \R^{n+1} ~ \abs{\xi}<1 \quad \text{so that} \quad \sscalar{\c_m, \xi}+\k_m = 0 ~~ \forall m\in I.
\] \end{lem}
\begin{proof}
This follows from the proof of \cite[Lemma 10.15]{EMilmanNeeman-TripleAndQuadruple}.
\end{proof}

\subsection{Spherical Voronoi partitions in $\S^n$}

Recall Definition \ref{def:prelim-Voronoi} of spherical Voronoi generalized partition $\Omega^\S$ from Subsection \ref{subsec:prelim-Voronoi}. Note that the second condition implies that a non-empty $\Sigma^\S_{ij}$ is a subset of the geodesic sphere $S^\S_{ij} = \{ p \in \S^{n} : \sscalar{\c_{ij}^\S,p} + \k_{ij}^\S = 0\}$. Consequently, given the second condition, the requirement in the first condition for every non-empty $\Sigma^\S_{ij}$ that $\c^\S_{ij}$ and $\k^\S_{ij}$ are the quasi-center and curvature of $S^\S_{ij}$, respectively, is equivalent to the requirement that 
\begin{equation} \label{eq:ckS-relation}
|\c_{ij}^\S |^2 = |\n^\S_{ij} - \k^\S_{ij} p|^2 = 1 +  (\k_{ij}^\S)^2 . 
\end{equation}

We also repeat the definition of M\"obius-flatness in $\S^n$ to be consistent with the definition we will give in $\HH^n$ below, and record the following immediate corollary of Lemma \ref{lem:Mobius-flat-xi-sphere}.
\begin{definition} A spherical Voronoi generalized partition $\Omega^\S$ in $\S^n$ is called M\"obius-flat if the collection $\{\Sigma_{ij}^\S : \Sigma_{ij}^\S \neq \emptyset \}$ of its non-empty interfaces is  M\"obius-flat.
\end{definition}

\begin{corollary} \label{cor:xiS}
A spherical Voronoi generalized partition $\Omega^\S$ in $\S^n$ is M\"obius-flat if and only if:
\begin{equation} \label{eq:xiS}
\exists \xi \in \R^{n+1} ~ \abs{\xi}<1 \quad \text{so that} \quad \sscalar{\c^\S_{ij}, \xi}+\k^\S_{ij} = 0 ~~ \forall \Sigma^\S_{ij} \neq \emptyset ,
\end{equation}
where $\{\c^\S_i\}$ and $\{\k^\S_i\}$ are its quasi-center and curvature parameters, respectively. 
\end{corollary}

\begin{prop} \label{prop:VSn}
Let $\Omega^\S$ be a regular, M\"obius-flat, spherical Voronoi, $q$-partition in $\S^n$. Then there exists a conformally boundary flattening potential $V > 0$ on $\S^n$ such that
\begin{align*}
L_{Jac} V = n-1, ~V \hRic^V_{\Sigma^1}= (n-2) g_{\Sigma}~\text{on}~\Sigma^1.
\end{align*}
Specifically, if $\xi \in \R^{n+1}$ is as in (\ref{eq:xiS}), then $V: \S^n \to \R_+$ may be chosen as 
\begin{equation*}
V(p)= 1 -\scalar{p,\xi}, \quad p \in \S^n.
\end{equation*}
\end{prop}

\begin{proof}
Clearly $V$ is a smooth positive function on $\S^n$ as $|\xi|<1$. By (\ref{eq:xiS}), we have $\sscalar{\c_{ij}^\S, \xi} + \k_{ij}^\S =0$ for all non-empty $\Sigma_{ij}^\S$, and hence
\[
 \nabla_{\n_{ij}^\S} V = -\scalar{\xi, \n_{ij}} = -\sscalar{\xi, \c_{ij}^\S}- \k_{ij}^\S \scalar{\xi, p}= \k_{ij}^\S  V.
\]
By Lemma \ref{lem:RegularVoronoiIsStationary}, $\Omega^\S$ is stationary. Consequently, $\n^\S_{\partial ij} = \frac{1}{\sqrt{3}} (\n_{ik}^\S +\n_{jk}^\S)$ on $\Sigma_{ijk}$, and since $\bar{\II}^{\partial ij}_{ijk} = \frac{1}{\sqrt{3}} (\k_{ik}^\S +\k_{jk}^\S)$, it follows that $\nabla_{\n_{\partial ij}^\S} V =  \bar{\II}^{\partial ij} V$ on $\partial \Sigma_{ij}^\S$. Therefore $V$ satisfies the non-oriented conformal BCs, and since all interfaces are generalized spheres, Lemma \ref{lem:constant-curvature} implies that $V$ is a conformally flattening boundary potential.

By a direct calculation, on each non-empty interface $\Sigma_{ij}^\S$, we have
\begin{align*}
\nabla_{\Sigma^1}^2 V=& \nabla_{\R^{n+1}}^2 V+\sscalar{p+ \k_{ij}^\S \n_{ij}^\S, \xi} g_{\Sigma^1}\\
=& \sscalar{p+ (\k_{ij}^\S)^2 p + \k_{ij}^\S \c_{ij}^\S, \xi} g_{\Sigma^1}\\
=& \brac{1+(\k_{ij}^\S)^2} \scalar{p, \xi} g_{\Sigma^1} -(\k_{ij}^\S)^2 g_{\Sigma^1}\\
=& \brac{ 1- (1+ (\k_{ij}^\S)^2) V} g_{\Sigma^1}.
\end{align*}
Then $\Delta_{\Sigma^1} V = (n-1) (1- (1+ (\k_{ij}^\S)^2) V)$ on $\Sigma_{ij}^\S$. Since $\Ric_{\S^n} (\n_{ij}^\S, \n_{ij}^\S) + || \II^{ij}||_2^2 = (n-1)(1+ (\k_{ij}^\S)^2)$ on $\Sigma_{ij}^\S$, we deduce on the entire $\Sigma^1$:
\begin{equation*}
L_{Jac} V =\Delta_{\Sigma^1} V + \brac{ \Ric_{\S^n} (\n, \n) + \norm{ \II}^2 } V=n-1.
\end{equation*}
On the other hand, by the Gauss equation, the Ricci curvature of the non-empty interface $\Sigma_{ij}^\S$ is $\Ric_{\Sigma^1} = (n-2) (1+ (\k_{ij}^\S)^2) g_{\Sigma^1}$. Thus,
\begin{align*}
V \hRic_{\Sigma^1}^V =& V \Ric_{\Sigma^1} - \nabla_{\Sigma^1}^2 V + \Delta_{\Sigma^1} V g_{\Sigma^1}\\
=&\left ( (n-2)(1+ (\k_{ij}^\S)^2) V - \brac{ 1- (1+ (\k_{ij}^\S)^2) V} \right . \\
& \left . \;\; + (n-1) \brac{1- (1+ (\k_{ij}^\S)^2) V} \right )  g_{\Sigma^1}\\
=&(n-2) g_{\Sigma^1}.
\end{align*}
This completes the proof.
\end{proof}

\subsection{Spherical Voronoi partitions in $\mathbb{R}^n$}

The standard stereographic projection $\pi_\R: \R^n \to \S^n \subset \R^{n+1}$ is given by
\begin{equation*}
\pi_{\R} (x) = \frac{2x}{\abs{x}^2+1} + \frac{\abs{x}^2-1}{\abs{x}^2+1} e_{n+1}, \quad x \in \R^n.
\end{equation*}
It is a diffeomorphism between $\R^n$ and $\S^n \setminus \{ N\}$, where recall $N = e_{n+1}$ denotes the North pole.

\begin{definition}[(M\"obius-flat) Spherical Voronoi partition in $\R^n$] \label{def:VoronoiR}
A partition $\Omega^\R$ in $\R^n$ is called a spherical Voronoi partition, if there exists a spherical Voronoi generalized partition $\Omega^\S$ in $\S^n$ so that $\Omega^\R = \pi_\R^{-1}(\Omega^\S)$. $\Omega^\R$ is called M\"obius-flat if $\Omega^\S$ is M\"obius-flat. 
\end{definition}

The above definition uniquely determines $\Omega^\S$ from $\Omega^\R$, up to the allocation of the North pole $N$; we will always use the convention that $N \in \Omega^\S_i$ iff $\R^n \setminus \Omega^\R_i$ is bounded, thereby determining $\Omega^\S$ uniquely. 
We assume throughout this subsection that $\Omega^\R$ is a spherical Voronoi partition in $\R^n$ with $\Omega^\R = \pi_\R^{-1}(\Omega^\S)$, where $\Omega^\S$ is a spherical Voronoi generalized partition in $\S^n$ with quasi-center parameters $\{\c_i^\S \}_{i=1, \ldots, q} \subset \R^{n+1}$ and curvature parameters $\{\k_i^\S\}_{i=1, \ldots, q} \subset \R$. For $i=1,\ldots,q$, write:
\[
\c^\S_i = (\underline \c_i^\S, (\c_0^\S)_i) \in \R^n \times \R ,
\]
and denote:
\begin{equation} \label{eq:kRkS}
\c_i^\R := \underline \c_i^\S ~,~ \k_i^\R := \k_i^\S + (\c_0^\S)_i . 
\end{equation}
An equivalent characterization of spherical Voronoi partitions in $\R^n$ is given by the following lemma, established in \cite[Lemma 8.16]{EMilmanNeeman-TripleAndQuadruple}.

\begin{lem} \label{lem:VorS-VorR}
Assume that $\Omega^\R$ is a spherical Voronoi $q$-partition in $\R^n$. Then $\Omega^\R$ has the following Voronoi representation:
\begin{align*}
\Omega_i^\R =& ~ \set{x \in \R^n : \argmin_{j=1, \ldots, q} \brac{\k_j^\R|x|^2+ 2 \sscalar{\c_j^\R, x}+2\k_j^\S -\k_j^\R } = \{i\}  }\\
=&\bigcap_{j \neq i} \set{x \in \R^n : \k_{ij}^\R |x|^2+ 2 \sscalar{\c_{ij}^\R, x} + 2 \k_{ij}^\S - \k_{ij}^\R<0 }, ~ i=1, \ldots, q. 
\end{align*}
\end{lem}

\begin{prop}\label{prop:VRn}
Let $\Omega^\R$ be a regular, M\"obius-flat, spherical Voronoi, $q$-partition in $\R^n$. Then there exists a conformally boundary flattening potential $V > 0$ on $\R^n$ such that
\begin{equation*}
L_{Jac} V = n-1 ~,~ V \hRic^V_{\Sigma^1}= (n-2) g_{\Sigma}~\text{on}~\Sigma^1 . \end{equation*}
Specifically, let $\xi = (\underline \xi, \xi_0)\in \mathbb{R}^{n} \times \R$ with $|\xi|<1$ be the vector from (\ref{eq:xiS}) guaranteed to satisfy $\sscalar{\c^\S_{ij}, \xi} + \k^\S_{ij}=0$ for all non-empty $\Sigma^\S_{ij}$. Then $V:\R^n \to \R_+$ can be chosen as
\begin{equation*}
V(x) = \frac{\abs{x}^2}{2}- \scalar{x, \theta} + \eta, \quad x \in \R^n,
\end{equation*}
where 
\[
\theta := \frac{\underline \xi}{1-\xi_0} ~,~ \eta := \frac{1}{2} \cdot \frac{1+ \xi_0}{1-\xi_0} .
\]
\end{prop}

For the proof, we first require a lemma. 

\begin{lem}\label{lem:partR-partS}
Let $\Sigma_{ij}^\R$ be a non-empty interface of a M\"obius-flat, spherical Voronoi partition $\Omega^\R$ in $\R^n$. Let $\theta$ and $\eta$ be as in Proposition \ref{prop:VRn}. Then
\begin{equation}
\k_{ij}^\R \frac{|x|^2}{2} + \sscalar{\c_{ij}^\R, x- \theta} -\k_{ij}^\R \eta=0, \quad \forall x \in \Sigma_{ij}^\R. \label{eq:rel-3}
\end{equation}
\end{lem}

\begin{proof}
Recalling (\ref{eq:kRkS}), we have:
\begin{align}
&\sscalar{\c_{ij}^\R, \theta} + \k_{ij}^\R \eta+ \frac{\k_{ij}^\S-(\c_0^\S)_{ij}}{2} \label{eq:rel-1} \\
& = \sscalar{\underline \c_{ij}^\S, \theta} +(\c_0^\S)_{ij} \brac{\eta-\frac{1}{2}}+ \k_{ij}^\S \brac{\eta+\frac{1}{2}} \nonumber \\
& = \frac{1}{1-\xi_0} \brac{\sscalar{\underline \c_{ij}^\S, \underline \xi}+ (\c_0^\S)_{ij}\xi_0+ \k_{ij}^\S  } \nonumber\\
& = \frac{1}{1-\xi_0} \brac{\sscalar{\c_{ij}^\S, \xi} + \k_{ij}^\S}=0.  \nonumber
\end{align} 
Since $2\k_{ij}^\S - \k_{ij}^\R = \k_{ij}^\S -(\c_0^\S)_{ij}$, the Voronoi representation of $\Omega^\R$ in Lemma \ref{lem:VorS-VorR} implies
\begin{equation*}
\k_{ij}^\R \frac{|x|^2}{2} + \sscalar{\c_{ij}^\R, x} + \frac{\k_{ij}^\S -(\c_0^\S)_{ij}}{2} =0, \quad \forall x \in \Sigma_{ij}^\R.
\end{equation*}
Then (\ref{eq:rel-3}) follows by this formula and (\ref{eq:rel-1}). 
\end{proof}

\begin{proof}[Proof of Proposition \ref{prop:VRn}]
Since $|\xi|<1$, we have $|\xi_0|<1$, and hence
\begin{align*}
V(x) =& \; \frac{|x|^2}{2} - \scalar{x, \frac{\underline \xi}{1-\xi_0}}+ \frac{1+\xi_0}{2(1-\xi_0)}\\
=& \; \frac{1}{2} \brac{ \abs{x - \frac{\underline \xi}{1-\xi_0}}^2 - \frac{|\underline \xi|^2}{(1-\xi_0)^2} + \frac{1+\xi_0}{1-\xi_0} }\\
=& \; \frac{1}{2} \brac{ \abs{x- \frac{\underline \xi}{1-\xi_0}}^2 + \frac{1-|\xi|^2}{(1-\xi_0)^2} }>0 ,
\end{align*}
establishing that $V$ is a smooth positive function on $\R^n$.
Using (\ref{eq:rel-3}), we obtain on $\Sigma_{ij}^\R$
\begin{align}
\nabla_{\n_{ij}^\R} V & = \sscalar{x-\theta, \c_{ij}^\R + \k_{ij}^\R x } \nonumber\\
& = \k_{ij}^\R \brac{ \frac{|x|^2}{2} - \sscalar{x, \theta} +\eta }+ \k_{ij}^\R \frac{|x|^2}{2} + \sscalar{\c_{ij}^\R, x-\theta} - \k_{ij}^\R \eta \nonumber\\
& =\k_{ij}^\R V. \label{eq:BCv-Eucl}
\end{align}
By Lemma \ref{lem:RegularVoronoiIsStationary}, $\Omega^\R$ is stationary, and so as in the proof of Proposition \ref{prop:VSn}, $\nabla_{\n_{\partial ij}^\R} V =  \bar{\II}^{\partial ij} V$ on $\partial \Sigma_{ij}^\R$. Therefore $V$ satisfies the non-oriented conformal BCs, and since all interfaces are generalized spheres, Lemma \ref{lem:constant-curvature} implies that $V$ is a conformally flattening boundary potential.

Using (\ref{eq:BCv-Eucl}) and since $\nabla_{\R^n}^2 V = g_{\R^n}$, on each non-empty interface $\Sigma_{ij}^\R$ we have
\begin{equation*}
\nabla_{\Sigma^1}^2 V= \nabla_{\R^n}^2 V - \nabla_{\n_{ij}^\R} V \, \II_{\Sigma^1}=\brac{1- (\k_{ij}^\R)^2 V } g_{\Sigma^1}.
\end{equation*}
Then $\Delta_{\Sigma^1} V = (n-1)(1- (\k_{ij}^\R)^2 V)$ on $\Sigma_{ij}^\R$.  Since $\Ric_{\R^n} (\n, \n) + \norm{\II}^2= (n-1) (\k_{ij}^\R)^2$ on $\Sigma_{ij}^\R$, we deduce on the entire $\Sigma^1$:
\begin{equation*}
L_{Jac} V = \Delta_{\Sigma^1} V + \brac{ \Ric_{\R^n} (\n, \n) + \norm{\II}^2} V= n-1 .
\end{equation*}
On the other hand, we have that $\Ric_{\Sigma^1} = (n-2) (\k_{ij}^\R)^2 g_{\Sigma^1}$ on $\Sigma_{ij}^\R$. Thus,
\begin{align*}
V \hRic_{\Sigma^1}^V =& V \Ric_{\Sigma^1} - \nabla_{\Sigma^1}^2 V + \Delta_{\Sigma^1} V g_{\Sigma^1}\\
=&\brac{ (n-2)(\k_{ij}^\R)^2 V - (1- (\k_{ij}^\R)^2 V) + (n-1) (1- (\k_{ij}^\R)^2 V)} g_{\Sigma^1}\\
=&(n-2) g_{\Sigma^1}.
\end{align*}
This completes the proof.
\end{proof}

\subsection{Spherical Voronoi partitions in $\HH^n$}

Contrary to the spherical Voronoi partitions (or clusters) in $\S^n$ and $\R^n$, which were already introduced and studied in \cite{EMilmanNeeman-TripleAndQuadruple,EMilmanNeeman-QuintupleBubble},  spherical Voronoi partitions in $\HH^n$ have not been previously defined, and there are some subtleties to take care of. In particular, we will not define a spherical Voronoi partition in $\HH^n$ as a stereographic projection of a spherical Voronoi (generalized) partition in $\S^n$, as we did for $\R^n$, because our stereographic projection $\pi_\HH$ will map $\HH^n$ onto the Northern hemisphere $\S^n_+$ instead of the punctured sphere $\S^n \setminus \{ N\}$. 
The reason is that $\HH^n$ is not conformal to $\S^n \setminus \{ N \}$, and so in particular there is no stereographic projection mapping $\HH^n$ to $\S^n \setminus \{ N \}$. This is related to the fact that there are no non-trivial conformal automorphisms of $\HH^n$ -- every conformal automorphism is necessarily an isometry. 
\medskip

Recall the hyperboloid model of $\HH^n$ in Minkowski space-time $(\R^{n,1}, \scalar{\cdot,\cdot}_1)$ introduced in Subsection \ref{subsec:prelim-gen-spheres}, given by: 
\begin{equation*}
\HH^n= \set{ y = (\underline{y}, y_0)  \in \R^{n,1} : \sabs{y}^2_1 = -1, ~ y_0 >0 },
\end{equation*}
where we employ the useful abbreviation $\sabs{y}^2_1 := \sscalar{y,y}_1$. 
In particular, if $Z \in T_y \HH^n$ then $\scalar{Z,y}_1 = 0$. 
Also recall that a complete hypersurface $S$ in $\HH^n$ of constant curvature $\k$ with respect to the unit-normal $\n$ is called a generalized sphere, and that its quasi-center $\c$ is defined as $\c = \n - \k y$. Note that $S$ is then given by $\{ y \in \HH^{n} : \sscalar{\c,y}_1 - \k = 0\}$.

\begin{definition}[Spherical Voronoi partition in $\HH^n$]
A  partition $\Omega^\HH = (\Omega^\HH_1, \ldots, \Omega^\HH_q)$ in $\HH^n$, all of whose cells are non-empty, is called a spherical Voronoi partition if there exist $\{\c_i^\HH \}_{i=1, \ldots, q} \subset \R^{n,1}$ and $\{\k_i^\HH \}_{i=1, \ldots, q} \subset \R$ so that the following holds:
\begin{enumerate}
\item For every non-empty interface $\Sigma_{ij}^\HH \neq \emptyset$, $\Sigma_{ij}^\HH$ is a relatively open subset of a generalized sphere $S_{ij}^\HH$ with quasi-center $\c_{ij}^\HH = \c_i^\HH - \c_j^\HH$ and curvature $\k_{ij}^\HH = \k_i^\HH - \k_j^\HH$.
\item The following Voronoi representation holds:
\begin{equation*}
\Omega_i^\HH = \set{y \in \HH^n :\argmin_{j=1, \ldots, q} (\sscalar{\c_j^\HH, y}_1 -\k_{j}^\HH) =\{i\}}
= \bigcap_{j \neq i} \set{y \in \HH^n : \sscalar{\c_{ij}^\HH, y}_1 -\k_{ij}^\HH<0}.
\end{equation*}
\end{enumerate}
\end{definition}
Note that the second condition implies that a non-empty $\Sigma^\HH_{ij}$ is a subset of the generalized sphere $S^\HH_{ij} = \{ y \in \HH^{n} : \sscalar{\c_{ij}^\HH,y}_1 - \k_{ij}^\HH = 0\}$. Consequently, given the second condition, the requirement in the first condition for every non-empty $\Sigma^\HH_{ij}$ that $\c^\HH_{ij}$ and $\k^\HH_{ij}$ are the quasi-center and curvature of $S^\HH_{ij}$, respectively, is equivalent to the requirement that
\begin{equation} \label{eq:ckH-relation}
\sabs{\c_{ij}^\HH}^2_1 = 
\sscalar{\n_{ij}^\HH-\k_{ij}^\HH y, \n_{ij}^\HH-\k_{ij}^\HH y}_1=1- (\k_{ij}^\HH)^2 . 
\end{equation}

\begin{lemma}\label{lem:HnStationary}
A regular spherical Voronoi partition in $\HH^n$ is stationary.
\end{lemma}
\begin{proof}
Let $\Omega^\HH$ be a regular spherical Voronoi partition in $\HH^n$. Since $\Omega^\HH$ is locally finite and $\II_{\Sigma^1}=\k_{ij}^\HH g_{\Sigma^1}$ on each non-empty interface $\Sigma_{ij}^\HH$, we conclude that $\Omega$ has locally bounded curvature. Note that $H_{\Sigma_{ij}^\HH} = (n-1)\k_i^\HH-(n-1)\k_j^\HH$ for all interfaces $\Sigma_{ij}^\HH$ and that $\sum_{(i,j)\in \cyclic(u,v,w) } \n_{ij}^\HH =\sum_{(i,j)\in \cyclic(u,v,w) } (\c_{ij}^\HH+\k_{ij}^\HH y )=0 $ for all $y \in \Sigma_{uvw}$ and $u<v<w$. It follows by Lemma \ref{lem:stationarity} that $\Omega^\HH$ is stationary.
\end{proof}

\begin{lemma} \label{lem:bounded}
A spherical Voronoi $(q-1)$-cluster $\Omega^\HH$ in $\HH^n$ is bounded. 
\end{lemma}
While the analogous statement in $\R^n$ is obvious, this is not entirely so in $\HH^n$, since a set $U$ of the form $U = \bigcap_m \set{y \in \HH^n : \sscalar{\c_{m}^\HH, y}_1 -\k_{m}^\HH<0}$ may be unbounded and yet have finite volume and perimeter.
\begin{proof}
For any $i \in \{1,\ldots,q-1\}$ such that $\Sigma^\HH_{iq} \neq \emptyset$, consider the generalized sphere $S^\HH_{iq} = \{ y \in \HH^{n} : \sscalar{\c^\HH_{iq},y}_1 - \k^\HH_{iq} = 0\}$, which has curvature $\k^\HH_{iq}$. Note that the cells $\Omega^\HH_i$ and $\Omega^\HH_q$ lie on different sides of $S^\HH_{iq}$. If $\sabs{\k^\HH_{iq}} \leq 1$ then $S_{iq}$ would divide $\HH^n$ into two halves having infinite volume, implying that $V(\HH^n \setminus \Omega^\HH_q) = \infty$, in contradiction to our assumption that $V(\Omega^\HH_i) < \infty$ for all $i=1,\ldots,q-1$. We deduce that for every $i \in \{1,\ldots,q-1\}$ such that $\Sigma^\HH_{iq} \neq \emptyset$, $\sabs{\k^\HH_{iq}} > 1$ and hence $S_{iq}$ is a geodesic sphere. Consequently, $\Omega^\HH_q$ is the intersection of a finite number of complements of geodesic balls, and hence $\HH^n \setminus \Omega^\HH_q$ is the union of a finite number of geodesic balls. This shows that $\Omega^\HH_i \subset \HH^n \setminus \Omega^\HH_q$ is bounded for all $i=1,\ldots,q-1$. 
\end{proof}

Define the standard stereographic projection $\pi_{\HH}: \HH^n \to \S_+^n \subset \R^{n+1}$ by
\begin{equation*}
\pi_{\HH}(y) =\frac{(\underline{y},1)}{y_0}, \quad y = (\underline{y}, y_{0}) \in \HH^n.
\end{equation*} 
Here $\S^n_{\pm} = \{ p \in \S^n : \pm \sscalar{p,N} > 0 \}$ denote the open Northern and Southern hemispheres of $\S^n$ (respectively), in its canonical embedding in $\R^{n+1} = \R^n \times \R$. Being a stereographic projection, $\pi_{\HH}$ maps generalized spheres in $\HH^n$ onto the intersection of a geodesic sphere in $\S^n$ with $\S^n_+$. The precise relation between the corresponding quasi-centers and curvatures is as follows. 

\begin{lemma}\label{lem:rel-cS-cH}
Let $\c^\M \in \R^n \times \R$ and $\k^\M \in \R$, $\M \in \{\S, \HH\}$. Assume that $(\c^\S,\k^\S)$ and $(\c^\HH,\k^\HH)$ are related by the following invertible transformation:
\begin{align}
\label{eq:cH-cS} \c^\HH & = (\underline \c^\S, -\k^\S), \quad \k^\HH  = -\c_0^\S ,\\
\label{eq:cS-cH} \c^\S     & = (\underline\c^\HH, -\k^\HH), \quad \k^\S  = -\c_0^\HH.
\end{align}
Then:
\begin{equation} \label{eq:SH-relation}
\sscalar{\c^\S, \pi_\HH(y)} + \k^\S = \frac{1}{y_{0}} (\sscalar{\c^\HH, y}_1 - \k^\HH) ~ ~ \forall y \in \HH^n ,
\end{equation}
and
\[
\sabs{\c^\HH}^2_1 - (1 - (\k^\HH)^2)  = |\c^\S|^2 - (1 + (\k^\S)^2)  . 
\]
In particular, if $S^\S = \{ p \in \S^n : \sscalar{\c^\S,p} + \k^\S = 0\}$ denotes a geodesic sphere in $\S^n$ and  $S^\HH = \{ y \in \HH^n : \sscalar{ \c^\HH, y}_1- \k^\HH = 0\}$ denotes a generalized sphere in $\HH^n$, then $S^\HH = \pi_\HH^{-1}(S^\S)$, and $(\c^\S,\k^\S)$ are the quasi-center and curvature parameters of $S^\S$ iff $(\c^\HH,\k^\HH)$ are the quasi-center and curvature parameters of $S^\HH$. 
\end{lemma}
\begin{proof}
For any $y = (\underline y, y_{0}) \in \HH^n$, we have
\[
\sscalar{\c^\S, \pi_\HH(y)} + \k^\S=\frac{1}{y_{0}} \brac{\sscalar{(\underline \c^\S, \c_0^\S), (\underline y,1)}+ \k^\S y_{0}}= \frac{1}{y_{0}} \brac{\sscalar{(\underline \c^\S, -\k^\S), y }_1+\c_0^\S},
\]
establishing (\ref{eq:SH-relation}) and that $y \in S^\HH$ iff $\pi_\HH(y) \in S^\S$. In addition:
\[
\sabs{\c^\HH}^2_1 - (1 - (\k^\HH)^2) = \sabs{(\underline \c^\S, -\k^\S)}_1^2 +(\c_0^\S)^2-1=\sabs{\underline \c^\S}^2-(\k^\S)^2+(\c_0^\S)^2-1= |\c^\S|^2 - (1 + (\k^\S)^2) . 
\]
Recalling (\ref{eq:ckS-relation}) and (\ref{eq:ckH-relation}), it follows that the latter expression is equal to zero iff $(\c^\M,\k^\M)$ are the quasi-center and curvature parameters of $S^\M$, $\M \in \{\S,\HH\}$, concluding the proof. 
\end{proof}

\begin{lemma}\label{lem:pull-back SphVor-hyper}
Let $\Omega^\S$ be a spherical Voronoi generalized $q$-partition in $\S^n$ with quasi-center parameters $\{\c_i^\S\}_{i=1,\ldots,q}$ and curvature parameters $\{\k_i^\S\}_{i=1,\ldots,q}$. Assume that $\Omega^\S$ is locally finite in $\S^n_+$. 
Let $I := \{ i \in \{1,\ldots,q\} : \Omega^\S_i \cap \S^n_+ \neq \emptyset \}$, and set $\Omega^{\S_+} := (\Omega_i^\S \cap \S_+^n)_{i \in I}$. 
Then $\Omega^\HH := \pi_{\HH}^{-1} (\Omega^\S) = \pi_{\HH}^{-1} (\Omega^{\S_+})$ is a spherical Voronoi $|I|$-partition in $\HH^n$ with quasi-center parameters $\{\c_i^\HH\}_{i\in I}$ and curvature parameters $\{\k_i^\HH\}_{i\in I}$, where $(\c_i^\HH, \k_i^\HH)$ are obtained from $(\c_i^\S, \k_i^\S)$ via (\ref{eq:cH-cS}). 
\end{lemma}
\begin{proof}
The assumption implies that $\Omega^{\S_+}$ 
is a locally-finite partition of $\S_+^n$, and hence $\Omega^\HH$ is indeed a locally-finite partition of $\HH^n$. Since $\Omega^\S_j \cap \S^n_+ = \emptyset$ for all $j \notin I$, we have for all $i \in I$:
\begin{align*}
\Omega_i^{\S_+} =  \Omega^\S_i \cap \S^n_+ & = \set{p \in \S_+^n : \argmin_{j=1, \ldots, q} (\sscalar{\c^\S_j, p} + \k_{j}^\S) =\{i\}} \\
& = \set{p \in \S_+^n : \argmin_{j \in I} (\sscalar{\c^\S_j, p} + \k_{j}^\S) =\{i\}} .
\end{align*}
Hence, by (\ref{eq:SH-relation}) and since $y_{0} > 0$:
\begin{align*}
\Omega^\HH_i = \pi_\HH^{-1}(\Omega_i^{\S_+}) & = \set{ y \in \HH^n : \argmin_{j \in I} (\sscalar{\c^\S_j, \pi_\HH(y)} + \k_{j}^\S) =\{i\}} \\
& = \set{ y \in \HH^n : \argmin_{j \in I} \frac{1}{y_{0}} (\sscalar{\c^\HH_j, y}_1 - \k_{j}^\HH) =\{i\}} \\
& = \set{ y \in \HH^n : \argmin_{j \in I} (\sscalar{\c^\HH_j, y}_1 - \k_{j}^\HH) =\{i\}} .
\end{align*}
If $\Sigma^\HH_{ij} \neq \emptyset$ for $i,j \in I$ then necessarily $\Sigma^\S_{ij} \neq \emptyset$, and hence $(\c^\S_{ij} , \k^\S_{ij})$ are the quasi-center and curvature parameters of $\Sigma^\S_{ij}$. It follows by Lemma \ref{lem:rel-cS-cH} that $(\c^\HH_{ij},\k^\HH_{ij})$ are the quasi-center and curvature parameters of $\Sigma^\HH_{ij}$, and we conclude that $\Omega^\HH$ is a spherical Voronoi $|I|$-partition in $\HH^n$, as asserted. 
\end{proof}

\begin{proposition}\label{prop:cluster hyper-clus spher}
Let $\Omega^{\HH}$ be a spherical Voronoi $(q-1)$-cluster in $\HH^n$. Then there exists a unique spherical Voronoi $(q-1)$-cluster $\Omega^\S$ in $\S^n$ such that $\Omega^\HH = \pi_{\HH}^{-1} (\Omega^\S)$. Moreover, $\Omega_{q}^\S \supset \overline{\S^n_-}$, and in particular $\Sigma_{ij}^\HH \neq \emptyset$ iff $\Sigma_{ij}^\S \neq \emptyset$. The quasi-center and curvature parameters  $\{\c_i^\M\}_{i=1,\ldots,q}$ and $\{\k_i^\M\}_{i=1,\ldots,q}$, respectively, of $\Omega^\M$, $\M \in \{\S,\HH\}$, are related via (\ref{eq:cH-cS}) and (\ref{eq:cS-cH}). 
\end{proposition}
\begin{proof}
Define $(\c^\S_i,\k^\S_i)$ from $(\c^\HH_i,\k^\HH_i)$ via (\ref{eq:cS-cH}), $i=1,\ldots,q$. Since all cells of $\Omega^\HH$ are non-empty by definition, necessarily $\{(\c^\HH_i,\k^\HH_i)\}$ are all distinct, and hence so are $\{(\c^\S_i,\k^\S_i)\}$. Now define
\[
\Omega^\S_i := \set{p \in \S^n : \argmin_{j=1, \ldots, q} (\sscalar{\c^\S_j, p} + \k_{j}^\S) =\{i\}} , \quad i=1,\ldots,q . 
\]
We claim that $\Omega^\S = (\Omega^\S_1,\ldots,\Omega^\S_q)$ is a cluster. Indeed, the cells are pairwise disjoint, and denoting:
\[
S^\S_{ij} := \set{p \in \S^n : \sscalar{\c^\S_{ij}, p} + \k_{ij}^\S =0 } ,
\]
clearly $\S^n \setminus \cup_{i=1}^q \Omega^\S_i \subset \cup_{i<j} S^\S_{ij}$. Since $\{(\c^\S_i , \k^\S_i)\}$ are all distinct, $\H^{n-1}(S^\S_{ij}) < \infty$ for all $i<j$, and it follows that $\Omega^\S$ is a genuine cluster. Furthermore, repeating the argument in the proof of Lemma \ref{lem:pull-back SphVor-hyper}, we see that $\pi_\HH^{-1}(\Omega^\S_i) = \Omega^\HH_i$ for all $i$. 

Since $\Omega^\HH_i$ are bounded for all $i=1,\ldots,q-1$ by Lemma \ref{lem:bounded}, the corresponding $\pi_{\HH}(\Omega^\HH_i) = \Omega^\S_i \cap \S^n_+$ are bounded away from the equator $\{ p \in \S^n : \sscalar{p,N} = 0 \}$. Consequently, there exists $\eps > 0$ so that, denoting $\S^n_{+,\eps} := \{ p \in \S^n : \sscalar{p,N} \in (0,\eps) \}$, we have:
\[
\forall p \in \S^n_{+,\eps} \;\; \forall i=1,\ldots,q-1 \;\; \sscalar{\c^\S_q, p} + \k^\S_q < \sscalar{\c^\S_i, p} + \k^\S_i  . 
\]
 Furthermore, since all $\pi_{\HH}(\Omega^\HH_i) = \Omega^\S_i \cap \S^n_+$ are non-empty, for each $i=1,\ldots,q-1$ there must be equality in the inequality above at some point $p_i \in \S^n_+$. We conclude that $S^\S_{iq} \cap \S^n_+ \neq \emptyset$ but $S^\S_{iq} \cap \S^n_{+,\eps} = \emptyset$ for all $i=1,\ldots,q-1$. But since $n \geq 2$, $S^\S_{iq}$ is connected, and hence necessarily $S^\S_{iq} \subset \S^n_+ \setminus \S^n_{+,\eps}$. Since $\Omega^\S_i$ and $\Omega^\S_q$ must lie on opposite sides of $S^\S_{iq}$, and since $\Omega^\S_q \supset \S^n_{+,\eps}$, it follows that $\Omega^\S_i  \subset \S^n_+ \setminus \S^n_{+,\eps}$ for all $i=1,\ldots,q-1$. Therefore $\Omega_{q}^\S \supset \overline{\S^n_-}$, and in particular all non-empty interfaces $\Sigma^\S_{ij}$ are contained in $\S^n_+$, and hence $\pi_{\HH}(\Sigma^\HH_{ij}) = \Sigma^\S_{ij}$. It follows that $\Sigma_{ij}^\HH \neq \emptyset$ iff $\Sigma_{ij}^\S \neq \emptyset$, and hence by Lemma \ref{lem:rel-cS-cH}, $\Omega^\S$ is a spherical Voronoi $(q-1)$-cluster with quasi-center and curvature parameters  $\{\c_i^\S\}_{i=1,\ldots,q}$ and $\{\k_i^\S\}_{i=1,\ldots,q}$, respectively. 
 
It remains to show the assertion regarding uniqueness. Let $\Omega^{\S,'}$ be any spherical Voronoi $(q-1)$-cluster in $\S^n$ such that $\pi_{\HH}^{-1}(\Omega^{\S,'}) = \Omega^\HH$. Repeating the argument from the previous paragraph, we deduce that $\Omega^{\S,'}_q \supset \overline{\S^n_-}$. But since $\Omega^{\S,'}_i \cap \S^n_+ = \Omega^{\S}_i \cap \S^n_+$ for all $i=1,\ldots,q$, it follows that $\Omega^{\S,'} = \Omega^\S$. This concludes the proof. 
\end{proof}

\begin{remark}
Proposition \ref{prop:cluster hyper-clus spher} is false for general partitions not induced by a cluster. The reason is that the constructed $\Omega^\S$ may posses additional non-empty interfaces $\Sigma^\S_{ij}$ in the Southern hemisphere $\S^n_-$, even though $\Sigma^\HH_{ij} = \emptyset$. Consequently, there is no guarantee that $\sabs{\c_{ij}^\S}^2 = 1 + (\k^\S_{ij})^2$ for the new interfaces $\Sigma^\S_{ij}$, and hence $\Omega^\S$ may not be spherical Voronoi. Concrete counterexamples are easy to construct. 
\end{remark}

\begin{definition} A spherical Voronoi partition $\Omega^\HH$ in $\HH^n$ is called M\"obius-flat if the collection $\{\pi_\HH (\Sigma_{ij}^\HH) : \Sigma_{ij}^\HH \neq \emptyset \}$ of stereographic projections of its non-empty interfaces is  M\"obius-flat. \end{definition}

The following characterization follows immediately from Lemmas \ref{lem:Mobius-flat-xi-sphere} and \ref{lem:rel-cS-cH}.
\begin{lemma}\label{lem:crit-MF-hyper}
A spherical Voronoi partition $\Omega^\HH$ in $\HH^n$ is M\"obius-flat if and only if
\begin{equation} \label{eq:xi-hyper}
\exists \xi= (\underline \xi, \xi_0) \in \R^n \times \R ~~ \sabs{\underline \xi}^2+ \xi_0^2<1 \quad \text{so that} \quad \sscalar{\c_{ij}^\HH, (\underline \xi, 1)}_1- \k_{ij}^\HH \xi_0=0 ~ ~ \forall \Sigma_{ij}^\HH \neq \emptyset ,
\end{equation}
where $\{\c^\HH_i\}$ and $\{\k^\HH_i\}$ are its quasi-center and curvature parameters, respectively. Equivalently, denoting $\{\c^\S_i\}$ and $\{\k^\S_i\}$ via (\ref{eq:cS-cH}), $\Omega^\HH$ is M\"obius-flat if and only if
\begin{equation} \label{eq:xi-hyper2}
\exists \xi \in \R^{n+1} ~~ \sabs{\xi}<1 \quad \text{so that} \quad \sscalar{\c_{ij}^\S, \xi} + \k_{ij}^\S =0 ~ ~ \forall \Sigma_{ij}^\HH \neq \emptyset .
\end{equation}
\end{lemma}

\begin{proposition}\label{prop:VHn}
Let $\Omega^\HH$ be a regular, M\"obius-flat, spherical Voronoi partition in $\HH^n$. 
Then there exists a conformally boundary flattening potential $V>0$ on $\HH^n$ such that
\begin{equation*}
L_{Jac} V =  (n-1) \xi_0 ~ , ~ V \hRic^V_{\Sigma^1}= (n-2) \xi_0 g_{\Sigma}~\text{on}~\Sigma^1,  
\end{equation*}
where $\xi = (\underline \xi,\xi_0)$ is given by Lemma \ref{lem:crit-MF-hyper}. 
Specifically, $V$ can be chosen as
\begin{equation}\label{V-hyperbolic}
V(y) = -\xi_0 - \sscalar{y, (\underline \xi, 1)}_1 , \quad y \in \HH^n.
\end{equation}
\end{proposition}

\begin{proof}
Define function $V$ on $\HH^n$ as in (\ref{V-hyperbolic}). First, we prove the positivity of $V$. 
Using the Cauchy-Schwarz inequality, $|\underline \xi|^2+\xi_0^2<1$ and $y =(\underline y, y_{0}) \in \HH^n$, we have $
\brac{|\underline y||\underline \xi| + |\xi_0|}^2 < |\underline y|^2+1= y_{0}^2$. Then
\begin{equation*}
V(y) \geq  -|\xi_0|-|\underline y||\underline \xi|+y_{0}>0, \quad \forall y \in \HH^n.
\end{equation*}
Using (\ref{eq:xi-hyper}), we have on each non-empty interface $\Sigma_{ij}^\HH$,
\begin{equation}\label{eq:BCv-hyper}
\nabla_{\n_{ij}^\HH} V= -\scalar{(\underline \xi, 1), \c_{ij}^\HH+\k_{ij}^\HH y}_1= -\k_{ij}^\HH \xi_0- \k_{ij}^\HH \sscalar{(\underline \xi, 1), y}_1 = \k_{ij}^\HH V.
\end{equation}
By Lemma \ref{lem:HnStationary}, $\Omega^\HH$ is stationary, and so as in the proof of Proposition \ref{prop:VSn}, $\nabla_{\n_{\partial ij}^\HH} V =  \bar{\II}^{\partial ij} V$ on $\partial \Sigma_{ij}^\HH$. Therefore $V$ satisfies the non-oriented conformal BCs, and since all interfaces are generalized spheres, Lemma \ref{lem:constant-curvature} implies that $V$ is a conformally flattening boundary potential.

Differentiating $V$ twice on $\Sigma_{ij}^\HH$ and using (\ref{eq:BCv-hyper}), we have at $y \in \Sigma_{ij}^\HH$,
\begin{align*}
\nabla_{\Sigma^1}^2 V=& \; \nabla_{\R^{n,1}}^2 V -\sscalar{(\underline \xi,1), y - \k_{ij}^\HH \n_{ij}^\HH }_1 g_{\Sigma^1}\\
=& \; -\sscalar{(\underline \xi, 1), y}_1  g_{\Sigma^1}  - (\k_{ij}^\HH)^2 V g_{\Sigma^1}\\
=& \; -\brac{-\xi_0+ ((\k_{ij}^\HH )^2-1) V} g_{\Sigma^1}.
\end{align*}
Then $\Delta_{\Sigma^1} V = -(n-1) (-\xi_0+ ((\k_{ij}^\HH )^2-1) V)$. Since $\Ric_{\HH^n} (\n_{ij}^\HH, \n_{ij}^\HH) + \norm{\II^{ij}}^2 = - (n-1) + (n-1)(\k_{ij}^\HH)^2$ on each non-empty interface $\Sigma_{ij}^\HH$, we have
\begin{equation*}
L_{Jac} V =\Delta_{\Sigma^1} V + \brac{- (n-1) + (n-1)(\k_{ij}^\HH)^2} V =  (n-1) \xi_0.
\end{equation*} 
On the other hand, the Gauss equation implies $\Ric_{\Sigma^1} = (n-2) ((\k_{ij }^\HH)^2-1) g_{\Sigma^1}$. Thus,
\begin{align*}
V \hRic_{\Sigma^1}^V =& \; V \Ric_{\Sigma^1} - \nabla_{\Sigma^1}^2 V + \Delta_{\Sigma^1} V g_{\Sigma^1}\\
=& \; \left (  (n-2) ( (\k_{ij }^\HH)^2-1 ) V +(-\xi_0+ ((\k_{ij}^\HH )^2-1) V ) \right . \\
& \;\;\;  \left . - (n-1) (-\xi_0+ ((\k_{ij}^\HH )^2-1) V ) \right ) g_{\Sigma^1}\\
=& \; (n-2) \xi_0 g_{\Sigma^1}.
\end{align*}
This completes the proof.
\end{proof}
In light of Propositions \ref{prop:intro-LJacNegative} and \ref{prop:VHn}, the sign of $\xi_0$ determines whether volume constraints are necessary for establishing the stability of partitions. Hence, we introduce the following definition.
\begin{definition}[Hypo/Epi-M\"obius-flat]
A  M\"obius-flat spherical Voronoi partition $\Omega^\HH$ in $\HH^n$ is called hypo-M\"obius-flat if there exists $\xi = (\underline \xi, \xi_0) \in \R^{n+1}$ satisfying condition (\ref{eq:xi-hyper}) such that $\xi_0 \leq 0$. Otherwise, $\Omega^\HH$ is called epi-M\"obius-flat.
\end{definition}

By Lemma \ref{lem:pull-back SphVor-hyper}, there exist numerous hypo-M\"obius-flat, spherical Voronoi partitions in $\HH^n$ --- these are obtained by pulling back via $\pi_\HH$ M\"obius-flat spherical Voronoi partitions in $\S^n$, such that $\xi = (\underline \xi, \xi_0)$ from (\ref{eq:xiS}) satisfies $\xi_0 \leq 0$. However, we conclude this section by showing that \emph{clusters} in $\HH^n$ can only be epi-M\"obius flat.

\begin{proposition}
Any M\"obius-flat, spherical Voronoi cluster in $\HH^n$ is necessarily epi-M\"obius-flat.    
\end{proposition}
\begin{proof}
Let $\Omega^\HH$ be a  M\"obius-flat, spherical Voronoi, $(q-1)$-cluster in $\HH^n$, $q \geq 2$. By Lemma \ref{lem:bounded}, it is necessarily bounded. Let $\xi \in \R^n \times \R$ be any vector given by Lemma \ref{lem:crit-MF-hyper}; in particular $\sabs{\xi}<1$. 
By Proposition \ref{prop:cluster hyper-clus spher}, there exists a M\"obius-flat, spherical Voronoi, $(q-1)$-cluster $\Omega^\S$ in $\S^n$ with quasi-center and curvature parameters $\{\c^\S_i\}$ and $\{ \k^\S_i \}$, respectively,  such that $\Omega^\HH = \pi_{\HH}^{-1} (\Omega^\S)$, $\Omega_{q}^\S \supset \overline{\S^n_-}$, and $\Sigma^\HH_{ij} \neq \emptyset$ iff $\Sigma^\S_{ij} \neq \emptyset$. Consequently, by (\ref{eq:xi-hyper2}), we have $\sscalar{\c_{ij}^\S, \xi} +\k_{ij}^\S=0$ for all  $\Sigma_{ij}^\S \neq \emptyset$. Since by definition $\Omega^\HH$ and hence $\Omega^\S$ have no empty cells, and since $q \geq 2$, $\partial^* \Omega^\S_q \neq \emptyset$ and hence 
$\exists q' \in \{1, \ldots, q-1\}$ such that $\Sigma_{q q'}^\S \neq \emptyset$. Therefore,
\begin{equation*}
\overline{\S_-^n} \subset \Omega_q^\S \subset \set{p \in \S^n :  \sscalar{\c_{qq'}^\S, p}+\k_{qq'}^\S <0} ,
\end{equation*}
and we see that the vector $\xi = (\underline \xi, \xi_0)$ in the unit ball of $\R^{n+1}$ lies on the boundary of a halfspace in $\R^{n+1}$ which contains $\overline{\S_-^n}$. This implies $\xi_0>0$, and hence $\Omega^\HH$ is epi-M\"obius-flat. 
\end{proof}

\bibliographystyle{plain}
\bibliography{../../../ConvexBib}

\end{document}